\documentclass[10pt]{article}
\usepackage[utf8]{inputenc} 
\usepackage[T1]{fontenc}
\usepackage{lmodern}

\usepackage{amsthm,amsmath,amsfonts,amssymb,stmaryrd,bbm,dsfont}
\usepackage{mathtools}
\usepackage{enumitem}
\usepackage{mathrsfs} 
\usepackage{comment}

\usepackage{subcaption} 

\usepackage{graphicx}
\usepackage[dvipsnames]{xcolor}
\usepackage{caption,tikz}

\usepackage[english]{babel}
\usepackage[colorlinks=true, linkcolor=blue, urlcolor=black, citecolor=blue,pdfstartview=FitH]{hyperref}

\usepackage{bm} 
\usepackage{mathtools} 

\numberwithin{equation}{section}

\let\originalleft\left
\let\originalright\right
\renewcommand{\left}{\mathopen{}\mathclose\bgroup\originalleft}
\renewcommand{\right}{\aftergroup\egroup\originalright}

\newlength{\bibitemsep}
\newlength{\bibparskip}
\let\oldthebibliography\thebibliography
\renewcommand\thebibliography[1]{\oldthebibliography{#1}
  \setlength{\parskip}{\bibitemsep}
  \setlength{\itemsep}{\bibparskip}}

\DeclareMathOperator{\im}{Im}
\DeclareMathOperator{\Tr}{Tr}

\DeclareMathOperator{\OO}{O}
\DeclareMathOperator{\oo}{o}

\DeclareMathOperator{\Id}{Id}
\DeclareMathOperator{\diag}{diag}

\DeclareMathOperator{\rank}{rank}

\newcommand{\tr}{\mathrm{tr}}

\newcommand{\mc}[1]{\mathcal{#1}}
\newcommand{\mf}[1]{\mathfrak{#1}}

\newcommand{\ii}{\mathrm{i}}

\newcommand{\defeq}{\vcentcolon=}
\newcommand{\eqdef}{=\vcentcolon}

\renewcommand{\epsilon}{\varepsilon}

\renewcommand{\leq}{\leqslant}
\renewcommand{\geq}{\geqslant}
\renewcommand{\le}{\leq}
\renewcommand{\ge}{\geq}

\renewcommand{\P}{\mathbb{P}}
\newcommand{\E}{\mathbb{E}}
\newcommand{\R}{\mathbb{R}}
\newcommand{\C}{\mathbb{C}}
\newcommand{\N}{\mathbb{N}}
\newcommand{\Z}{\mathbb{Z}}

\newcommand{\abs}[1]{\left\lvert #1 \right\rvert}
\newcommand{\norm}[1]{\lVert #1 \rVert}

\newcommand{\vertiii}[1]{{\left\vert\kern-0.25ex\left\vert\kern-0.25ex\left\vert #1 
    \right\vert\kern-0.25ex\right\vert\kern-0.25ex\right\vert}}
\newcommand{\ip}[1]{\left\langle #1 \right\rangle}
\newcommand{\diff}{\mathop{}\!\mathrm{d}}

\theoremstyle{plain} 
\newtheorem{theorem}{Theorem}[section]

\newtheorem{lemma}[theorem]{Lemma}

\newtheorem{prop}[theorem]{Proposition}
\newtheorem{assn}{Assumption}

\newtheorem{defn}[theorem]{Definition}

\newtheorem{rem}[theorem]{Remark}

\definecolor{color1}{RGB}{119,170,221} 
\definecolor{color2}{RGB}{153,221,255} 
\definecolor{color3}{RGB}{68, 187, 153} 
\definecolor{color4}{RGB}{187,204,51} 
\definecolor{color5}{RGB}{170,170,0} 
\definecolor{color6}{RGB}{238,221,136} 
\definecolor{color7}{RGB}{238,136,102} 
\definecolor{color8}{RGB}{255,170,187} 
\definecolor{color9}{RGB}{221,221,221} 
\definecolor{color10}{RGB}{136,34,85} 

\makeatletter
\renewcommand{\section}{\@startsection
{section}
{1}
{0mm}
{-2\baselineskip}
{1\baselineskip}
{\normalfont\large\scshape\centering}} 
\makeatother

\makeatletter
\renewcommand{\subsection}{\@startsection
{subsection}
{2}
{0mm}
{-\baselineskip}
{0 \baselineskip}
{\normalfont\bf\itshape}} 
\makeatother

\makeatletter
\renewcommand{\subsubsection}{\@startsection
{subsubsection}
{3}
{0mm}
{-\baselineskip}
{0 \baselineskip}
{\normalfont\bf\itshape}} 
\makeatother

\def\author#1{\par
    {\centering{\authorfont#1}\par\vspace*{0.05in}}
}

\def\titlefont{\fontsize{13}{15}\bfseries\boldmath\selectfont\centering{}}
\def\authorfont{\fontsize{13}{15}}

\let\affiliationfont\rhfont

\def\address#1{\par
    {\centering{\affiliationfont#1\par}}\par\vspace*{11pt}
}

\def\title#1{
    \thispagestyle{plain}
    \vspace*{-14pt}
    \vskip 79pt
    {\centering{\titlefont #1\par}}%
    \vskip 1em
}

\begin{document}

~\vspace{-1.4cm}

\title{Universality for the global spectrum of random inner-product\\ kernel matrices in the polynomial regime}

\vspace{1cm}
\noindent
\begin{minipage}[c]{0.50\textwidth}
\author{Sofiia Dubova}%
\address{Harvard University \\
    Department of Mathematics \\
   E-mail: sdubova@math.harvard.edu}%
 \end{minipage}%
\begin{minipage}[c]{0.50\textwidth}
\author{Yue M. Lu}%
\address{Harvard University \\
    School of Engineering and Applied Sciences \\
   E-mail: yuelu@seas.harvard.edu}%
 \end{minipage}%

\vspace*{1cm}

\noindent
\begin{minipage}[c]{0.50\textwidth}
\author{Benjamin McKenna}%
\address{Harvard University \\
    Center of Mathematical Sciences and Applications \\
   E-mail: bmckenna@fas.harvard.edu}%
\end{minipage}%
\begin{minipage}[c]{0.50\textwidth}
\author{Horng-Tzer Yau}%
\address{Harvard University \\
    Department of Mathematics \\
   E-mail: htyau@math.harvard.edu}%
\end{minipage}%

\vspace*{1cm}

\begin{abstract}
We consider certain large random matrices, called \emph{random inner-product kernel matrices}, which are essentially given by a nonlinear function $f$ applied \emph{entrywise} to a sample-covariance matrix, $f(X^TX)$, where $X \in \R^{d \times N}$ is random and normalized in such a way that $f$ typically has order-one arguments. We work in the \emph{polynomial regime}, where $N \asymp d^\ell$ for some $\ell > 0$, not just the \emph{linear regime} where $\ell = 1$. Earlier work by various authors showed that, when the columns of $X$ are either uniform on the sphere or standard Gaussian vectors, and when $\ell$ is an integer (the linear regime $\ell = 1$ is particularly well-studied), the bulk eigenvalues of such matrices behave in a simple way: They are asymptotically given by the free convolution of the semicircular and Mar\v{c}enko--Pastur distributions, with relative weights given by expanding $f$ in the Hermite basis. In this paper, we show that this phenomenon is universal, holding as soon as $X$ has i.i.d. entries with all finite moments. In the case of non-integer $\ell$, the Mar\v{c}enko--Pastur term disappears (its weight in the free convolution vanishes), and the spectrum is just semicircular.
\end{abstract}

\vspace*{0.05in}

\noindent \emph{Date:} October 27, 2023

\vspace*{0.05in}

\noindent \hangindent=0.2in \emph{Keywords and phrases:} random inner-product kernel matrices, nonlinear random matrices, free convolution, orthogonal polynomials, polynomial regime

\vspace*{0.05in}


{
	\hypersetup{linkcolor=black}
	\tableofcontents
}


\section{Introduction}


\subsection{Our results}\

In this paper, we give a common \emph{global law} for the spectra of two related families of real-symmetric random matrices which are in some sense \emph{nonlinear}. Our matrices $A = A_N \in \R^{N \times N}$
and $\widetilde{A} = \widetilde{A_N} \in \R^{N \times N}$, called \emph{random inner-product kernel matrices}, have the entrywise form
\begin{align}
    A_{ij} &= \begin{cases} \frac{1}{\sqrt{N}} f\left( \frac{\ip{X_i,X_j}}{\sqrt{d}} \right) & \text{if } i \neq j, \\ 0 & \text{if } i = j, \end{cases} \label{eqn:intro_model_def} \\
    \widetilde{A}_{ij} &= \begin{cases} \frac{1}{\sqrt{N}} f\left( \frac{\sqrt{d}\ip{X_i,X_j}}{\|X_i\|\|X_j\|}\right) = \frac{1}{\sqrt{N}} f\left( \frac{\ip{X_i,X_j}}{\sqrt{d}}\frac{\sqrt{d}}{\|X_i\|} \frac{\sqrt{d}}{\|X_j\|} \right) & \text{if } i \neq j \text{ and } \|X_i\| \neq 0 \neq \|X_j\|, \\ 0 & \text{otherwise}, \end{cases} \label{eqn:intro_model_def_normalized}
\end{align}
and ``give a global law'' means that we find a deterministic measure $\rho$ that is the almost-sure weak limit of the empirical spectral measures 
\[
    \rho_N = \frac{1}{N} \sum_{i=1}^N \delta_{\lambda_i(A)}, \qquad \widetilde{\rho}_N = \frac{1}{N} \sum_{i=1}^N \delta_{\lambda_i(\widetilde{A})},
\]
where $(\lambda_i(A))_{i=1}^N$ (resp., $(\lambda_i(\widetilde{A}))_{i=1}^N$) are the eigenvalues of $A$ (resp., of $\widetilde{A}$). Here $f : \R \to \R$ is some fixed function, such as ReLU, representing the nonlinearity; $d$ and $N$ are parameters tending simultaneously to infinity in the so-called \emph{polynomial regime} where $d^\ell \asymp N$ for some $\ell > 0$; and the i.i.d. vectors $(X_i)_{i=1}^N$ have i.i.d. components drawn from some fixed $\mu$, which is a centered probability measure on $\R$ with unit variance.

Informally speaking, this normalization implies that $\ip{X_i,X_j}/\sqrt{d}$ is order-one, and $\sqrt{d}/\|X_i\| \approx 1 + \oo(1)$, so that $A$ and $\widetilde{A}$ are entrywise quite close. Indeed, we show that they have the same global law (i.e., $\rho_N$ and $\widetilde{\rho}_N$ tend to the same $\rho$). We include them both since, while $A$ may seem more natural at first, $\widetilde{A}$ has some better theoretical properties, namely it appears to have fewer outliers than $A$ (although we do not prove this in the current work, to keep this paper at a manageable length).

Our main result extends that of \cite{LuYau}, which studies this model when $\mu$ is Gaussian measure, so that $A$ deals with Gaussian vectors and $\widetilde{A}$ deals with vectors which are uniform on the sphere. In the current work, we show that their result is in fact \emph{universal} in $\mu$, holding as soon as $\mu$ has all finite moments.

If the function $f$ happens to be linear, then $A$ is just a sample covariance matrix with the diagonal set to zero. (Zeroing the diagonal keeps the spectrum of $A$ from translating off to infinity; see Remark \ref{rem:diagonal}.) For general $f$, then, the matrix $A$ is an \emph{entrywise} nonlinear function of a sample covariance matrix, scaled so that the spectrum is order one, and so that the nonlinearity $f$ typically has order-one arguments. Furthermore, since $f$ is applied entrywise, any expansion $f(x) = \sum_k c_k h_k(x)$ induces a corresponding expansion 
\begin{equation}\label{eqn:A_Ak}
A = \sum_k c_k A_k.
\end{equation}
The fundamental observation of Cheng and Singer \cite{CheSin2013} is that one should take $(h_k)_{k=1}^\infty$ to be an appropriate sequence of orthogonal polynomials, usually the Hermite polynomials. In this language, the main result of, say, \cite{LuYau} is essentially a rigorous version of the following heuristics (more precisely, the version just for integer $\ell$): Once placed in this basis,
\begin{enumerate}
\item the matrices $A_k$ in \eqref{eqn:A_Ak} are approximately independent;
\item the low-degree matrices $(A_k)_{k=0}^{\lceil \ell \rceil - 1}$ are essentially low-rank, so do not affect the global law;
\item if $\ell$ is an integer, the matrix $A_\ell$ is essentially a sample-covariance matrix, so its global law is given by the Mar\v{c}enko-Pastur distribution; but if $\ell$ is not an integer then there is no matrix $A_\ell$;
\item each of the high-degree matrices $(A_k)_{k=\ell_c}^\infty$, where $\ell_c$ is the least integer \emph{strictly} bigger than $\ell$, has an asymptotically semicircular distribution, because they are essentially degenerate sample-covariance matrices, in the parameter limit in which the Mar\v{c}enko-Pastur distribution degenerates to the semicircle law.
\end{enumerate}
As a consequence, the limiting measure $\rho$ is the \emph{free (additive) convolution} of the Mar\v{c}enko-Pastur and semicircular distributions, with weights given by the coefficients of $f$ in its Hermite expansion, when $\ell$ is an integer; and just a semicircular distribution without a Mar\v{c}enko-Pastur part, still with Hermite weights, when $\ell$ is not an integer.

In the work \cite{LuYau}, the authors intially consider vectors $X_i$ which are uniform on the sphere, for which certain classical algebraic identities simplify the problem. Roughly speaking, the model is linearized by spherical harmonics, making it easier to see the structure of the four-step heuristic above. Then they extend from spherical to Gaussian vectors by comparison. 

For general $\mu$, these algebraic identities are not available. Instead, our proof relies essentially on a ``pre-processing'' step, identifying by hand which parts of $A$ (resp. $\widetilde{A}$) are errors, replacing $A$ (resp. $\widetilde{A}$) by an error-free matrix $B$, then showing that the Stieltjes transform of $B$ approximately satisfies a self-consistent equation. The exact solution to this self-consistent equation describes the free convolution mentioned above, so we can conclude with perturbation theory.


\subsection{Related work}\

An extensive history of this problem was given by the recent paper \cite{LuYau}, so we only give a brief overview. Kernel matrices were first introduced in the classical scaling ($d$ fixed and $N \to \infty$) by Koltchinskii and Gin\'e \cite{KolGin2000}. El Karoui \cite{ElK2010} studied random kernel matrices in the high-dimensional linear scaling $\ell = 1$ (i.e., $d \asymp N$), but in a different normalization where the arguments of the nonlinearity $f$ are typically $\oo(1)$, so that the only surviving feature of $f$ is its behavior at zero. Our scaling (where $f$ typically has order-one arguments) was first studied by Cheng and Singer \cite{CheSin2013}, still in the linear regime $\ell = 1$, and when $\mu$ is Gaussian. As previously mentioned, Cheng and Singer introduced the idea of writing the nonlinearity $f$ in a good basis of orthogonal polynomials, which is fundamental in later works, including ours. Later, by comparing to \cite{CheSin2013} with the Lindeberg exchange method, Do and Vu \cite{DoVu2013} were able to allow for non-Gaussian data in the linear regime. More recently, Fan and Montanari \cite{FanMon2019} found sufficient conditions on $f$ so that, in the linear Gaussian model, the top eigenvalue sticks to the edge of the limiting distribution. As previously mentioned, two of the present authors \cite{LuYau} considered the (integer) polynomial case $\ell = 1, 2, 3 \ldots$ for Gaussian and spherical data; simultaneous and independent work by Misiakiewicz \cite{Mis2022} considered spherical or Bernoulli data for special nonlinearities $f$. With the exception of spherical data, the previous works all focus on the unnormalized model \eqref{eqn:intro_model_def}; we are not aware of previous results on the normalized model \eqref{eqn:intro_model_def_normalized} beyond the spherical case.  

Our proof shows that the matrices $A$ and $\widetilde{A}$ are well-approximated by a \emph{generalized sample covariance matrix} $B$, which we write (up to subtracting the diagonal) as $U^\ast TU$. In dealing with the $B$ matrix, one technical complication is that the entries of $T$ have different scales from one another; another is that the entries of $U$ are uncorrelated but not independent; and a third is that, because of our polynomial scaling, $U \in \R^{M \times N}$ but only with the weaker $\log M \asymp \log N$ rather than $M \asymp N$ (roughly, $M \approx d^L$, where $L$ is the degree of the largest nonzero Hermite coefficient of $f$). Sample covariance matrices with various combinations of these technical difficulties have previously been studied in \cite{BloErdKnoYauYin2014} and \cite{BloKnoYauYin2014}. We also use several ideas from \cite{KnoYin2017}: For example we embed our matrices of interest in a larger $2 \times 2$ block matrix, inspired by \cite[(3.2)]{KnoYin2017}, and ideas like their Lemma 4.6 appear here under the name of ``partial Ward inequalities'' in Section \ref{sec:ward}.

We also mention several related models: Instead of taking a nonlinearity of a sample covariance matrix (informally $f(X^TX)$), one can take a sample covariance of a nonlinearity (informally $f(X)^Tf(X)$). This changes the limiting spectrum; for details we direct readers to works of Pennington and Worah \cite{PenWor2019}, Benigni and P\'ech\'e \cite{BenPec2021, BenPec2022}, and Piccolo and Schr\"{o}der \cite{PicSch2021}. The \emph{non-Hermitian} (possibly rectangular) version $A_{ij} = \frac{\delta_{i \neq j}}{\sqrt{N}} f\left(\frac{\ip{X_i,Y_j}}{\sqrt{d}}\right)$ is closely related to the random-feature model (see, e.g., \cite{RahRec2007, LouLiaCou2018, HasMonRosTib2022, PenWor2019} for this model in general, and \cite{MeiMon2022, GerLouKrzMezZde2021, GolLouReeKrzMezZde2021, LouGerCuiGolKrzMezZde2022, HuLu2023} for corresponding random-matrix results, all in the linear regime $d \asymp N$).


\subsection{Organization}\

The structure of the paper is as follows: Our main results, Theorem \ref{thm:main_polynomial} (for polynomial nonlinearities $f$) and Theorem \ref{thm:main_general} (for general nonlinearities $f$), are given in Section \ref{sec:main}. One main step in the proof is to replace the matrices $A$ and $\widetilde{A}$ with a simpler matrix $B$, thus separating the ``main-term analysis'' of the matrix $B$, which is the bulk of the paper and constitutes Sections \ref{sec:main_term_sketch} through \ref{sec:tilde_m_self_consistent}, and the ``error analysis'' of the matrices $A-B$ and $\widetilde{A} - B$, which is Section \ref{sec:error-terms}. In Section \ref{sec:main_term_sketch} we give an overview of the main-term analysis, introducing several fundamental resolvent and resolvent-like quantities and stating that they approximately solve various self-consistent equations. In order to prove these claims, we first spend three sections establishing basic tools for our analysis: Various resolvent identities, in Section \ref{sec:resolvent_identities}; variants of the Ward identity that we call ``full Ward inequalities'' and ``partial Ward inequalities,'' in Section \ref{sec:ward}; and a collection of preliminary bounds, in Section \ref{sec:preliminary_bounds}. We then use these tools to prove these approximate self-consistent equations in Sections \ref{sec:concentration} and \ref{sec:tilde_m_self_consistent}. These sections complete the proof when the nonlinearity $f$ is a polynomial; in Appendix \ref{appx:general_nonlinearities} we explain how to prove the general case, by approximating general nonlinearities by polynomials.


\subsection{Notation}\

\emph{Stochastic domination:} We will use the following notation of high-probability boundedness up to small polynomial factors, introduced in \cite{ErdKnoYauYin2013local}. If $X = X_N(u)$ and $Y = Y_N(u)$ are two families of real random variables, indexed by $N$ and by $u$ in some set $U$, we write 
\[
    X \prec Y
\]
when, for every $\epsilon , D > 0$, there exist $C_{\epsilon,D}$ and $N_0(\epsilon,D)$ such that
\[
    \sup_{u \in U} \P(X_N(u) \geq N^\epsilon Y_N(u)) \leq C_{\epsilon,D} N^{-D} \quad \text{for all } N \geq N_0(\epsilon,D).
\]
If $X$ is complex, then $X \prec Y$ is defined as $\abs{X} \prec Y$. It will be convenient to write $\OO_\prec$, where for example $X \leq Y + \OO_\prec(Z)$ is defined as $X - Y \prec Z$.

We remark that our problem has two parameters tending to infinity simultaneously, namely $N$ and $d$. The definition given here is in terms of $N$, but one could equally write a definition in terms of $d$ (with $d^\epsilon$, $d \geq d_0$, and so on), and it is easy to check that this definition would produce the same result. We will sometimes switch between the two for convenience.

\emph{Stieltjes transforms:} If $T \in \R^{N \times N}$ is symmetric, we write its Stieltjes transform with the sign convention
\[
    s_T(z) = \frac{1}{N} \tr((T-z\Id)^{-1}).
\]

\emph{Summation conventions:} We frequently consider sums over multiple indices, but include only the terms where these indices are all distinct. We indicate this with an asterisk on top of the summation notation. For example, $\sum_{a,b}^{N,\ast} f_{a,b}$ is defined as $\sum_{a,b=1 : a \neq b}^N f_{a,b}$; the notation $\sum_{a,b,c}^{N,\ast} f_{a,b,c}$ means that $a,b,c$ should \emph{all} be distinct (i.e., $a=b\neq c$ is also excluded from the sum), and so on. Additionally, we use standard exclusion notation like $\sum_\nu^{(\mu)}$ to indicate, in this case, the sum over all $\nu$ except for $\nu = \mu$. 

\emph{Floors and ceilings:} If $\ell > 0$, then $\lceil \ell \rceil$ is the smallest integer \emph{at least} $\ell$ as usual, but we will also need 
\[
    \ell_c \defeq \begin{cases} \ell + 1 & \text{if } \ell \in \N, \\ \lceil \ell \rceil & \text{otherwise}, \end{cases}
\]
for the smallest integer \emph{strictly bigger than} $\ell$, and $\{\ell\}$ for the fractional part, i.e. $\{2.4\} = 0.4$ and $\{2\} = 0$. 

\emph{Other notation:} We write $\mathbb{H}$ for the complex upper half-plane $\mathbb{H} = \{z \in \C : \im(z) > 0\}$, and write $\llbracket a, b \rrbracket$ for the consecutive integers $[a,b] \cap \Z = \{a,a+1,\ldots,b\}$. We sometimes abuse notation by dropping ``$\Id$'' for constant matrices; for example, if $A$ is a matrix and $z$ is a constant, then we write $A - z$ for $A - z\Id$. 

\subsection{Acknowledgements}\

The work of H.-T. Y. is partially supported by the NSF grant DMS-2153335, and by a Simons Investigator award. B. M. is partially supported by NSF grant DMS-1760471. The work of Y. M. L. is partially supported by NSF grant CCF-1910410, and by the Harvard FAS Dean's Fund for Promising Scholarship.


\section{Main results}
\label{sec:main}

\subsection{Polynomial nonlinearities}\

In this first set of results, we take the nonlinearity $f : \R \to \R$ to be a \emph{fixed} polynomial (not depending on $d$ or $N$), expressible in the Hermite basis as 
\begin{equation}
\label{eqn:def_f_poly_form}
    f(x) = \sum_{k=0}^L c_k h_k(x)
\end{equation}
for some $L$ and some constants $(c_k)_{k=0}^L$, where the $h_k$ are the normalized (non-monic) Hermite polynomials given as
\[
    \E_{Z \sim \mc{N}(0,1)}[h_i(Z)h_j(Z)] = \delta_{ij},
\]
with the first several given by
\[
    h_0(x) = 1, \quad h_1(x) = x, \quad h_2(x) = \frac{1}{\sqrt{2}}(x^2-1), \quad h_3(x) = \frac{1}{\sqrt{6}}(x^3-3x).
\]

We will always work in the rectangular domain of the complex upper half plane defined by
\[
    \mathbf{D}_\tau \defeq \{z = E + \ii \eta : \tau \leq \eta \leq \tau^{-1}, \abs{E} \leq \tau^{-1}\}
\]
for arbitrary $\tau > 0$. Given $\ell > 0$, we will also need the parameter
\begin{equation}
\label{eqn:def_p_ell}
    p_\ell = \begin{cases} 1/2 & \text{if } \ell \in \N, \\ \min(\{\ell\},1-\{\ell\})/2 & \text{otherwise}, \end{cases}
\end{equation}
where we recall that $\{\ell\}$ is the fractional part of $\ell$. We remark that $p_\ell \in [0,1/4]$, unless $\ell$ is an integer, in which case $p_\ell = 1/2$.

\begin{theorem}
\label{thm:main_polynomial}
\textbf{(Main theorem, polynomial nonlinearities)}
Suppose that $\mu$ has all moments finite. Fix $\kappa, \tau, \ell > 0$, a positive integer $L$, and a polynomial $f$ of the form \eqref{eqn:def_f_poly_form}. Suppose that 
\begin{equation}
\label{eqn:assn:N/d^ell->kappa_weakly}
    \abs{\frac{N}{d^\ell} - \kappa} = \begin{cases} \oo(1) & \text{if $\ell$ is not an integer,} \\ \OO(d^{-\epsilon_0}) & \text{for some fixed $\epsilon_0 > 0$, if $\ell$ is an integer.} \end{cases}
\end{equation}
Then, for each fixed $z \in \mathbf{D}_\tau$, we have
\begin{align*}
    s_A(z) &\overset{d \to \infty}{\to} \mf{m}(z) \quad \text{almost surely}, \\
    s_{\widetilde{A}}(z) &\overset{d \to \infty}{\to} \mf{m}(z) \quad \text{almost surely},
\end{align*}
with the effective bounds
\begin{align*}
    \abs{s_A(z) - \mf{m}(z)} &\prec \begin{cases} \frac{1}{d^{p_\ell}} & \text{if $\ell$ is not an integer,} \\ \frac{1}{d^{p_\ell}} + \frac{1}{d^{\epsilon_0}} & \text{if $\ell$ is an integer,} \end{cases} \\
    \abs{s_{\widetilde{A}}(z) - \mf{m}(z)} &\prec \begin{cases} \frac{1}{d^{p_\ell}} & \text{if $\ell$ is not an integer,} \\ \frac{1}{d^{p_\ell}} + \frac{1}{d^{\epsilon_0}} & \text{if $\ell$ is an integer,} \end{cases},
\end{align*}
where $\mf{m}(z)$ is the unique solution in $\mathbb{H}$
to the equation
\begin{equation}
\label{eqn:frak_m_self_consistent}
    \mf{m}(z) \left(z + \frac{\gamma_a \mf{m}(z)}{1+\gamma_b \mf{m}(z)} + \gamma_c \mf{m}(z)\right) + 1 = 0
\end{equation}
with the $f$-dependent constants
\begin{equation}
\label{eqn:def_gammas}
    \gamma_a \defeq \begin{cases} c_\ell^2 & \text{if } \ell \in \N, \\ 0 & \text{otherwise}, \end{cases} \qquad \gamma_b \defeq \begin{cases} c_\ell \sqrt{\ell! \kappa} & \text{if } \ell \in \N, \\ 0 & \text{otherwise}, \end{cases} \qquad \gamma_c \defeq \sum_{k = \ell_c}^L c_k^2.
\end{equation}
\end{theorem}
 
\begin{rem}
As explained in the proof of \cite[Proposition 9]{LuYau}, it is easy to check that \eqref{eqn:frak_m_self_consistent} has a unique solution in the upper half plane. Indeed, as mentioned above and in the previous literature, if $\ell$ is an integer it is the Stieltjes transform of the free (additive) convolution of the semicircle law and the Mar\v{c}enko-Pastur law, scaled according to $\gamma_a$, $\gamma_b$, and $\gamma_c$; otherwise it is the Stieltjes transform of the semicircle law, rescaled according to $\gamma_c$.
\end{rem}

The first major step of the proof is to show that, for the purposes of a global law, the matrices $A$ and $\widetilde{A}$ are each well-approximated by the matrix
\begin{equation}
\label{eqn:main_results_definition_b}
\begin{split}
    B_{ij} &= \frac{\delta_{i\neq j}}{\sqrt{N}} \sum_{k = \lceil \ell \rceil}^L \frac{c_k\sqrt{k!}}{d^{k/2}}\sum_{\substack{a_1, \ldots, a_k=1 \\ a_1 < a_2 < \ldots < a_k}}^d X_{a_1 i} \ldots X_{a_k i} X_{a_1 j} \ldots X_{a_k j} \\
    &= \frac{\delta_{i\neq j}}{\sqrt{N}} \sum_{k = \lceil \ell \rceil}^L \frac{c_k}{d^{k/2}\sqrt{k!}}\sum_{a_1, \ldots, a_k=1}^{d,\ast} X_{a_1 i} \ldots X_{a_k i} X_{a_1 j} \ldots X_{a_k j},
\end{split}
\end{equation}
which we think of as storing the ``main terms'' present in $A$ and $\widetilde{A}$. (We recall that the notation $\sum_{a_1,\ldots,a_k=1}^{d,\ast}$ means that the $a_1,\ldots,a_k$ are all distinct, but not necessarily ordered. Each of the formulations $\sum_{a_1<\cdots<a_k}^d$ and $\sum_{a_1,\ldots,a_k}^{d,\ast}$ will be more convenient at some point of the proof.) 

\begin{rem}
\label{rem:emptysums}
We remind the reader of the usual convention that sums like $\sum_{k=\ell}^L$ are considered empty if, in this case, $L < \ell$. For example, if $L < \lceil \ell \rceil + 1$, then $\gamma_c = 0$; if $L < \lceil \ell \rceil$, then $B = 0$ as a matrix. In fact, if $L < \ell$, then $\gamma_a = \gamma_b = \gamma_c = 0$, and Theorem \ref{thm:main_polynomial} says that $A$ and $\widetilde{A}$ have bulk spectra tending to a delta mass at zero.
\end{rem}

We split the proof of Theorem \ref{thm:main_polynomial} into the following two propositions. In the statements, we need the parameters
\begin{align*}
    q_\ell &= \min(\ell,1,\ell_c-\ell)/2, \\
    r_\ell &= (1+\ell-\lceil \ell \rceil)/2,
\end{align*}
which satisfy $0 \leq q_\ell, r_\ell \leq 1/2$ for all $\ell > 0$. Together these imply the result, since one can easily compute
\[
    p_\ell = \min(q_\ell,r_\ell).
\]

\begin{prop}
\label{prop:main_stieltjes_conclusion}
Under the assumptions above, we have $s_B(z) \to \mf{m}(z)$, almost surely as $d \to \infty$, with 
\begin{equation}
\label{eqn:sB-mfm}
    \abs{s_B(z) - \mf{m}(z)} \prec \begin{cases} \frac{1}{d^{q_\ell}} & \text{if $\ell$ is not an integer}, \\ \frac{1}{d^{q_\ell}} + \abs{ \frac{N}{d^\ell} - \kappa} & \text{if $\ell$ is an integer}. \end{cases} 
\end{equation}
\end{prop}

\begin{prop}
\label{prop:error_stieltjes_conclusion}
Under the assumptions above, we have $s_A(z) - s_B(z) \to 0$ and $s_{\widetilde{A}}(z) - s_B(z) \to 0$, almost surely as $d \to \infty$, with 
\begin{align*}
    \abs{s_A(z) - s_B(z)} &\prec \frac{1}{d^{r_\ell}}, \\
    \abs{s_{\widetilde{A}}(z) - s_B(z)} &\prec \frac{1}{d^{r_\ell}},
\end{align*}
\end{prop}

As mentioned before, the proof of Proposition \ref{prop:main_stieltjes_conclusion} takes up the bulk of the paper, namely Sections \ref{sec:main_term_sketch} through \ref{sec:tilde_m_self_consistent}; the proof of Proposition \ref{prop:error_stieltjes_conclusion} is much shorter, and is given in Section \ref{sec:error-terms}.
\begin{rem}
\label{rem:diagonal}

We now explain why, in the definitions \eqref{eqn:intro_model_def} and \eqref{eqn:intro_model_def_normalized} of $A$ and $\widetilde{A}$, we set the diagonal entries to zero. Consider the diagonal matrices $K, \widetilde{K} \in \R^{N \times N}$ given entrywise by
\[
    K_{ii} = \frac{1}{\sqrt{N}} f\left(\frac{\|X_i\|^2}{\sqrt{d}}\right), \qquad \widetilde{K}_{ii} = \frac{1}{\sqrt{N}} f(\sqrt{d})
\]
i.e., $K$ contains the ``missing diagonal'' elements of $A$ (resp. $\widetilde{K}$ of $\widetilde{A}$), and at first glance the reader may find the matrices $A + K$ and $\widetilde{A} + \widetilde{K}$ with ``restored diagonal'' elements to be more natural.

Since $\widetilde{K}$ is a deterministic constant times identity, its role is easy to understand: In this case, the matrix $\widetilde{A}$ has an order-one limiting spectral measure, and $\widetilde{A}+\widetilde{K}$ simply translates this measure along the real line. According to the growth of $f$ at infinity and the power $\ell$ in $d^\ell \asymp N$, this shift may be asymptotically negligible, asymptotically constant, or, in the worst case, asymptotically infinity. Proving theorems directly about $\widetilde{A}$ avoids the need to spell out these cases; the reader interested in $\widetilde{A} + \widetilde{K}$ can simply add back the shift.

Since $K$ is genuinely random, its role is more nuanced. In this case we may decompose
\[
    K = \widetilde{K} + K_{\textup{fluct}} \eqdef \frac{1}{\sqrt{N}}f(\sqrt{d}) \Id + \diag \left(\frac{1}{\sqrt{N}} \left[ f\left(\frac{\|X_i\|^2}{\sqrt{d}}\right) - f(\sqrt{d}) \right] \right)_{i=1}^N.
\]
Of course $\widetilde{K}$ plays the same translation role as before, but the role of $K_{\textup{fluct}}$ is new: Roughly speaking, from the CLT we expect $\|X_i\|^2 \approx d + Z_i\sqrt{d}$, with $Z_i$ i.i.d. standard normal, so
\[
    (K_{\textup{fluct}})_{ii} \approx \frac{f(\sqrt{d}+Z_i) - f(\sqrt{d})}{\sqrt{N}} \approx \frac{f'(\sqrt{d})}{\sqrt{N}}Z_i.
\]
If $f'(\sqrt{d}) \ll \sqrt{N}$, this suggests that $\|K_{\textup{fluct}}\|_{\textup{op}}$ is asymptotically negligible, so that $K$ is asymptotically just a simple translation as before. But if $f'(\sqrt{d}) \gg \sqrt{N}$, then --- even discarding the translation $\widetilde{K}$ --- the bulk spectra of $A$ and $A+K_{\textup{fluct}}$ may be substantially different. This is potentially interesting, but to keep the current work to a manageable length, we consider only the zero-diagonal matrices $A$ and $\widetilde{A}$.
\end{rem}

\subsection{General nonlinearities}\

In these more general results, we allow nonlinearities $f : \R \to \R$, which should still be \emph{fixed} (not depending on $d$ or $N$), as long as they are in some sense well-approximable by polynomials. Our conditions on $f$ are the same as in \cite[Lemma 1]{LuYau}, and our (short) proof that Theorem \ref{thm:main_polynomial} on polynomial nonlinearities lifts to Theorem \ref{thm:main_general} through this approximation scheme essentially mimics theirs.

\begin{assn}
\label{assn:poly_approximable}
We assume that our nonlinear function $f(x)$ is piecewise continuous with a polynomial growth rate. Precisely, there exists a positive integer $K$, a finite subdivision $-\infty = \alpha_0 < \alpha_1 < \alpha_2 < \cdots < \alpha_K < \alpha_{K+1} = \infty$, and a finite positive constant $C$ such that
\begin{enumerate}
\item For every $i \in \{0, \ldots, K\}$, the function $f(x)$ is continuous and bounded on the open interval $(\alpha_i, \alpha_{i+1})$.
\item $\abs{f(x)} \leq C\abs{x}^C$ when $x < \alpha_1$ or $x > \alpha_K$.
\end{enumerate}
\end{assn}

Under these assumptions, it is easy to show that the sequence $(c_k)_{k=0}^\infty$ defined by
\begin{equation}
\label{eqn:def_ck_poly_approximable}
    c_k = \E_{Z \sim \mc{N}(0,1)}[f(Z)h_k(Z)]
\end{equation}
exists and is square-summable:
\begin{equation}
\label{eqn:def_sig2_poly_approximable}
    \sigma^2 \defeq \sum_{k=0}^\infty c_k^2 = \E_{Z \sim \mc{N}(0,1)}[f^2(Z)] < \infty.
\end{equation}

\begin{theorem}
\label{thm:main_general}
\textbf{(Main theorem, general nonlinearities)}
Suppose that $\mu$ has all moments finite. Fix $\kappa, \tau, \ell > 0$, and a function $f$ satisfying Assumption \ref{assn:poly_approximable} with corresponding constants $(c_k)_{k=0}^\infty$ and $\sigma^2$ from \eqref{eqn:def_ck_poly_approximable} and \eqref{eqn:def_sig2_poly_approximable}, respectively. Suppose that
\[
    \frac{N}{d^\ell} = \kappa + \OO(d^{-1/2}).
\]
Then, if $\mf{m}(z)$ is the unique solution in $\mathbb{H}$ to the equation
\[
    \mf{m}(z) \left(z + \frac{\gamma_a \mf{m}(z)}{1+\gamma_b \mf{m}(z)} + \widehat{\gamma_c} \mf{m}(z)\right) + 1 = 0
\]
with the $f$-dependent constants
\begin{equation}
\label{eqn:def_gammas_approximable}
    \gamma_a \defeq \begin{cases} c_\ell^2 & \text{if } \ell \in \N, \\ 0 & \text{otherwise}, \end{cases} \qquad \gamma_b \defeq \begin{cases} c_\ell \sqrt{\ell! \kappa} & \text{if } \ell \in \N, \\ 0 & \text{otherwise}, \end{cases} \qquad \widehat{\gamma_c} \defeq \sum_{k=\ell_c}^{\infty} c_k^2,
\end{equation}
then for each fixed $z \in \mathbf{D}_\tau$ we have
\begin{align*}
    s_A(z) &\overset{d \to \infty}{\to} \mf{m}(z) \quad \text{almost surely}, \\
    s_{\widetilde{A}}(z) &\overset{d \to \infty}{\to} \mf{m}(z) \quad \text{almost surely}.
\end{align*}
\end{theorem}

We prove this result by approximating $f$ by polynomials and using the result for polynomial nonlinearities, Theorem \ref{thm:main_polynomial}, as a black box. This is essentially the same approach as \cite[Theorem 2]{LuYau}, although they can allow slightly more general nonlinearities $f$ because they have exact formulas for the distributions, which additionally have better tail decay. In Appendix \ref{appx:general_nonlinearities}, we give the needed modifications for completeness. 


\section{Sketch proof of Proposition \ref{prop:main_stieltjes_conclusion}: Main terms as sample covariance matrices}
\label{sec:main_term_sketch}

The main idea of the proof is to rewrite $B$ in the form of a sample-covariance matrix $U^*TU$, where $U$ has random, centered real entries with unit variance, and where $T$ is a real deterministic diagonal matrix storing the prefactors in \eqref{eqn:main_results_definition_b}. (Actually, since the diagonal terms of $B$ are set to zero, we will need to write $B = U^\ast T U - D$ where $D$ stores the diagonal terms.) The main technical difficulty comes from the fact that, while the columns of $U$ in such a decomposition are independent, the entries in each column are only uncorrelated, not independent. A smaller technical difficulty comes from the fact that $U$ will be of size $M \times N$ for some $M \gg N$ (roughly speaking, $M \approx N^{L/\ell}$ -- and it suffices to restrict to the case $L > \ell$, see Remark \ref{rem:only_low_degree}), whereas the most-studied sample covariance matrices typically have $M \asymp N$.

To do this, we need the following notation. Fix once and for all some nonlinearity $f$, and recall that $c_k$ is the $k$th coefficient of $f$ in the Hermite basis, $f(x) = \sum_{k=0}^L c_k h_k(x)$. For every $k \in \llbracket \lceil \ell \rceil, L \rrbracket$ define 
\begin{equation}
\mathbf{M}_k = \begin{cases} \{(a_1,\ldots, a_k) \in [d]^k \text{ s.t. } a_1 < \ldots < a_k\} & \text{if } c_k \neq 0, \\ \emptyset & \text{if } c_k = 0, \end{cases}
\end{equation}
with size $M_k = \abs{\mathbf{M}_k}$. We will also need their union  $\mathbf{M} = \bigcup_{k=\lceil \ell \rceil}^L \mathbf{M}_k$, with total size $M = \sum_{k=\lceil \ell \rceil}^L M_k$.

For each $k$ with $M_k > 0$, we introduce the matrix $U^{[k]} \in \R^{M_k \times N}$ with entries 
\begin{equation}
    U^{[k]}_{\mu i} = \frac{1}{\sqrt{N}} X_{a_1 i} \ldots X_{a_k i},
\end{equation}
where $\mu$ is the tuple
\[
    \mu = (a_1,\ldots, a_k) \in \mathbf{M}_k.
\]
Given two tuples $\mu$ and $\nu$, we define their \emph{overlap} 
\[
    \ip{\mu,\nu} \defeq \abs{ \mu \cap \nu},
\]
where $\mu$ and $\nu$ are viewed as subsets of $\llbracket 1, d \rrbracket$. For example, $\ip{(2, 4, 5), (3, 5)} = 1$.

We also define the combined-degree $M \times N$ matrix $U$ such that $U_{\mu i} = U^{[k]}_{\mu i}$ for $\mu \in \mathbf{M}_k$ and all $k$. Note that, as claimed, $\E U_{\mu i} = 0$ and $\E U_{\mu i}^2 = \frac{1}{N}$ for all $\mu$ and $i$; the columns $U_i$ are independent for different $i$; but $U_{\mu i}$ and $U_{\nu i}$ are only uncorrelated for $\mu \neq \nu$, not necessarily independent, since $\ip{\mu,\nu}$ can be nonzero. The deterministic $M \times M$ matrix $T$ is defined blockwise
\begin{equation}
    T = \begin{pmatrix}
    \ddots & & 0 \\
     & c_k\sqrt{k!} \cdot \sqrt{\frac{N}{d^k}} I_{M_k} & \\
     0 & & \ddots
    \end{pmatrix},
\end{equation}
where we skip blocks with $c_k = 0$ (i.e., $I_0$ is a $0 \times 0$ matrix) by convention. The point of all these conventions is to define $T$ in an \emph{invertible} way, by omitting what would otherwise be zero blocks. For example, if $f = c_1h_1 + c_2h_2 + c_4h_4$ and $\ell \leq 1$, then
\[
    T = \begin{pmatrix}
    c_1 \sqrt{\frac{N}{d}} I_{M_1} & 0 & 0 \\
    0 & c_2 \sqrt{\frac{2!N}{d^k}} I_{M_2} & 0 \\
    0 & 0 & c_4 \sqrt{\frac{4!N}{d^k}} I_{M_4}
    \end{pmatrix}.
\]
If we change this example to keep the same $f$ but let $1 < \ell \leq 2$, say, then the first of the three blocks disappears, and $T$ becomes smaller. Since $T$ is diagonal, we will usually write $T_\mu$ instead of $T_{\mu\mu}$, and sometimes write $T_k$ instead of $T_\mu$ when $\mu \in \mathbf{M}_k$ (since the value of $T_\mu$ depends only on this $k$).

\begin{rem}
\label{rem:only_low_degree}
If $f$ has only low-degree terms, i.e. $f = \sum_{k=0}^{\lceil \ell \rceil - 1} c_k h_k(x)$, then by these conventions $T$ does not exist at all. But in this case, the matrix $B$ defined by \eqref{eqn:main_results_definition_b} is zero, so $s_B(z) = -\frac{1}{z} = \mf{m}(z)$, where $\mf{m}(z)$ is defined by \eqref{eqn:frak_m_self_consistent}, and Proposition \ref{prop:main_stieltjes_conclusion} is immediate. For such matrices, the main result -- namely, that both $A$ and $\widetilde{A}$ have bulk spectrum which is asymptotically a delta mass at zero -- follows from Proposition \ref{prop:error_stieltjes_conclusion}, whose proof does not use the matrix $T$. Thus in the following we will always assume that $f$ has some high-degree terms, so that $T$ is nontrivial.
\end{rem}

Notice that
\begin{equation}\label{eq:B_ij-sample-cov}
    B_{ij} = (U^*TU)_{ij} \delta_{i\neq j}.
\end{equation}
Denote the diagonal part removed in \eqref{eq:B_ij-sample-cov} by $D$, i.e. $D_{ii} = U_i^* T U_i$ and $D_{ij}=0$ for $i\neq j$. The fundamental observation is that
\begin{equation}
    B = U^* T U - D.
\end{equation}

Consider
\begin{equation}\label{def:H}
    H(z) = \begin{pmatrix}
    -T^{-1} & U \\
    U^* & -z - D
    \end{pmatrix}.
\end{equation}
Define the resolvent
\begin{equation}
    G(z) = H(z)^{-1}.
\end{equation}
From the definition of $H$ in \eqref{def:H} we can see that
\begin{equation}
    G(z) = \begin{pmatrix}
    G_M(z) & * \\
    * & G_N(z)
    \end{pmatrix},
\end{equation}
where the $*$'s are some block matrices that are irrelevant for our purposes, 
\begin{equation}
\label{eqn:gn_formula}
    G_N(z) = \frac{1}{U^* T U - z - D} = \frac{1}{B - z}
\end{equation}
is the resolvent we ultimately want to understand, and
\begin{equation}
\label{eqn:gm_formula}
    G_M(z) = \frac{1}{U(z+D)^{-1}U^* - T^{-1}}
\end{equation}
is an object we will understand as an intermediate step. 

Throughout this paper, Greek letter indices like $\mu$ and $\nu$ will refer to the first $M$ columns and rows of $G$, while letters like $i,j,k$ will refer to the last $N$ columns and rows of $G$, so that for example $G_{ij} = (G_N)_{ij}$ and $G_{\mu \nu} = (G_M)_{\mu \nu}$.

The fundamental quantities for the main-term analysis are
\begin{align*}
    s(z) &\defeq s_B(z) = \frac{1}{N} \tr G_N(z), \\
    \widetilde{s}(z) &= \frac{1}{M} \tr \left(G_M(z) + T \right), \\
    \phi &= \frac{M}{N},
\end{align*}
where we have dropped $B$ from the notation $s_B$ to save space. We stress that $s(z)$ and $\widetilde{s}(z)$ depend on $N$, although we suppress this from the notation. Later, we will show that $s(z)$ and $\phi\widetilde{s}(z)$ are order-one quantities for $z \in \mathbf{D}_\tau$. Since $z$ is fixed throughout the argument, we will often drop it from the notation, writing for example
\[
    s = s(z), \qquad \widetilde{s} = \widetilde{s}(z), \qquad G_N = G_N(z), \qquad G_M = G_M(z), \qquad \ldots
\]
Since $G_N$ is the resolvent of our matrix of interest, the eventual goal is to show that $s(z)$ approximately satisfies a self-consistent equation. To do this, we pass through the auxiliary matrix $G_M$, which is \emph{not} a resolvent, but which our analysis shows approximately behaves like one (for example, it approximately satisfies something like the Ward identity; see Lemmas \ref{lem:full-Ward} and \ref{lem:partial-Ward}). Precisely, we first show that $s(z)$ and $\phi\widetilde{s}(z)$ approximately determine each other through a \emph{joint} self-consistent equation (roughly, $s(z) \approx -(z+\phi\widetilde{s}(z))^{-1}$). Then we show show that $\phi\widetilde{s}(z)$ approximately satisfies its own self-consistent equation. From here we recover the self-consistent equation approximately satisfied by $s(z)$, which is exactly satisfied by $\mathfrak{m}(z)$, then use perturbation theory of that equation, already developed by \cite{LuYau}, to conclude. These steps are split into the following propositions, whose proofs constitute the bulk of the paper.

\begin{prop}
\label{prop:m_widetildem_self_consistent}
For any fixed $\tau > 0$ and any fixed $z \in \mathbf{D}_\tau$, we have
\[
    \abs{1+s(z)(z+\phi\widetilde{s}(z))} \prec \frac{1}{d^{\frac{1}{2}\min(1,\ell)}}.
\]
\end{prop}

\begin{prop}
\label{prop:widetildem_self_consistent}
Let $\gamma_a$, $\gamma_b$, $\gamma_c$ be as in \eqref{eqn:def_gammas}. For any fixed $\tau > 0$ and any fixed $z \in \mathbf{D}_\tau$, we have
\begin{equation}
\label{eqn:widetildem_self_consistent}
    \abs{\phi \widetilde{s}(z) - \frac{\gamma_a}{\gamma_b - z - \phi \widetilde{s}(z)} + \frac{\gamma_c}{z+\phi \widetilde{s}(z)}} \prec \begin{cases} \frac{1}{d^{q_\ell}} & \text{if $\ell$ is not an integer,} \\ \frac{1}{d^{q_\ell}} + \abs{\frac{N}{d^\ell} - \kappa} & \text{if $\ell$ is an integer.} \end{cases} 
\end{equation}
\end{prop}

\begin{prop}
\label{prop:m_self_consistent}
For any fixed $\tau > 0$ and any fixed $z \in \mathbf{D}_\tau$, we have
\begin{equation}
\label{eqn:m_self_consistent}
    \abs{\frac{1}{s(z)} + z + \frac{\gamma_a s(z)}{1+ \gamma_b s(z)} + \gamma_c s(z)} \prec \begin{cases} \frac{1}{d^{q_\ell}} & \text{if $\ell$ is not an integer,} \\ \frac{1}{d^{q_\ell}} + \abs{\frac{N}{d^\ell} - \kappa} & \text{if $\ell$ is an integer.} \end{cases} 
\end{equation}
\end{prop}

Modulo Propositions \ref{prop:m_widetildem_self_consistent}, \ref{prop:widetildem_self_consistent}, and \ref{prop:m_self_consistent}, the proof of Proposition \ref{prop:main_stieltjes_conclusion} is quite short. We re-write \eqref{eqn:frak_m_self_consistent}, $1+\mf{m}(z)(z+\gamma_a\mf{m}(z)/(1+\gamma_b\mf{m}(z)) + \gamma_c\mf{m}(z)) = 0$, as
\begin{equation}
\label{eqn:frak_m_self_consistent_flipped}
    \frac{1}{\mf{m}(z)} + z + \frac{\gamma_a \mf{m}(z)}{1+\gamma_b\mf{m}(z)} + \gamma_c \mf{m}(z) = 0,
\end{equation}
in order to recall the following stability analysis of this equation, due to \cite{LuYau}.

\begin{lemma}
\label{lem:luyau_prop_9}
\cite[Proposition 9]{LuYau} If deterministic $\mathfrak{s} = \mf{s}(z)$ approximately solves \eqref{eqn:frak_m_self_consistent_flipped} in the sense that
\[
    \frac{1}{\mf{s}} + z + \frac{\gamma_a \mf{s}}{1+\gamma_b \mf{s}} + \gamma_c \mf{s} = \omega
\]
with the error term $\omega$ satisfying
\[
    \abs{\omega} \leq \frac{\eta}{2}, \quad \eta = \im(z),
\]
and $\mf{m} = \mf{m}(z)$ exactly solves \eqref{eqn:frak_m_self_consistent_flipped}, then
\[
    \abs{\mf{s}-\mf{m}} \leq \frac{4\abs{\omega}}{\eta^2}.
\]
\end{lemma}

\begin{proof}[Proof of Proposition \ref{prop:main_stieltjes_conclusion}]
Write
\[
    \omega \defeq \frac{1}{s(z)} + z + \frac{\gamma_a s(z)}{1+\gamma_b s(z)} + \gamma_c s(z).
\]
Let $\delta_\ell = q_\ell$, if $\ell$ is not an integer, or $\delta_\ell = \min(q_\ell,\epsilon_0)$, if $\ell$ is an integer. For $\epsilon < \delta_\ell$, applying Lemma \ref{lem:luyau_prop_9}, we find 
\begin{align*}
    \P(\abs{s(z) - \mf{m}(z)} \geq d^{\epsilon-\delta_\ell}) &\leq \P(\abs{\omega} \geq \eta/2) + \P(\abs{s(z) - \mf{m}(z)} \geq d^{\epsilon-\delta_\ell}, \abs{\omega} \leq \eta/2) \\
    &\leq \P(\abs{\omega} \geq \eta/2) + \P(\abs{\omega} \geq (\eta^2/4) d^{\epsilon-\delta_\ell}) \\
    &\leq 2\P(\abs{\omega} \geq (\eta^2/4)d^{\epsilon-\delta_\ell}) \\
    &\leq C_{\epsilon,D} d^{-D},
\end{align*}
where the last inequality follows from Proposition \ref{prop:m_self_consistent}
for $d$ sufficiently large. This verifies \eqref{eqn:sB-mfm}, and the almost-sure convergence of $s(z) - \mf{m}(z)$ to zero follows from the Borel-Cantelli lemma. (This is why we require $\abs{\frac{N}{d^\ell} - \kappa} \leq d^{-\epsilon_0}$ when $\ell$ is an integer; if e.g. $\abs{\frac{N}{d^\ell} - \kappa} \sim \frac{1}{\log(d)}$, then this argument would give $\abs{s(z)-\mf{m}(z)} \leq \frac{d^\epsilon}{\log(d)}$ for all $d$ sufficiently large, which is of course insufficient for almost-sure convergence.)
\end{proof}


\section{Basic tools: Resolvent identities}
\label{sec:resolvent_identities}

The goal of this section is to prove several exact equalities relating $G_N$, $G_M$, and their corresponding minors, which will be used throughout the paper.

\begin{defn}[Minors]
The matrix $H(z)$ from \eqref{def:H} is $(M+N) \times (M+N)$. If $E \subset \llbracket 1, N \rrbracket$ is any so-called exclusion set, we will write $H^{(E)}(z)$ for the $(M+(N - \abs{E})) \times (M+(N - \abs{E}))$ matrix obtained from $H(z)$ given by erasing the rows and columns indicated by $E$. Most frequently we will use $E = \{i\}$ for some index $i$, in which case we abuse notation by writing $H^{(i)}$ instead of $H^{(\{i\})}$. We define the corresponding resolvent by
\[
    G^{(E)}(z) \defeq H^{(E)}(z)^{-1}.
\]
If $E = \emptyset$, by convention we set $H^{(E)} = H$ and $G^{(E)} = G$. Although $H^{(E)}(z)$ and $G^{(E)}(z)$ have fewer rows and columns than $H(z)$ and $G(z)$, we keep the original values of the matrix indices: For example, $G^{(i)}(z)$ has entries $G^{(i)}_{jk}$ for $j,k \in \{1, \ldots, i-1, i+1, \ldots, N\}$, not $\{1, \ldots, N-1\}$.
\end{defn}

Notice from the definition that we will only ever need minors that remove some $i$ and $j$ indices, never those that remove some $\mu$ and $\nu$ indices.

Recall that $G_N = (B - z)^{-1}$, where $B_{ij} = U_i^\ast T U_j \delta_{i\neq j}$. We introduce the notation $B_j$ for the $j$-th column of $B$ with diagonal element excluded, i.e. $B_j = \left(B_{1j},\ldots, B_{j-1,j}, B_{j+1,j}, \ldots, B_{Nj}\right)^T$.

\begin{lemma}[Resolvent identities]
For any $i$, any $\mu$, and any (possibly empty) exclusion set $E$, we have
\begin{align}
    G_{\mu\nu}^{(i)} &= G_{\mu\nu} - \frac{G_{\mu i}G_{i \nu}}{G_{ii}} \label{eq:res-id-off-diagonal} \\
    G_{i\mu} &= - G_{ii} \sum_{\alpha\in \mathbf{M}} U_{\alpha i} G_{\alpha\mu}^{(i)} \label{eq:res-id-G-mu-i} \\
    G_{ii} &= \left(-z - \sum_{\mu, \nu\in\mathbf{M}} U_{\mu i} \left(G^{(i)}_M+T\right)_{\mu \nu} U_{\nu i}\right)^{-1} = \left( -z-D_{ii} - \sum_{\mu,\nu \in \mathbf{M}} U_{\mu i} G^{(i)}_{\mu \nu} U_{\nu i} \right)^{-1} \label{eq:res-id-schur} \\
    G_M^{(E)}+T &= TU^{(E)}G_N^{(E)}(U^{(E)})^\ast T
    \label{eq:res-id-G+T} \\
    (G_M + T)_{\mu\nu} &= -T_\mu \sum_{j=1}^N G_{jj} \sum_{\alpha\in\mathbf{M}} U_{\mu j} U_{\alpha j} G_{\alpha\nu}^{(j)} \label{eq:res-id-G-mu-nu}
\end{align}
\end{lemma}

\begin{proof}
The identities \eqref{eq:res-id-off-diagonal} and \eqref{eq:res-id-G-mu-i} are very standard in the local-law literature (see, e.g., \cite[Lemma 3.5]{BenKno2016} for a proof). The usual version of \eqref{eq:res-id-G-mu-i} simplifies to what we have written here since the $N \times N$ block of $H$ is diagonal, i.e. $H_{ik} = 0$ for $i \neq k$, so that
\[
    G_{i\mu} = -G_{ii} \left( \sum_{\alpha \in \mathbf{M}} U_{\alpha i} G_{\alpha \mu}^{(i)} + \sum_{k}^{(i)} H_{ik} G_{k\mu}^{(i)} \right) = -G_{ii} \sum_{\alpha \in \mathbf{M}} U_{\alpha i} G_{\alpha \mu}^{(i)}.
\]
The identity \eqref{eq:res-id-schur} is just the usual Schur complement formula, again summing only over $\mu$ and $\nu$ for the same reason.

In proving \eqref{eq:res-id-G+T}, by erasing rows and columns as necessary we may assume $E = \emptyset$; then it is a simple arithmetic consequence of the formulas \eqref{eqn:gn_formula} and \eqref{eqn:gm_formula} for $G_N$ and $G_M$, respectively. Namely, \eqref{eqn:gn_formula} gives
\[
    \Id = G_N(U^\ast T U - D - z) = G_N U^\ast T U - G_N(z+D),
\]
so 
\[
    (z+D)^{-1} - G_NU^\ast TU(z+D)^{-1}+G_N = (\Id - G_NU^\ast TU)(z+D)^{-1} + G_N = -G_N(z+D)(z+D)^{-1}+G_N = 0,
\]
so that
\begin{align*}
    \Id &= \Id -TU[(z+D)^{-1} - G_N U^\ast T U (z+D)^{-1} + G_N]U^\ast \\
    &= -TU(z+D)^{-1}U^\ast + \Id + TUG_NU^\ast TU(z+D)^{-1}U^\ast - TUG_NU^\ast TT^{-1} \\
    &= (-T+TUG_NU^\ast T)(U(z+D)^{-1}U^\ast - T^{-1}) \\
    &= (-T+TUG_NU^\ast T)G_M^{-1},
\end{align*}
where the last equality is \eqref{eqn:gm_formula}. This means that
\[
    G_M = -T + TUG_NU^\ast T,
\]
which can be rearranged to obtain \eqref{eq:res-id-G+T}. 

The proof of \eqref{eq:res-id-G-mu-nu} is the most involved. It starts with the Sherman--Morrison formula for the inverse of a rank-one update, which is usually formulated for a matrix $A$, column vector $q$, and scalar $\tau$ as 
\[
    (A+\tau qq^\ast)^{-1} = A^{-1} - \frac{\tau A^{-1} qq^\ast A^{-1}}{1+\tau q^\ast A^{-1}q}.
\]
Left-multiplying by $q^\ast$ and simplifying on the right-hand side, we obtain
\[
    q^\ast (A+\tau qq^\ast)^{-1} = \frac{(1+\tau q^\ast A^{-1} q)q^\ast A^{-1} - \tau q^\ast A^{-1} qq^\ast A^{-1}}{1+\tau q^\ast A^{-1}q} = \frac{1}{1+\tau q^\ast A^{-1}q} q^\ast A^{-1}.
\]
If we use this with $q = U_j$, $\tau = (z+D_j)^{-1}$, and $A = U(z+D)^{-1}U^\ast - U_j(z+D_j)^{-1}U_j^\ast - T^{-1} = (G_M^{(j)})^{-1}$, so that $A+\tau qq^\ast = (G_M)^{-1}$, we obtain an equation over vectors of length $M$; taking the $\nu$ entry of each side, we find
\begin{equation}
\label{eqn:sherman-morrison-consequence}
    \sum_\alpha U_{\alpha j} G_{\alpha \nu} = \frac{1}{1+(z+D_j)^{-1} U_j^\ast G_M^{(j)} U_j} \sum_\alpha U_{\alpha j} G^{(j)}_{\alpha \nu}.
\end{equation}
Store this for a moment; at the same time, rearrange \eqref{eqn:gm_formula} to obtain
\begin{equation}
    \left(U(z+D)^{-1} U^* - T^{-1}\right)G_M = \Id,
\end{equation}
then take the $(\mu\nu)$ entry of both sides to get
\begin{equation}\label{eq:mu-mu-entry}
    \left(\sum_j (z + D_{j})^{-1}\right) U_{\mu j} \left(\sum_\alpha U_{\alpha j} G_{\alpha \nu}\right) - T_\mu^{-1} G_{\mu\nu} = \delta_{\mu\nu}.
\end{equation}
Now we substitute \eqref{eqn:sherman-morrison-consequence} into the left-hand side, and multiply both sides by $T_\mu$; this yields
\[
    \sum_j \frac{T_\mu}{z+D_j+U_j^\ast G_M^{(j)} U_j} \sum_\alpha U_{\mu j} U_{\alpha j} G^{(j)}_{\alpha \nu} - G_{\mu\nu} = T_{\mu\nu}.
\]
Substituting $\frac{1}{z+D_j+U_j^\ast G_M^{(j)}U_j} = -G_{jj}$, from \eqref{eq:res-id-schur}, finishes the proof of \eqref{eq:res-id-G-mu-nu}.
\end{proof}


\section{Basic tools: Full and partial Ward inequalities}
\label{sec:ward}

Since $G_N = (B-z)^{-1}$ is actually a resolvent, it satisfies the usual Ward identity
\begin{equation}
    G_N G_N^* = \frac{\im G_N}{\eta},
\end{equation}
where $\eta = \im z$, and actually the extension to minors
\begin{equation}\label{eq:GN-Ward}
    G_N^{(E)} (G_N^{(E)})^\ast = \frac{\im G_N^{(E)}}{\eta}.
\end{equation}

Since $G_M+T$ is not a resolvent, it does not satisfy the Ward identity. However, the goal of this section is to show that it approximately satisfies an \emph{inequality} that looks like one direction of the Ward identity (roughly speaking, $(G_M +T) (G_M+T)^\ast \lesssim \im (G_M+T)/\eta$, or in coordinates $\sum_{\nu} \abs{(G_M+T)_{\mu \nu}}^2 \lesssim \im (G_M+T)_{\mu\mu}/\eta$; then $\im(T)$ disappears since $T$ is real). This says that, for each $\mu$, the sum of $\abs{(G_M+T)_{\mu \nu}}^2$ over $\nu$ is much smaller than a naive estimate would predict. Actually we show something better, which is crucial for our proof of Lemma \ref{prop:m_widetildem_self_consistent}: This sum is also smaller than expected if it is taken, not over \emph{all} tuples $\nu$, but just over some of them, namely those with a fixed overlap with $\mu$. (Recall that $\ip{\mu,\nu}$ denotes the overlap of the multi-indices $\mu$ and $\nu$, i.e., the number of $X$'s they have in common.) We call such estimates \emph{partial Ward inequalities}, in contrast with the original estimates, which we call \emph{full Ward inequalities}. 

\begin{lemma}[Full Ward inequality]
\label{lem:full-Ward}
For any $\mu$, and any (possibly empty) exclusion set $E \subset [N]$ with size $\OO(1)$, we have
\[
    \sum_{\nu \in \mathbf{M}} \abs{(G^{(E)}_M+T)_{\mu\nu}}^2 \prec \frac{\im G^{(E)}_{\mu\mu}}{\eta},
\]
uniformly over $\mu \in \mathbf{M}_k$.
\end{lemma}

Note that full Ward inequality follows directly from partial Ward inequalities, so we omit the proof.

\begin{lemma}[Partial Ward inequality]
\label{lem:partial-Ward}
For any $k_1, k_2 \in \llbracket \lceil \ell \rceil, L \rrbracket$, $t \in \llbracket 0, \min(k_1,k_2) \rrbracket$ and $\mu \in \mathbf{M}_{k_1}$, let 
\[
    S_\mu^{k_2, t} = \left\{\nu\in\mathbf{M}_{k_2} \mid \langle\mu,\nu\rangle = t\right\}.
\]
Then for any $\mu \in \mathbf{M}_{k_1}$ and any (possibly empty) exclusion set $E \subset [N]$ with size $\OO(1)$, we have
\begin{equation}\label{eq:partial-Ward}
    \sum_{\nu \in S_\mu^{k_2,t}} \left|(G_M^{(E)}+T)_{\mu \nu}\right|^2 \prec \frac{\im G^{(E)}_{\mu\mu}}{\eta} d^{\max\{-t, \ell-k_2\}},
\end{equation}
uniformly over $\mu \in \mathbf{M}_{k_1}$.
\end{lemma}

\begin{proof}[Proof of Lemma \ref{lem:partial-Ward}]
For notational simplicity, let $S \defeq \abs{S_\mu^{k_2,t}}$ and $\tilde{N} \defeq N - \abs{E}$. Because of our convention that $\mathbf{M}_k = \emptyset$ when the $k$th Hermite coefficient of $f$ vanishes, we can have $S = 0$, but in this case \eqref{eq:partial-Ward} is trivial, so we can and will assume $S_\mu^{k_2,t} \neq \emptyset$. Consider the restrictions of $T$ and $U$ to this subset, denoted by $T\big|_{S_\mu^{k_2, t}}$ and $U\big|_{S_\mu^{k_2, t}}$; these are matrices of size $S \times S$ and $S \times \tilde{N}$, respectively, and $T\big|_{S_\mu^{k_2, t}} = c_{k_2} \sqrt{\frac{k_2! N}{d^{k_2}}}I$. We will write $(U^{(E)}\big|_{S_\mu^{k_2, t}})_i$ for the $i$th column of $U^{(E)}\big|_{S_\mu^{k_2, t}}$. Using the resolvent identity $G_M^{(E)} + T = TU^{(E)}G_N^{(E)}(U^{(E)})^\ast T$ from \eqref{eq:res-id-G+T}, we can view the left-hand side of \eqref{eq:partial-Ward} as
\begin{align}
    \sum_{\nu \in S^{k_2,t}_\mu} \left|(G_M^{(E)}+T)_{\mu \nu}\right|^2 &= \sum_{\nu \in S^{k_2,t}_\mu} \abs{(TU^{(E)}G_N^{(E)}(U^{(E)})^\ast T)_{\mu \nu}}^2 = \sum_{\nu \in S^{k_2,t}_\mu} \sum_{i,j,k,\ell=1}^{N,(E)} T_\mu U_{\mu i} G_{ij} U_{\nu j} T_\nu T_\nu U_{\nu k} (G^\ast)_{k \ell} U_{\mu \ell} T_\mu \\
    &= c_{k_2}^2 k_2! \frac{N}{d^{k_2}} \sum_{\nu \in S^{k_2,t}_\mu} \sum_{i,j,k,\ell=1}^{N,(E)} T_\mu U_{\mu i} G_{ij} U_{\nu j} U_{\nu k} (G^\ast)_{k \ell} U_{\mu \ell} T_\mu \\
    &= c_{k_2}^2 k_2! \frac{N}{d^{k_2}} \sum_{i,j,k,\ell=1}^{N,(E)} T_\mu U_{\mu i} G_{ij} \left( U\big|_{S_\mu^{k_2, t}}^*U\big|_{S_\mu^{k_2, t}} \right)_{jk} (G^\ast)_{k\ell} (U^\ast)_{\ell \mu} T_\mu \\
    &= c_{k_2}^2 k_2! \frac{N}{d^{k_2}} \left(T_\mu U^{(E)}G_N^{(E)} U^{(E)}\big|_{S_\mu^{k_2, t}}^*U^{(E)}\big|_{S_\mu^{k_2, t}}(G_N^{(E)})^*(U^{(E)})^* T_\mu\right)_{\mu\mu} \\
    &= c_{k_2}^2 k_2! \frac{N}{d^{k_2}} \ip{((G_N^{(E)})^\ast (U^{(E)})^\ast T e_\mu),\left( U^{(E)}\big|_{S_\mu^{k_2, t}}^*U^{(E)}\big|_{S_\mu^{k_2, t}} \right) ((G_N^{(E)})^\ast (U^{(E)}) ^\ast T e_\mu)}
\end{align}
(recalling that notations like $\sum_{i=1}^{N,(E)}$ mean summing over indices $i \in \llbracket 1, N \rrbracket \setminus E$), where $e_\mu$ is the standard basis vector of size $S \times 1$, so that $(G_N^{(E)})^\ast (U^{(E)})^\ast T e_\mu$ is a vector in $\C^{\tilde{N}}$. Bounding the quadratic form by the operator norm times the vector norm, then applying the (non-partial) Ward identity \eqref{eq:GN-Ward} for $G_N$, we find
\begin{align}
    \sum_{\nu \in S^t_\mu} \left|(G_M^{(E)}+T)_{\mu \nu}\right|^2 &\lesssim \frac{N}{d^{k_2}}\left\|U^{(E)}\big|_{S_\mu^{k_2, t}}^*U^{(E)}\big|_{S_\mu^{k_2, t}}\right\|_{\textup{op}} \left(TU^{(E)}G_N^{(E)} (G_N^{(E)})^*(U^{(E)})^*T\right)_{\mu\mu} \\
    &= \frac{N}{d^{k_2}} \left\|U^{(E)}\big|_{S_\mu^{k_2, t}}^*U^{(E)}\big|_{S_\mu^{k_2, t}}\right\|_{\textup{op}} \frac{1}{\eta} \left(TU^{(E)}\im G_N^{(E)} (U^{(E)})^*T\right)_{\mu\mu} \\
    &= \frac{N}{d^{k_2}} \left\|U^{(E)}\big|_{S_\mu^{k_2, t}}^*U^{(E)}\big|_{S_\mu^{k_2, t}}\right\|_{\textup{op}} \frac{\im G^{(E)}_{\mu\mu}}{\eta},
\end{align}
where in the last equality we use that $T$ and $U^{(E)}$ have real entries, so that 
\[
    TU^{(E)}\im G_N^{(E)} (U^{(E)})^\ast T = \im(TU^{(E)}G_N^{(E)} (U^{(E)})^\ast T),
\]
and once again the identity $G_M^{(E)}+T = TU^{(E)}G_N^{(E)}(U^{(E)})^\ast T$ from \eqref{eq:res-id-G+T}.

It remains to show that
\begin{equation}\label{eq:partial-ward-norm-bound}
    \frac{N}{d^{k_2}} \left\|U^{(E)}\big|_{S_\mu^{k_2, t}}^*U^{(E)}\big|_{S_\mu^{k_2, t}}\right\|_{\textup{op}} \prec d^{\max\{-t, \ell-k_2\}} \eqdef \Phi.
\end{equation}
Since the distribution of the left-hand side depends only on the length of $\mu$ and not which $X$'s it contains, the final result \eqref{eq:partial-Ward} will be uniform in $\mu \in \mathbf{M}_{k_1}$ as claimed. 

Recall that the columns $\left( U^{(E)}\big|_{S_\mu^{k_2, t}} \right)_i$ are independent; thus, since $\|AA^\ast\|_{\textup{op}} = \|A^\ast A\|_{\textup{op}}$, 
\begin{equation}
\begin{split}
    \frac{N}{d^{k_2}} \left\|U^{(E)}\big|_{S_\mu^{k_2, t}}^*U^{(E)}\big|_{S_\mu^{k_2, t}}\right\|_{\textup{op}} &= \frac{N}{d^{k_2}} \left\|U^{(E)}\big|_{S_\mu^{k_2, t}}U^{(E)}\big|_{S_\mu^{k_2, t}}^*\right\|_{\textup{op}} \\
    &= \left\|\sum_{i=1}^{N,(E)} \frac{N}{d^{k_2}} \left(U\big|_{S_\mu^{k_2, t}}\right)_i\left(U\big|_{S_\mu^{k_2, t}}\right)_i^*\right\|_{\textup{op}} \eqdef \left\| \sum_{i=1}^{N,(E)} Z_i \right\|_{\textup{op}}
\end{split}
\end{equation}
reduces the problem to bounding the operator norm a sum of independent and identically distributed matrices $Z = \sum_{i=1}^{N,(E)} Z_i$, each $Z_i \in \R^{S \times S}$, which we can do with the matrix Bernstein inequality (see, e.g., Theorem 5.4.1 of \cite{Ver2018}). Notice that matrices $Z_i$ are not centered and do not have bounded norm; thus we need to modify them to use this inequality. To simplify the notation, from now on, we assume without loss of generality that $E = \llbracket N-\abs{E}+1, \ldots, N \rrbracket$, so that $\sum_{i=1}^{N,(E)} = \sum_{i=1}^{\tilde{N}}$. We additionally abuse notation by writing $N$ instead of $\tilde{N}$; this is fine in asymptotics since $\abs{E} = \OO(1)$. 

To show \eqref{eq:partial-ward-norm-bound}, it suffices to show that for any $\epsilon, D > 0$ there is $N_0(\epsilon, D) \in \Z_{>0}$ such that for any $N \ge N_0(\epsilon, D)$ the following inequality holds:
\begin{equation}
    \P\left(\left\| \sum_{i=1}^{N} Z_i \right\|_{\textup{op}} > d^\epsilon \Phi\right) \le N^{-D}. 
\end{equation}
Fix $\epsilon > 0$ and $D > 0$. Pick $\delta \in (0, \epsilon/2)$. Define the centered matrices
\begin{equation}\label{eq:Ztilde-def}
    \tilde{Z}_i = Z_i \mathds{1}(\|Z_i\| \le d^{-t+\delta}) - \E\left[Z_i \mathds{1}(\|Z_i\| \le d^{-t+\delta})\right].
\end{equation}
Notice that 
\begin{equation}
    \|Z_i\| = \frac{N}{d^{k_2}} \left\|\left(U\big|_{S_\mu^{k_2, t}}\right)_i\right\|_2^2 = \frac{N}{d^{k_2}} \sum_{\alpha\in S_\mu^{k_2, t}} U_{\alpha i}^2 \prec \left|S_\mu^{k_2, t}\right| \frac{N}{d^{k_2}} \frac{1}{N} \prec d^{k_2 - t}d^{-k_2} = d^{-t},
\end{equation}
and this bound is uniform in $i \in [N]$, as the $Z_i$ are independent and identically distributed matrices. Then there exists $\tilde{N}_0$ such that for any $N \ge \tilde{N}_0$ we have 
\begin{equation}\label{eq:Zi-bound-probability}
    \P\left(\|Z_i\| > d^{-t+\delta}\right) \le N^{-D-2}
\end{equation}
for all $i\in [N]$. Consider the event 
\begin{equation}
    \Omega(N, \epsilon, D) = \left\{ \forall i\in [N] \vcentcolon  \|Z_i\| \le d^{-t+\delta}\right\}.
\end{equation}
We will use it later to upgrade the result of matrix Bernstein inequality from $\tilde{Z}_i$ to $Z_i$; for now we just store that, from \eqref{eq:Zi-bound-probability} and the union bound, we have
\begin{equation}\label{eq:Omega-prob-bound}
    \P\left(\Omega(N, \epsilon, D)\right) \ge 1 - N^{-D-1}.
\end{equation}

The matrix Bernstein inequality for the real-symmetric $S \times S$ matrices $(\tilde{Z}_i)_{i=1}^N$ reads
\begin{equation}
\label{eqn:matrix-Bernstein}
    \P\left(\left\|\sum_{i=1}^N \tilde{Z}_i\right\|_{\textup{op}} > t \right) \leq 2S\exp\left( - \frac{t^2/2}{\sigma^2 + Kt/3} \right)
\end{equation}
for any $t > 0$, where $\max_i \|\tilde{Z}_i\|_{\textup{op}} \leq K$ almost surely and $\|\sum_{i=1}^N \E[\tilde{Z}_i^2] \|_{\textup{op}} \leq \sigma^2$. Thus we need to give deterministic upper bounds for $\|\tilde{Z}_i\|$ and $\|\sum_{i=1}^N \E[\tilde{Z}_i^2]\|$. First, we bound their norm as
\[
    \left\|\tilde{Z}_i\right\| \le d^{-t+\delta} + \left\|\E\left[Z_i \mathds{1}(\|Z_i\| \le d^{-t+\delta})\right]\right\|,
\]
and since $Z_i$ is positive semidefinite we have
\[
    \left\|\E\left[Z_i \mathds{1}(\|Z_i\| \le d^{-t+\delta})\right]\right\| = \max_{\norm{v} = 1} \E \left[v^\ast Z_i v \mathds{1}(\|Z_i\| \le d^{-t+\delta})\right]\le \max_{\norm{v} = 1} \E \left[v^\ast Z_i v\right]= \left\|\E Z_i\right\|.
\]
Notice that $Z_i$ has the form $Z_i = \frac{N}{d^{k_2}} vv^\ast$ for a vector $v \in \R^S$ with uncorrelated coordinates $v_\nu = U_{\nu i}$; thus $\E Z_i = \frac{1}{N} \cdot \frac{N}{d^{k_2}}\Id = \frac{1}{d^{k_2}}\Id$. In the following we allow $C$ to change from line to line.  Since $t \le k_2$ and $\delta$ is small, we thus have
\[
    \|\widetilde{Z}_i\| \leq Cd^{-t+\delta}.
\]
Next, we have
\begin{align*}
    \|\E[\tilde{Z}_i^2]\| &= \|\E[Z_i^2 \mathds{1}(\|Z_i\| \le d^{-t+\delta})] - (\E[Z_i \mathds{1}(\|Z_i\| \le d^{-t+\delta})])^2\| \\
    &\leq \|\E[Z_i^2 \mathds{1}(\|Z_i\| \le d^{-t+\delta})]\| + \|\E[Z_i \mathds{1}(\|Z_i\| \le d^{-t+\delta})])\|^2.
\end{align*}
The second term on the right-hand side has been studied above. We only need to provide an estimate for the first term. Since $Z_i = \frac{N}{d^{k_2}} vv^\ast$, we also have $Z_i^2 = Z_i\|Z_i\|$; using this identity, we find
\begin{align*}
    \|\E[Z_i^2 \mathds{1}(\|Z_i\| \le d^{-t+\delta})]\| &= \max_{\norm{v} = 1} \E \left[v^\ast Z_i v \|Z_i\| \mathds{1}(\|Z_i\| \le d^{-t+\delta})\right] \leq d^{-t+\delta} \max_{\norm{v} = 1} \E \left[ v^\ast Z_i v \right] \\
    &= d^{-t+\delta} \|\E Z_i\| \leq Cd^{-t-k_2+\delta}.
\end{align*}
Since the $\tilde{Z}_i$ are i.i.d. over $i$, we thus have
\[
    \left\| \sum_{i=1}^N \E[\tilde{Z}_i^2]\right\| = N\left\|\E[\tilde{Z}_i^2]\right\| \leq Cd^\delta (Nd^{-t-k_2} + Nd^{-2k_2}) \leq Cd^\delta d^{\ell-t-k_2} \leq Cd^\delta d^{2\max(-t,\ell-k_2)} \leq Cd^\delta \Phi^2,
\]
since $t < k_2$ and since $a+b \leq 2\max(a,b)$ for real $a, b$.

Now we can plug these estimates into the matrix Bernstein inequality \eqref{eqn:matrix-Bernstein}, choosing $t = d^{\epsilon/2} \Phi$, $K = Cd^{-t+\delta}$, and $\sigma^2 = Cd^\delta \Phi^2$ to obtain
\[
    \P\left(\left\|\sum_{i=1}^N \tilde{Z}_i\right\|_{\textup{op}} > d^{\epsilon/2} \Phi \right) \leq 2S \exp\left( -\frac{d^{\epsilon} \Phi^2/2}{Cd^\delta \Phi^2 + Cd^{-t+\delta}d^{\epsilon/2}\Phi} \right).
\]
Since $\Phi = d^{\max\{-t,\ell-k_2\}}$, we have $Cd^\delta \Phi^2 + Cd^{-t+\delta}d^{\epsilon/2} \Phi \leq C\Phi^2d^{\delta+\epsilon/2}$. Thus the argument of the exponential is upper-bounded by $-Cd^{\epsilon/2-\delta}$; since $S$ grows polynomially in $d$ and $\delta < \epsilon/2$, this implies 
\begin{equation}
\label{eqn:tildeZ-sum-bound}
    \P\left(\left\|\sum_{i=1}^N \tilde{Z}_i\right\|_{\textup{op}} > d^{\epsilon/2} \Phi \right) \leq N^{-D-1}
\end{equation}
for sufficiently large $N$.

To upgrade this to an estimate on $\|\sum_{i=1}^N Z_i\|_{\textup{op}}$, we introduce the event 
\begin{equation}
    \tilde{\Omega}(N, \epsilon, D) = \Omega(N, \epsilon, D) \cap \left\{\left\| \sum_{i=1}^N \tilde{Z}_i \right\|_{\textup{op}} \le d^{\epsilon/2} \Phi\right\}.
\end{equation}
Combining \eqref{eqn:tildeZ-sum-bound} with \eqref{eq:Omega-prob-bound} gives us
\[
    \P(\tilde{\Omega}(N,\epsilon,D)^c) \leq N^{-D}
\]
for large enough $N$. Furthermore, on this event, we have $\tilde{Z}_i = Z_i - \E\left[Z_i \mathds{1}(\|Z_i\| \le d^{-t+\delta})\right]$, so that
\begin{align*}
    \left\|\sum_{i=1}^N Z_i\right\|_{\textup{op}} &= \left\|\sum_{i=1}^N \tilde{Z}_i + \sum_{i=1}^N \E\left[Z_i \mathds{1}(\|Z_i\| \le d^{-t+\delta})\right]\right\|_{\textup{op}} \\
    &\le \left\|\sum_{i=1}^N \tilde{Z}_i\right\|_{\textup{op}} + \sum_{i=1}^N \left\|\E\left[Z_i \mathds{1}(\|Z_i\| \le d^{-t+\delta})\right]\right\|_{\textup{op}} \leq d^{\epsilon/2} \Phi + CNd^{-k_2} \leq d^\epsilon\Phi.
\end{align*}
This shows that
\[
    \P\left(\left\| \sum_{i=1}^N Z_i \right\|_{\textup{op}} > d^\epsilon \Phi\right) \leq \P(\tilde{\Omega}(N,\epsilon,D)^c) \leq N^{-D},
\]
which finishes the proof. 
\end{proof}


\section{Basic tools: Preliminary bounds}
\label{sec:preliminary_bounds}

The goal of this section is to prove several preliminary bounds on various quantities that we will use later. All of the estimates used outside of this section are summarized in the statements of Lemmas \ref{lem:m_order_one} and \ref{lem:streamlined_preliminary}.

\begin{lemma}
\label{lem:b_column_norm}
Let $B_j$ be the $j$th column of $B$ with the $(j,j)$th entry (which is zero) removed. Then 
\[
    \max_{j=1}^N \|B_j\| \prec 1.
\]
\end{lemma}
\begin{proof}
Since the distribution of $\|B_j\|$ does not depend on $j$, it suffices to prove $\|B_j\| \prec 1$. We also prefer to split
\[
    (U^\ast T U)_{ij} = \sum_\mu U_{\mu i} T_\mu U_{\mu j} = \sum_{k=\lceil \ell \rceil}^L \sum_{\mu \in \mathbf{M}_k} U_{\mu i} T_\mu U_{\mu j} \eqdef \sum_{k=\lceil \ell \rceil}^L V^{(k)}_{ij},
\]
which suggests that we decompose $B_j$ into a sum of vectors 
\[
    B_j = \sum_{k = \lceil \ell \rceil}^L V_j^{(k)},
\]
where $V_j^{(k)}$ is the vector whose $i$th entry is $V_{ij}^{(k)}$. Then
\[
    \|B_j\| \le \sum_{k = \lceil\ell\rceil}^L \|V_j^{(k)}\|.
\]
Since this sum has a constant number of terms, it suffices to show that $\|V_j^{(k)}\| \prec 1$ for any $k \in \llbracket\lceil\ell\rceil, L\rrbracket$. Now we compute high moments of $\|V_j^{(k)}\|$. For $p\in \N$,
\begin{align*}
    \E\left[\|V_j^{(k)}\|^{2p}\right] &= \sum_{i_1,\ldots,i_{p}}^{(j)} \E\left[ \prod_{a=1}^{p} \left(V_{i_aj}^{(k)}\right)^2 \right] \\
    &= T_k^{2p} \sum_{i_1,\ldots,i_{p}}^{(j)} \sum_{\mu_1,\nu_1,\ldots, \mu_{p}, \nu_{p} \in \mathbf{M}_k} \E\left[\prod_{a=1}^p U_{\mu_a i_a} U_{\nu_a i_a}\right] \E\left[\prod_{a=1}^p U_{\mu_a j} U_{\nu_a j}\right]
\end{align*}
where the last expectations factor since all the $i_a$'s are distinct from $j$. The expectations vanish unless the $X$'s pair, so $\sum_{\mu_1, \nu_1, \ldots, \mu_p \nu_p \in \mathbf{M}_{k}}$ has order $d^{kp}$ nonzero terms instead of the naive $d^{2kp}$; each product of expectations contributes order $N^{-2p}$; we have $T_k^{2p} = C_{2p} N^p d^{-kp}$, and $\sum_{i_1,\ldots,i_p}^{(j)}$ contributes $N^p$, so
\[
    \E[\|V_j^{(k)}\|^{2p}] \leq C_{2p},
\]
which concludes the proof.
\end{proof}

\begin{lemma}
\label{lem:m_order_one}
For any fixed $\tau > 0$ and any fixed $z \in \mathbf{D}_\tau$, we have
\begin{align}
    \min_{j=1}^N \abs{G_{jj}(z)} \succ 1, \label{eqn:min_gii} \\
    \min_{j=1}^N \im(G_{jj}(z)) \succ 1, \label{eqn:min_im_gii}
\end{align}
and therefore
\begin{align}
    1 \prec \abs{s(z)} \prec 1, \label{eqn:absmorderone} \\
    1 \prec \im(s(z)) \prec 1. \label{eqn:immorderone}
\end{align}
\end{lemma}
\begin{proof}
Since $G_{jj}$ is a diagonal element of the resolvent of $B$, which has $B_{jj} = 0$, the Schur complement formula gives us
\[
    G_{jj} = \frac{1}{-z-\ip{B_j,(B^{(j)}-z)^{-1}B_j}}
\]
where $B_j$ is the $j$th column of $B$ except for the $(j,j)$ element, and $B^{(j)}$ is the corresponding minor. Since $\abs{z}$ is order one, Lemma \ref{lem:b_column_norm} gives
\[
    \abs{\ip{B_j,(B^{(j)}-z)^{-1}B_j}} \leq \|(B^{(j)}-z)^{-1}\|_{\textup{op}} \|B_j\|^2 \leq \frac{1}{\eta} \|B_j\|^2 \prec 1
\]
which proves \eqref{eqn:min_gii}. For \eqref{eqn:min_im_gii}, we note that, if $(u_i)_{i=1}^{N-1}$ is a (real) orthonormal eigenbasis for $B^{(j)}$ with corresponding eigenvalues $(\lambda_i)_{i=1}^{N-1}$, then 
\[
    \im\left(\ip{B_j,(B^{(j)}-z)^{-1}B_j}\right) = \im\left( \sum_{i=1}^{N-1} \frac{1}{\lambda_i-z} \ip{B_j,u_i}^2 \right) = \sum_{i=1}^{N-1} \frac{\eta}{\abs{\lambda_i-z}^2} \ip{B_j,u_i}^2 \geq 0,
\]
so that
\begin{align*}
    \im G_{jj} &= \frac{\im(\overline{-z-\ip{B_j,(B^{(j)}-z)^{-1}B_j}})}{\abs{-z-\ip{B_j,(B^{(j)}-z)^{-1}B_j}}^2} = \frac{\im(z)+\im(\ip{B_j,(B^{(j)}-z)^{-1}B_j})}{\abs{-z-\ip{B_j,(B^{(j)}-z)^{-1}B_j}}^2} \\
    &\geq \frac{\im(z)}{\abs{-z-\ip{B_j,(B^{(j)}-z)^{-1}B_j}}^2} = \im(z)\abs{G_{jj}}^2.
\end{align*}
Now we put a minimum over $j$, apply \eqref{eqn:min_gii}, and use $\im(z) \succ 1$ to obtain \eqref{eqn:min_im_gii}.

The lower bound in \eqref{eqn:immorderone} follows immediately from \eqref{eqn:min_im_gii}. On the event $\{\im(s(z)) > 0\}$, we have $\abs{s(z)} \geq \im(s(z))$; this implies $1 \prec \abs{s(z)}$. Similarly, the upper bound in \eqref{eqn:immorderone} follows from the upper bound in \eqref{eqn:absmorderone}, which is immediate since $\abs{s(z)} \leq \frac{1}{\eta}$ deterministically as the trace of a resolvent.
\end{proof}

\begin{lemma}
\label{lem:streamlined_preliminary}
For any fixed $\tau > 0$ and any fixed $z \in \mathbf{D}_\tau$, we have
\begin{align}
    \max_{k=\lceil\ell\rceil}^L \max_{\mu \in \mathbf{M}_k} \frac{d^k}{N} \abs{(G_M(z)+T)_{\mu\mu}} &\prec 1, \label{eqn:gm+t} \\
    \max_{k=\lceil\ell\rceil}^L \max_{i=1}^N \max_{\mu \in \mathbf{M}_k} d^k \abs{(G_M^{(i)}(z)-G_M(z))_{\mu\mu}} &\prec 1, \label{eqn:gi-g} \\
    \max_{k=\lceil\ell\rceil}^L \max_{i=1}^N \max_{\mu \in \mathbf{M}_k} \frac{d^k}{N} \abs{(G^{(i)}_M(z)+T)_{\mu\mu}} &\prec 1. \label{eqn:gmi+t}
\end{align}
\end{lemma}

In order to prove this lemma, we need the control parameters
\begin{align*}
    \Lambda_c(z) &= \max_{k=\lceil\ell\rceil}^L \max_{\nu\in \mathbf{M}_{k}} \max_{j = 1}^N \frac{d^{k/2}}{\sqrt{N}} \left|(G_M(z))_{\nu j}\right|, \\
    \Lambda_d(z) &= \max_{k = \lceil\ell\rceil}^L \max_{\mu \in \mathbf{M}_k} \frac{d^k}{N} \abs{(G_M(z)+T)_{\mu \mu}},
\end{align*}
as well as the following two lemmas:
\begin{lemma}
\label{lem:removing-i}
For any fixed $\tau > 0$ and any fixed $z \in \mathbf{D}_\tau$, we have
\begin{equation}
\label{eqn:gimunu_gmunu}
    \max_{k=\lceil\ell\rceil}^L \max_{i=1}^N \max_{\mu \in \mathbf{M}_k} d^k \abs{(G_M^{(i)}-G_M)_{\mu\mu}} \prec N\Lambda_c(z)^2.
\end{equation}
\end{lemma}

\begin{lemma}
\label{lem:control_parameter_relationships}
For any fixed $\tau > 0$ and any fixed $z \in \mathbf{D}_\tau$, we have
\begin{align}
    \Lambda_c(z) \prec \sqrt{\frac{\Lambda_d(z)+\Lambda_c(z)^2+1}{N}}, \label{eqn:lambda_c} \\
    \Lambda_d(z) \prec \sqrt{\Lambda_d(z)+\Lambda_c(z)^2+1}. \label{eqn:lambda_e}
\end{align}
\end{lemma}

From here the proof is straightforward:
\begin{proof}[Proof of Lemma \ref{lem:streamlined_preliminary}, modulo Lemmas \ref{lem:removing-i} and \ref{lem:control_parameter_relationships}]
From the definition of stochastic domination, it is an elementary exercise to verify, for $X_N$ and $Y_N$ positive and real, that $X_N \prec \frac{X_N}{N} + Y_N$ implies $X_N \prec Y_N$; thus one upgrades \eqref{eqn:lambda_c} into
\begin{equation}
\label{eqn:lambda_c_simplified}
    \Lambda_c(z) \prec \sqrt{\frac{\Lambda_d(z)+1}{N}}.
\end{equation}
Plugging this into \eqref{eqn:lambda_e}, one obtains similarly
\[
    \Lambda_d(z) \prec \sqrt{\Lambda_d(z) + \frac{\Lambda_d(z)+1}{N} + 1} \prec \sqrt{\Lambda_d(z)+1},
\]
from which it is another elementary exercise from the definition of stochastic domination to conclude $\Lambda_d(z) \prec 1$, which is exactly \eqref{eqn:gm+t}. Plugging this back into \eqref{eqn:lambda_c_simplified}, one obtains 
\[
    \Lambda_c(z) \prec \frac{1}{\sqrt{N}}.
\]
Combining this with Lemma \ref{lem:removing-i} yields \eqref{eqn:gi-g}. Finally, \eqref{eqn:gmi+t} is immediate from combining \eqref{eqn:gm+t} and \eqref{eqn:gi-g}.
\end{proof}

The proofs of Lemmas \ref{lem:removing-i} and \ref{lem:control_parameter_relationships} regularly use the following short lemma. The proof is a short exercise in the definition of stochastic domination, so we omit it.
\begin{lemma}
\label{lem:stochastic_domination_combining}
Suppose we have random variables $X_{N,k,i,\mu}$ depending on $k \in \llbracket \lceil\ell\rceil, L \rrbracket$, on $i \in \llbracket 1, N \rrbracket$, and on $\mu \in \mathbf{M}_k$, such that for some $Y_N$ we have
\[
    \abs{X_{N,k,i,\mu}} \prec Y_N
\]
for each $k$, $i$, and $\mu$. If, for each fixed $k$, the distribution of $X_{N,k,i,\mu}$ depends neither on $i$ nor on $\mu \in \mathbf{M}_k$, then 
\[
    \max_{k=\lceil\ell\rceil}^L \max_{i=1}^N \max_{\mu \in \mathbf{M}_k} \abs{X_{N,k,i,\mu}} \prec Y_N.
\]
\end{lemma}

\begin{proof}[Proof of Lemma \ref{lem:removing-i}]
Combining the resolvent identity $G^{(i)}_{\mu\mu} = G_{\mu\mu} - \frac{G_{\mu i}G_{i \mu}}{G_{ii}}$ from \eqref{eq:res-id-off-diagonal} with $\min_{j=1}^N \abs{G_{jj}(z)} \succ 1$ from \eqref{eqn:min_gii}, we obtain 
\[
    d^k \abs{(G_M^{(i)}-G_M)_{\mu\mu}} \prec N\Lambda_c(z)^2
\]
for each $k$, $i$, and $\mu$. Applying Lemma \ref{lem:stochastic_domination_combining} completes the proof.
\end{proof}

\begin{proof}[Proof of Lemma \ref{lem:control_parameter_relationships}]
First, we claim that for each $\mu$ and each $i$ we have
\begin{equation}
\label{eqn:linear-large-deviations}
    \abs{\sum_{\nu \in \mathbf{M}} U_{\nu i} G^{(i)}_{\nu \mu}} \prec \frac{1}{\sqrt{N}} \sqrt{\sum_{\nu \in \mathbf{M}} \abs{G^{(i)}_{\nu \mu}}^2}.
\end{equation}
Indeed, by summing over $k \in \llbracket \lceil\ell\rceil, L \rrbracket$, to show \eqref{eqn:linear-large-deviations} it suffices to show
\begin{equation}
\label{eqn:linear-large-deviations-k}
    \abs{\sum_{\nu \in \mathbf{M}_k} U_{\nu i} G^{(i)}_{\nu \mu}} \prec \frac{1}{\sqrt{N}} \sqrt{\sum_{\nu \in \mathbf{M}_k} \abs{G^{(i)}_{\nu \mu}}^2}
\end{equation}
for each $k$; we write this in terms of the $X$'s, cancelling the factor $1/\sqrt{N}$, as
\[
    \abs{\sum_{a_1 < \cdots < a_k}^{d} X_{a_1i} \cdots X_{a_ki} G^{(i)}_{\nu\mu}} \prec \sqrt{ \sum_{\nu \in \mathbf{M}_k} \abs{G^{(i)}_{\nu\mu}}^2},
\]
where $\nu = (a_1, \ldots, a_k)$. But estimates of this form are essentially standard, and are typically called ``large deviations bounds'' in the local-law literature. Simple ones take the form $\sum_{a_1 \neq a_2} X_{a_1} X_{a_2} b_{a_1a_2} \prec (\sum_{a_1 \neq a_2} \abs{b_{a_1a_2}}^2)^{1/2}$ (see, e.g., \cite[Theorem C.1]{ErdKnoYauYin2013local}, which is based on \cite[Lemmas B.2--B.4]{ErdKnoYauYin2013delocalization}), where the $X_a$'s are i.i.d. centered random variables with unit variance, the $b_{a_1a_2}$ are deterministic, and the result is crucially uniform in $b_{a_1a_2}$. In our case, the sum is over $k$ indices rather than two, but this generalization is routine, as already noted in \cite[Theorem C.1]{ErdKnoYauYin2013local} and \cite[Lemmas B.2--B.4]{ErdKnoYauYin2013delocalization}. Furthermore, in our case the role of $b_{a_1a_2}$ is played instead by $G^{(i)}_{\mu \nu}$. These resolvent entries are not deterministic, but they are independent of $X_{i}$, so we can condition on them; since the result is uniform in deterministic $b_{a_1a_2}$, we can safely integrate over the randomness in $G^{(i)}$, obtaining \eqref{eqn:linear-large-deviations-k} and thus \eqref{eqn:linear-large-deviations}.

Now fix $k$ and $\mu \in \mathbf{M}_k$. On the one hand, if we start with the resolvent identity $G_{i\mu} = - G_{ii} \sum_{\nu \in \mathbf{M}} U_{\nu i} G_{\nu\mu}^{(i)}$ from \eqref{eq:res-id-G-mu-i} and use the estimates $\max_i \abs{G_{ii}} \leq 1/\eta \prec 1$, which is trivial since the $G_{ii}$ are the diagonal elements of a resolvent, and \eqref{eqn:linear-large-deviations}, we obtain
\begin{equation}
\label{eqn:off-diag}
    \max_{i=1}^N \abs{G_{i\mu}} \prec \max_{i=1}^N \frac{1}{\sqrt{N}} \sqrt{ \sum_{\nu \in \mathbf{M}} \abs{G_{\nu\mu}^{(i)}}^2}.
\end{equation}
On the other hand, if we start with the resolvent identity $(G_M+T)_{\mu\mu} = -T_\mu \sum_{i=1}^N G_{ii} \sum_{\nu \in \mathbf{M}} U_{\mu i} U_{\nu i} G_{\nu \mu}^{(i)}$, from \eqref{eq:res-id-G-mu-nu}, and use $\abs{U_{\mu j}} \prec 1/\sqrt{N}$ as well as \eqref{eqn:linear-large-deviations} (to which we can add $\max_{i=1}^N$ on both sides by Lemma \ref{lem:stochastic_domination_combining}), we obtain
\begin{equation}
\label{eqn:on-diag}
    \abs{(G_M+T)_{\mu\mu}} \prec \sqrt{\frac{N}{d^k}} \max_{i=1}^N \sqrt{ \sum_{\nu \in \mathbf{M}} \abs{G^{(i)}_{\nu\mu}}^2}.
\end{equation}
Assume momentarily the estimate
\begin{equation}
\label{eqn:using-full-Ward}
    \frac{d^k}{N} \sum_{\nu \in \mathbf{M}} \abs{G_{\nu\mu}^{(i)}}^2 \prec \Lambda_e(z) + \Lambda_c(z)^2 + 1.
\end{equation}
Lemma \ref{lem:stochastic_domination_combining} allows us to add $\max_{k=\lceil \ell \rceil}^L \max_{i=1}^N \max_{\mu \in \mathbf{M}_k}$ to the left-hand side for free. Combining this with \eqref{eqn:off-diag} yields \eqref{eqn:lambda_c}; combining it with \eqref{eqn:on-diag} instead yields \eqref{eqn:lambda_e}. 

Thus it remains only to prove \eqref{eqn:using-full-Ward}. We do this using the full Ward inequality, Lemma \ref{lem:full-Ward}: Since $T$ is diagonal and $\abs{a+b}^2 \leq 2\abs{a}^2 + 2\abs{b}^2 \prec \abs{a}^2 + \abs{b}^2$, we find
\[
    \sum_{\nu \in \mathbf{M}} \abs{G_{\nu\mu}^{(i)}}^2 = \sum_{\nu}^{(\mu)} \abs{(G_M^{(i)}+T)_{\nu\mu}}^2 + \abs{(G_M^{(i)}+T)_{\mu\mu} -T_{\mu}}^2 \prec \sum_{\nu \in \mathbf{M}} \abs{(G^{(i)}+T)_{\nu\mu}}^2 + T_\mu^2 \prec \frac{\im G_{\mu\mu}^{(i)}}{\eta} + \frac{N}{d^k}.
\]
Since $T$ is real, we have
\[
    \frac{\im G_{\mu\mu}^{(i)}}{\eta} \prec \im G_{\mu\mu}^{(i)} = \im (G_M^{(i)}-G_M)_{\mu\mu} + \im (G_M+T)_{\mu\mu} \prec \frac{N}{d^k}(\Lambda_e(z) + \Lambda_c(z)^2),
\]
which completes the proof of \eqref{eqn:using-full-Ward}, and thus of the lemma.
\end{proof}


\section{Self-consistent equations I: Proof of Proposition \ref{prop:m_widetildem_self_consistent}}
\label{sec:concentration}

In this section, we prove Proposition \ref{prop:m_widetildem_self_consistent}. The bulk of the proof is Lemma \ref{lem:m_main_error_analysis}, which says roughly that, for each $i$, we have $\sum_{\mu, \nu} U_{\mu i} (G^{(i)}_M+T)_{\mu \nu} U_{\nu i} \approx \phi\widetilde{s}$. This should be thought of as a kind of concentration result: Since $\E[U_{\mu i} U_{\nu i}] = \delta_{\mu \nu}/N$ and $G_M^{(i)}+T$ is independent of the $X_i$'s, the partial expectation of $\sum_{\mu, \nu} U_{\mu i} (G^{(i)}+T)_{\mu \nu} U_{\nu i}$ over just the $X_i$'s is $\frac{1}{N} \Tr(G^{(i)}_M+T)$; and if one replaces $G^{(i)}_M$ with $G_M$ in this expression, one gets exactly $\phi\widetilde{s}$. Lemma \ref{lem:m_main_error_analysis} itself relies fundamentally on the partial Ward inequalities from Section \ref{sec:ward}.

\begin{lemma}
\label{lem:m_main_error_analysis}
For any fixed $\tau > 0$, any fixed $z \in \mathbf{D}_\tau$, and any $i \in [N]$, we have
\[
    \abs{\sum_{\mu,\nu} U_{\mu i}(G_M^{(i)}(z)+T)_{\mu\nu} U_{\nu i} - \phi\widetilde{s}(z)} \prec \frac{1}{d^{\frac{1}{2}\min(1,\ell)}}.
\]
\end{lemma}

\begin{proof}[Proof of Proposition \ref{prop:m_widetildem_self_consistent}, modulo Lemma \ref{lem:m_main_error_analysis}]
From the Schur complement formula \eqref{eq:res-id-schur} and Lemma \ref{lem:m_main_error_analysis}, we have
\begin{align*}
    G_{ii}^{-1} &= 
    -z - \sum_{\mu, \nu} U_{\mu i} \left(G_M^{(i)}+T\right)_{\mu \nu} U_{\nu i} = -z-\phi\tilde{s}(z) + \OO_{\prec}\left(\frac{1}{d^{\frac{1}{2}\min(1,\ell)}}\right).
\end{align*}
Multiplying both sides by $G_{ii}$ and using the deterministic bound $\abs{G_{ii}} \leq 1/\eta$ (so that $\abs{G_{ii}} \OO_{\prec}(d^{\frac{1}{2}\min(1,\ell)}) = \OO_{\prec}(d^{\frac{1}{2}\min(1,\ell)})$), we find
\begin{equation}
\label{eqn:gii_self_consistent}
    1 = (-z - \phi\widetilde{s}(z))G_{ii} + \OO_{\prec}\left(\frac{1}{d^{\frac{1}{2}\min(1,\ell)}}\right).
\end{equation}
The error term is uniform in $i$, since all the variables are exchangeable in $i$. Thus we can average both sides over $i$ and rearrange to obtain the result.
\end{proof}

\begin{proof}[Proof of Lemma \ref{lem:m_main_error_analysis}]
Since $\phi \widetilde{s} = \frac{1}{N} \tr(G_M + T) = \frac{1}{N} \sum_\mu (G_M+T)_{\mu\mu}$, it suffices to show the following three bounds:
\begin{align}
    \mc{E}_1 &\defeq \sum_{\mu \neq \nu} U_{\mu i}(G^{(i)}+T)_{\mu\nu} U_{\nu i} = \sum_{\mu \neq \nu} U_{\mu i}G^{(i)}_{\mu\nu} U_{\nu i} \prec d^{-\frac{1}{2}\min(1,\ell)}, \label{eqn:error1bound}\\
    \mc{E}_2 &\defeq \sum_\mu \left( U_{\mu i}^2 - \frac{1}{N} \right) (G^{(i)}+T)_{\mu\mu} \prec \frac{1}{\sqrt{d}}, \label{eqn:error2bound} \\
    \mc{E}_3 &\defeq \frac{1}{N} (\tr(G^{(i)}_M) - \tr(G_M)) \prec \frac{1}{d^\ell} \label{eqn:error3bound}.
\end{align}
Notice that the $\min$ only appears in the estimate of $\mc{E}_1$, and that the estimate on $\mc{E}_3$ is much better than needed. First we consider the $\mc{E}_1$ term, which is the most complicated. It is convenient to split the sum over all $\mu \neq \nu$ into terms which fix the length of (number of $X$'s contained in) $\mu$, fix the length of $\nu$, and fix their overlap, by defining
\[
    \mc{G}^{(k_1,k_2,s)} \defeq \ip{(\mu, \nu) : \mu \neq \nu, \mu \in \mathbf{M}_{k_1}, \nu \in \mathbf{M}_{k_2}, \ip{\mu,\nu} = s}
\]
and 
\begin{equation}
\label{eqn:e1k1k2s}
    \mc{E}_{1,k_1,k_2,s} \defeq \sum_{(\mu, \nu) \in \mc{G}^{(k_1,k_2,s)}} U_{\mu i}G^{(i)}_{\mu\nu} U_{\nu i} = \frac{\alpha_{k_1,k_2,s}}{N} \sum_{\substack{a_1, \ldots,  a_{k_1} \\ b_{s+1}, \ldots, b_{k_2}}}^{d,\ast} X_{a_1 i}^2 \ldots X_{a_s i}^2 \cdot X_{a_{s+1} i} \ldots X_{a_{k_1} i} G^{(i)}_{\mu\nu} X_{b_{s+1} i} \ldots X_{b_{k_2} i},
\end{equation}
where $\alpha_{k_1,k_2,s} = \frac{1}{s!(k_1-s)!(k_2-s)!}$ is a combinatorial factor accounting for the fact that the indices in $\mu$ and $\nu$ are not only distinct but also ordered: We can reconstruct $\mu$ and $\nu$ from $(a_1,\ldots,a_{k_1})$ and $(a_1,\ldots,a_s,b_{s+1},\ldots,b_{k_2})$ just by ordering, but given $\mu$ and $\nu$ with $s$ shared indices, there are $s!$ ways to label the shared indices as $(a_1,\ldots,a_s)$, $(k_1-s)!$ ways to label the remaining $\mu$ indices as $(a_{s+1},\ldots,a_{k_1})$, and $(k_2-s)!$ ways to label the remaining $\nu$ indices as $(b_{s+1},\ldots,b_{k_2})$. Thus $\alpha_{k_1,k_2,s} \prec 1$, which is shortly how we will absorb it. 

For fixed $k_1$ and $k_2$, the possible $s$ values range from $0$ up to $\min(k_1,k_2)$, unless $k_1 = k_2 = k$, in which case the range is $s \in \llbracket 0, k-1 \rrbracket$, since if $\ip{\mu,\nu}$ equals the common length of $\mu$ and $\nu$, then $\mu = \nu$, which is forbidden in the sum.

Notice $\mc{E}_1 =\sum_{k_1,k_2,s} \mc{E}_{1,k_1,k_2,s}$; since $(k_1,k_2,s)$ takes values in a finite set, we can estimate each $\mc{E}_{1,k_1,k_2,s}$ separately. Now for each $\{a_1,\ldots,a_s\}$ all distinct we write
\[
    S_{a_1,\ldots, a_s} = \sum_{a_{s+1},\ldots,a_{k_1},b_{s+1},\ldots,b_{k_2}=1}^{d,\ast} X_{a_{s+1}i} \cdots X_{a_{k_1}i} G_{\mu \nu}^{(i)} X_{b_{s+1}i} \cdots X_{b_{k_2}i},
\]
where for each $a_{s+1},\ldots,a_{k_1}, b_{s+1},\ldots,b_{k_2}$, the $\mu$ and $\nu$ inside the sum denote the (reordered as necessary) tuples $(a_1,\ldots,a_{k_1})$ and $(a_1,\ldots,a_s,b_{s+1},\ldots,b_{k_2})$, respectively. (Recall that our definition of $\mu$ and $\nu$ involves strict ordering, so this is unambiguous.) Define $\mc{G}^{(k_1,k_2,s)}_{a_1,\ldots,a_s}$ to be the set of all pairs $(\mu,\nu)$ indicated by this sum (i.e., the set of all pairs $(\mu,\nu) \in \mc{G}^{(k_1,k_2,s)}$ for which the specific overlapping indices are $\{a_1, \ldots, a_s\}$), and notice that these partition:
\begin{equation}
\label{eqn:gk1k2s_decomposition}
    \bigsqcup_{a_1 < \cdots < a_s} \mc{G}^{(k_1,k_2,s)}_{a_1,\ldots,a_s} = \mc{G}^{(k_1,k_2,s)}.
\end{equation}
We claim (absorbing another order-one combinatorial factor into $\prec$) that 
\begin{equation}
\label{eqn:s_large_deviations}
    \abs{S_{a_1,\ldots,a_s}}^2 \prec \sum_{\mu, \nu \in \mc{G}^{(k_1,k_2,s)}_{a_1,\ldots,a_s}} \abs{G_{\mu \nu}^{(i)}}^2,
\end{equation}
uniformly in $\{a_1,\ldots,a_s\}$. Indeed, this is simply another standard large-deviations bound, as discussed in the proof of \eqref{eqn:linear-large-deviations} above.

At the same time, since we assumed all finite moments, one can easily see $X_{ai}^4 \prec 1$, uniformly in $a$ (and $i$); hence $X_{a_1i}^4 \cdots X_{a_si}^4 \prec 1$, uniformly in $\{a_1, \ldots, a_s\}$; hence
\[
    X_{a_1i}^4 \cdots X_{a_si}^4 \abs{S_{a_1,\ldots,a_s}}^2 \prec \sum_{\mu, \nu \in \mc{G}^{(k_1,k_2,s)}_{a_1,\ldots,a_s}} \abs{G_{\mu \nu}^{(i)}}^2,
\]
uniformly in $\{a_1, \ldots, a_s\}$ (we absorb another combinatorial factor into $\prec$). Combining these with \eqref{eqn:gk1k2s_decomposition}, we find
\begin{align*}
    \sum_{a_1 < \cdots < a_s} X_{a_1i}^4 \cdots X_{a_si}^4 \abs{S_{a_1,\ldots,a_s}}^2 &\prec \sum_{a_1 < \cdots < a_s} \sum_{\mu, \nu \in \mc{G}^{(k_1,k_2,s)}_{a_1,\ldots,a_s}} \abs{G_{\mu \nu}^{(i)}}^2 = \sum_{(\mu, \nu) \in \mc{G}^{(k_1,k_2,s)}} \abs{G_{\mu \nu}^{(i)}}^2, \\
\end{align*}
Now we apply Cauchy-Schwarz to \eqref{eqn:e1k1k2s} and use these estimates to obtain
\begin{align*}
    \abs{\mc{E}_{1,k_1,k_2,s}} &= \frac{\alpha_{k_1,k_2,s}s!}{N} \abs{\sum_{a_1 < \cdots < a_s} 1 \cdot (X_{a_1i}^2 \cdots X_{a_si}^2 S_{a_1,\ldots,a_s})} \\
    &\leq \frac{d^{s/2}\alpha_{k_1,k_2,s}}{N} \left( \sum_{a_1 < \cdots < a_s} X_{a_1i}^4 \cdots X_{a_si}^4 \abs{S_{a_1,\ldots,a_s}}^2 \right)^{1/2} \\
    &\prec \frac{d^{s/2}}{N} \left( \sum_{(\mu, \nu) \in \mc{G}^{(k_1,k_2,s)}} \abs{G_{\mu \nu}^{(i)}}^2 \right)^{1/2}.
\end{align*}

Suppose without loss of generality that $k_1 \leq k_2$. Then we apply the partial Ward inequality, Lemma \ref{lem:partial-Ward}, to each fixed $\mu$ (recall that $T$ is diagonal, so that $(G^{(i)}+T)_{\mu\nu} = G^{(i)}_{\mu\nu}$ whenever $\mu \neq \nu$); since the result is uniform in $\mu$, we can also sum over $\mu$ in the sense of stochastic domination to find
\[
    \abs{\mc{E}_{1,k_1,k_2,s}} \prec \frac{d^{s/2}}{N} \left(\sum_{\mu \in \mathbf{M}_{k_1}} \frac{\im G^{(i)}_{\mu \mu}}{\eta} d^{\max(-s,\ell-k_2)} \right)^{1/2}.
\]
Recall that $\eta$ is order one, and that \eqref{eqn:gmi+t} yields 
\[
    \im G^{(i)}_{\mu\mu} = \im (G^{(i)}+T)_{\mu\mu} \leq \abs{(G^{(i)}+T)_{\mu \mu}} \prec \frac{N}{d^{k_1}}.
\]
Notice also that $\abs{\mathbf{M}_{k_1}} \leq d^{k_1}$, and that $s \leq k_2-1$ (indeed, either $k_1 < k_2$, in which case $s \leq k_1 < k_2$, or $k_1 = k_2$, in which case $s \leq k_2 - 1$ because of the $\mu \neq \nu$ restriction explained above); thus
\begin{equation}
\label{eqn:error_k1k2s}
    \abs{\mc{E}_{1,k_1,k_2,s}} \prec \frac{d^{s/2}}{N} (Nd^{\max(-s,\ell-k_2)})^{1/2} = \frac{1}{\sqrt{N}} d^{\frac{1}{2}\max(0,s+\ell-k_2)} \leq \frac{1}{\sqrt{N}} d^{\frac{1}{2} \max(0,\ell-1)} \prec d^{\frac{1}{2}\max(-\ell,-1)},
\end{equation}
which finishes the proof that $\abs{\mc{E}_1} \prec d^{-\frac{1}{2}\min(1,\ell)}$.

Next we estimate $\mc{E}_2$. Again it is convenient to split the sum over all $\mu$'s into finitely many partial sums
\[
    \mc{E}_{2,k} = \sum_{\mu \in \mathbf{M}_k} \left(U_{\mu i}^2 - \frac{1}{N} \right)(G^{(i)}+T)_{\mu\mu}
\]
and show 
\begin{equation}
\label{eqn:error_2k}
    \abs{\mc{E}_{2,k}} \prec \frac{1}{\sqrt{d}}
\end{equation}
for each $k$. Fix $\epsilon$ and $D$; since $G^{(i)}$ is independent of $U_i$, we have
\begin{align*}
    \P(\abs{\mc{E}_{2,k}} > d^{\epsilon-1/2}) &\leq \E_{G^{(i)}}\left[ \E_{U_i}\left[\mathbf{1}\{\abs{\mc{E}_{2,k}} > d^{\epsilon-1/2}\} \right] \mathbf{1}\left\{\max_{\mu \in \mathbf{M}_k} \abs{(G^{(i)}+T)_{\mu\mu}} \leq \frac{N}{d^k} d^{\epsilon/2} \right\} \right] \\
    &\quad + \P\left( \max_{\mu \in \mathbf{M}_k} \abs{(G^{(i)}+T)_{\mu\mu}} \geq \frac{N}{d^k} d^{\epsilon/2} \right).
\end{align*}
Applying Lemma \ref{lem:u^2-1/N} below and \eqref{eqn:gmi+t} 
to the first and second terms on the right-hand side, respectively, we find that each is at most some $C_{\epsilon,D} d^{-D}$. This gives $\abs{\mc{E}_2} \prec 1/\sqrt{d}$ as claimed.

Finally we estimate $\mc{E}_3$, again splitting $\mc{E}_3 = \sum_{k=\lceil\ell\rceil}^L \mc{E}_{3,k}$ with
\[
    \mc{E}_{3,k} = \frac{1}{N} \sum_{\mu \in \mathbf{M}_k} (G^{(i)}_{\mu \mu} - G_{\mu\mu}).
\]
From \eqref{eqn:gi-g} we have
\begin{equation}
\label{eqn:error_3k}
    \abs{\mc{E}_{3,k}} \leq \frac{d^k}{N} \max_{\mu \in \mathbf{M}_k} \abs{G^{(i)}_{\mu \mu} - G_{\mu\mu}} \prec \frac{1}{N},
\end{equation}
which completes the proof.
\end{proof}

\begin{lemma}
\label{lem:u^2-1/N}
Fix $k$, and fix some deterministic sequence $(b_\mu = b^{(d)}_\mu)_{\mu \in \mathbf{M}_{k}}$ of complex numbers with
\begin{equation}
\label{eqn:maxb}
    \sup_{\mu \in \mathbf{M}_{k}} \abs{b_\mu} \leq \alpha_d
\end{equation}
for some sequence $(\alpha_d)_{d=1}^\infty$ (recall that the set $\mathbf{M}_{k}$ depends on $d$). Then
\[
    \abs{\sum_{\mu \in \mathbf{M}_{k}} \left( U_{\mu i}^2 - \frac{1}{N} \right) b_\mu} \prec \frac{d^{k}\alpha_d}{N\sqrt{d}}
\]
uniformly in $(b_\mu)$ subject to \eqref{eqn:maxb}. 
\end{lemma}

\begin{proof}[Proof of Lemma \ref{lem:u^2-1/N}]
The proof goes by high moments: For $p \in \N$, we have
\begin{align*}
    \E\left[ \abs{\sum_{\mu \in \mathbf{M}_{k}} \left( U_{\mu i}^2 - \frac{1}{N} \right) b_\mu}^{2p} \right] &= \sum_{\mu_1,\mu'_1,\mu_2,\mu'_2,\ldots,\mu_p,\mu'_p \in \mathbf{M}_{k}} \E\left[ \prod_{j=1}^p \left(U_{\mu_ji}^2 - \frac{1}{N} \right) \left( U_{\mu'_j i}^2 - \frac{1}{N} \right) \right] \prod_{j=1}^p b_{\mu_j} \overline{b_{\mu'_j}} \\
    &\leq (\alpha_d)^{2p} \sum_{\mu_1,\mu'_1,\mu_2,\mu'_2,\ldots,\mu_p,\mu'_p \in \mathbf{M}_{k}} \abs{\E\left[ \prod_{j=1}^p \left(U_{\mu_ji}^2 - \frac{1}{N} \right) \left( U_{\mu'_j i}^2 - \frac{1}{N} \right) \right]} \\
    &= (\alpha_d)^{2p} \sum_{\mu_1,\ldots,\mu_{2p} \in \mathbf{M}_k} \abs{\E\left[ \prod_{j=1}^{2p} \left(U_{\mu_ji}^2 - \frac{1}{N} \right)\right]}
\end{align*}
where the last equality is just a convenient relabeling (after we stop distinguishing the complex conjugates between $b_{\mu_j}$ and $\overline{b_{\mu'_j}}$, we no longer need to pair the terms $\mu_j$ and $\mu'_j$).

Given a tuple of tuples $(\mu_1,\ldots,\mu_{2p}) \in (\mathbf{M}_{k})^{2p}$, we say that some tuple $\mu_j$ is \emph{isolated} if $\max_{m \neq j} \ip{\mu_j,\mu_m} = 0$, i.e., if $\mu_j$ has its own set of $X$'s, none of which appears in any other tuple $\mu_m$. Consider the set 
\[
    \mc{G}^{2p} = \{(\mu_1,\ldots,\mu_{2p}) \in (\mathbf{M}_{k})^{2p} : \text{No $\mu_j$ is isolated}\}.
\]
On the complement of this set, at least one tuple $\mu_j$ is isolated in this sense, meaning that at least one $(U_{\mu_ji}^2-\frac{1}{N})$ is independent of everything else; since these variables have mean zero, such expectations vanish, meaning that
\begin{equation}
\label{eqn:restrict_to_g2p}
    \sum_{\mu_1,\ldots,\mu_{2p} \in \mathbf{M}_{k}} \abs{\E\left[ \prod_{j=1}^{2p} \left(U_{\mu_ji}^2 - \frac{1}{N} \right) \right]} = \sum_{(\mu_1,\ldots,\mu_{2p}) \in \mc{G}^{2p}} \abs{\E\left[ \prod_{j=1}^{2p} \left(U_{\mu_ji}^2 - \frac{1}{N} \right) \right]}.
\end{equation}
Furthermore, we claim that 
\begin{equation}
\label{eqn:g2p_size}
    \abs{\mc{G}^{2p}} \leq C_{2p} d^{2pk-p}.
\end{equation}
Indeed, the total number of tuples $(\mu_1,\ldots,\mu_{2p}) \in (\mathbf{M}_k)^{2p}$ is at most $d^{2pk}$, because each of the $2p$ $\mu_j$'s includes $k$ $X$'s. To ensure that none is isolated, while using as many $X$'s as possible, each $\mu_j$ should use $k-1$ of its own $X$'s and have a final $X$ which it shares with exactly one other tuple $\mu_{j'}$; this pairing subtracts $p$ off the naive count, while adding a combinatorial factor $C_{2p}$ tracking which $\mu_j$'s pair. 

At the same time, we claim that for each $p$ there exists $C_{2p}$ with 
\begin{equation}
\label{eqn:g2p_estimate}
    \sup_{\mu_1,\ldots,\mu_{2p} \in \mathbf{M}_{k}} \abs{\E\left[ \prod_{j=1}^{2p} \left(U_{\mu_ji}^2 - \frac{1}{N} \right) \right]} \leq \frac{C_{2p}}{N^{2p}}
\end{equation}
Indeed, writing $A_j \defeq U_{\mu_j i}^2-1/N$, we can use the generalized H\"{o}lder's inequality $\abs{\E[\prod_{j=1}^{2p} A_j]} \leq \prod_{j=1}^{2p} (\E[A_j^{2p}])^{1/2p}$, then the triangle inequality: $(\E[A_j^{2p}])^{1/2p} = \|U_i^2-1/N\|_{2p} \leq \|U_i^2\|_{2p} + 1/N \leq C_{k,p}/N$. 

Combining \eqref{eqn:restrict_to_g2p}, \eqref{eqn:g2p_size}, and \eqref{eqn:g2p_estimate}, we find
\[
    \E\left[ \abs{\sum_{\mu \in \mathbf{M}_k} \left( U_{\mu i}^2 - \frac{1}{N} \right) b_\mu}^{2p} \right] \leq \frac{C_{2p}}{N^{2p}} (\alpha_d)^{2p} \abs{\mc{G}^{2p}} \leq C_{2p} \left( \frac{d^k \alpha_d}{N\sqrt{d}} \right)^{2p},
\]
which suffices. 
\end{proof}


\section{Self-consistent equations II: Proofs of Propositions \ref{prop:widetildem_self_consistent} and \ref{prop:m_self_consistent}}
\label{sec:tilde_m_self_consistent}

The goal of this section is to prove Propositions \ref{prop:widetildem_self_consistent} and \ref{prop:m_self_consistent}. The latter follows from the former fairly quickly.

\begin{lemma}
\label{lem:phi_m_tilde_order_one}
For any fixed $\tau > 0$ and any fixed $z \in \mathbf{D}_\tau$, we have
\begin{align}
    1 \prec \abs{z+\phi\widetilde{s}(z)} \prec 1, \label{eqn:abs_zphimtilde_orderone} \\
    1 \prec \im(z+\phi\widetilde{s}(z)) \prec 1. \label{eqn:im_zphimtilde_orderone}
\end{align}
\end{lemma}
\begin{proof}
The estimate \eqref{eqn:abs_zphimtilde_orderone} follows immediately from Lemma \ref{lem:m_order_one}, which shows $1 \prec \abs{s(z)} \prec 1$, and from Proposition \ref{prop:m_widetildem_self_consistent}, which shows $\abs{1+s(z)(z+\phi\widetilde{s}(z))} \prec d^{-\frac{1}{2}\min(1,\ell)}$. The upper bound of \eqref{eqn:im_zphimtilde_orderone} is immediate from that of \eqref{eqn:abs_zphimtilde_orderone}. For the lower bound, we note that the imaginary part of $\widetilde{s}$ is almost surely nonnegative: Indeed, from \eqref{eq:res-id-G+T} we have
\[
    \im\widetilde{s}(z) = \frac{1}{M} \Tr(\im(TUG_N(z)U^\ast T)) = \Tr(TU(\im G_N(z)) U^\ast T) \geq 0
\]
where we used that $G_N(z)$ is a resolvent, so that its imaginary part is positive definite, as well as the general result that $B^\ast A B = (A^{1/2} B)^\ast (A^{1/2}B)$ is positive semidefinite if $A$ is (square and) positive definite and $B$ is any (possibly rectangular) matrix. 
\end{proof}

\begin{prop}
\label{prop:main_error_term_phitildem}
For any fixed $\tau > 0$ and any fixed $z \in \mathbf{D}_\tau$, we have
\[
    \abs{\phi \widetilde{s}(z) - \frac{1}{N} \sum_\mu  \frac{T_\mu^2}{T_\mu -z - \phi \widetilde{s}(z)}} \prec \frac{1}{d^{\frac{1}{2}\min(1,\ell)}}.
\]
\end{prop}
\begin{proof}[Proof of Proposition \ref{prop:main_error_term_phitildem}]
Define
\begin{align*}
    \mathcal{E}_\mu^{(1)} &\defeq -T_{\mu} \sum_{j=1}^N \left(G_{jj} + (z + \phi\widetilde{s}(z))^{-1}\right) \sum_\nu U_{\mu j}U_{\nu j} G^{(j)}_{\nu \mu}, \\
    \mathcal{E}_\mu^{(2)} &\defeq T_{\mu} (z + \phi\widetilde{s}(z))^{-1} \sum_{j=1}^N \sum_\nu^{(\mu)} U_{\mu j}U_{\nu j} G^{(j)}_{\nu \mu}, \\
    \mathcal{E}_\mu^{(3)} &\defeq T_{\mu} (z + \phi\widetilde{s}(z))^{-1} \sum_{j=1}^N  \left(U_{\mu j}^2-\frac{1}{N}\right) G^{(j)}_{\mu\mu}, \\
    \mathcal{E}_\mu^{(4)} &\defeq \frac{1}{N} T_{\mu} (z + \phi\widetilde{s}(z))^{-1} \sum_{j=1}^N \left(G_{\mu\mu}^{(j)} - G_{\mu\mu}\right),
\end{align*}
so that, by the resolvent identity $G_{\mu\mu} + T_\mu = -T_\mu \sum_j G_{jj} \sum_\nu U_{\mu j} U_{\nu j} G^{(j)}_{\nu\mu}$ from \eqref{eq:res-id-G-mu-nu}, we have
\begin{align*}
    \mc{E}_\mu &\defeq \mc{E}^{(1)}_\mu + \mc{E}^{(2)}_\mu + \mc{E}^{(3)}_\mu + \mc{E}^{(4)}_\mu = -T_\mu \sum_j G_{jj} \sum_\nu U_{\mu j} U_{\nu j} G^{(j)}_{\nu \mu} - \frac{T_\mu G_{\mu \mu}}{z+\phi \widetilde{s}(z)} = G_{\mu \mu} + T_\mu - \frac{T_\mu G_{\mu \mu}}{z+\phi \widetilde{s}(z)} \\
    &= \left( 1- \frac{T_\mu}{z+\phi\widetilde{s}(z)} \right) (G_{\mu \mu} + T_\mu) + \frac{T_\mu^2}{z+\phi\widetilde{s}(z)} = \frac{(z+\phi\widetilde{s}(z)-T_\mu)(G_{\mu\mu}+T_\mu) + T_\mu^2}{z+\phi\widetilde{s}(z)}.
\end{align*}
Thus
\[
    G_{\mu\mu} + T_\mu - \frac{T_\mu^2}{T_\mu - z - \phi\widetilde{s}(z)} = \left( \frac{z+\phi\widetilde{s}(z)}{z+\phi\widetilde{s}(z) - T_\mu} \right) \mc{E}_\mu,
\]
so if we define
\[
    \mc{E}_k^{(a)} \defeq \sum_{\mu \in \mathbf{M}_k} \mc{E}_\mu^{(a)}, \qquad a = 1, 2, 3, 4,
\]
then 
\begin{align*}
    \phi \widetilde{s}(z) - \frac{1}{N} \sum_\mu  \frac{T_\mu^2}{T_\mu -z - \phi \widetilde{s}(z)} &= \frac{1}{N} \sum_\mu \left( G_{\mu\mu} + T_\mu - \frac{T_\mu^2}{T_\mu - z - \phi\widetilde{s}(z)} \right) = \frac{1}{N} \sum_{\mu} \left( \frac{z+\phi\widetilde{s}(z)}{z+\phi\widetilde{s}(z)-T_\mu} \mc{E}_\mu \right) \\
    &= \frac{1}{N} \sum_{k=\lceil \ell \rceil}^L \sum_{a=1}^4 \frac{z+\phi\widetilde{s}(z)}{z+\phi\widetilde{s}(z)-T_\mu} \mc{E}^{(a)}_k.
\end{align*}
Thus the problem reduces to showing
\[
    \frac{1}{N} \abs{ \frac{z+\phi\widetilde{s}(z)}{z+\phi\widetilde{s}(z)-T_k} \mc{E}^{(a)}_k} \prec \frac{1}{d^{\frac{1}{2}\min(1,\ell)}}
\]
for $k \in \llbracket \lceil \ell \rceil, L \rrbracket$ and $a \in \llbracket 1, 4 \rrbracket$. Lemma \ref{lem:phi_m_tilde_order_one} shows $\abs{z+\phi\widetilde{s}(z)} \prec 1$ as well as
\begin{equation}
\label{eqn:z_plus_phitildem_minus_T}
    \abs{z+\phi\widetilde{s}(z) - T_k} \geq \im(z+\phi\widetilde{s}(z) - T_k) = \im(z+\phi\widetilde{s}(z)) \succ 1,
\end{equation}
so $\abs{\frac{z+\phi\widetilde{s}(z)}{z+\phi\widetilde{s}(z) - T_k}} \prec 1$, and we only need show
\[
    \frac{1}{N} \abs{ \mc{E}^{(a)}_k} \prec \frac{1}{d^{\frac{1}{2}\min(1,\ell)}}
\]
for $k \in \llbracket \lceil \ell \rceil, L \rrbracket$ and $a \in \llbracket 1, 4 \rrbracket$. In the following we will often, but not always, use the estimate $\abs{T_\mu (z+\phi\widetilde{s}(z))^{-1}} \prec 1$, from \eqref{eqn:abs_zphimtilde_orderone}. We handle one $a$ at a time:
\begin{itemize}
\item \textbf{($a = 1$):} On the one hand, from \eqref{eqn:gii_self_consistent} and \eqref{eqn:abs_zphimtilde_orderone} we have
\begin{equation}
\label{eqn:error_e1k_gii}
    \abs{G_{jj} + (z+\phi\widetilde{s}(z))^{-1}} = \frac{\abs{G_{jj}(z+\phi\widetilde{s}(z))+1}}{\abs{z+\phi\widetilde{s}(z)}} \prec \frac{1}{d^{\frac{1}{2}\min(1,\ell)}}.
\end{equation}
Since the distribution of the left-hand side does not depend on $j$, we can put a maximum over $j$ on the left-hand side. On the other hand, we claim
\begin{equation}
\label{eqn:error_e1k_claim}
    \abs{T_k \sum_{\mu \in \mathbf{M}_k} \sum_\nu U_{\mu j} U_{\nu j} G^{(j)}_{\nu \mu}} \prec 1.
\end{equation}
Assume \eqref{eqn:error_e1k_claim} momentarily. Since we can put $\max_{j=1}^N$ on the left-hand side for the same reasons as above, we use it along with \eqref{eqn:error_e1k_gii} to find
\[
    \frac{1}{N} \abs{\mc{E}^{(1)}_k} \leq \left( \max_{j=1}^N \abs{G_{jj}+(z+\phi\widetilde{s}(z))^{-1}} \right) \left( \max_{j=1}^N \abs{T_k \sum_{\mu \in \mathbf{M}_k} \sum_\nu U_{\mu j} U_{\nu j} G^{(j)}_{\nu \mu}} \right) \prec \frac{1}{d^{\frac{1}{2}\min(1,\ell)}}.
\]
Thus it remains only to check \eqref{eqn:error_e1k_claim}. We split $\sum_\nu$ into the term $\nu = \mu$ and the remainder. For the latter, we recall that \eqref{eqn:error_k1k2s} shows that $\abs{\sum_{\mu \in \mathbf{M}_k} \sum_{\nu \in \mathbf{M}_{k'}, \ip{\mu,\nu} = s} U_{\mu j} U_{\nu j} G^{(j)}_{\nu\mu}} \prec d^{-\frac{1}{2}\min(1,\ell)}$ for each $k'$ and $s \leq \min(k,k')$ (except when $k = k'$, in which case $s$ is at most $k-1$). By summing this over the various values of $k'$ and $s$, and using the trivial bound $\abs{T_k} \prec 1$, we obtain
\begin{equation}
\label{eqn:error_k1k2s_modified}
    \abs{ T_k \sum_{\mu \in \mathbf{M}_k} \sum_{\nu}^{(\mu)} U_{\mu j} U_{\nu j} G^{(j)}_{\nu \mu}} \prec \abs{ \sum_{\mu \in \mathbf{M}_k} \sum_{\nu}^{(\mu)} U_{\mu j} U_{\nu j} G^{(j)}_{\nu \mu}} \prec \frac{1}{d^{\frac{1}{2}\min(1,\ell)}},
\end{equation}
which is better than claimed -- this is not the main term. When $\nu = \mu$, we have
\[
    \abs{T_k \sum_{\mu \in \mathbf{M}_k} U_{\mu j}^2 G^{(j)}_{\mu \mu}} \prec \sqrt{\frac{N}{d^k}}\abs{\sum_{\mu \in \mathbf{M}_k} U_{\mu j}^2 (G^{(j)}+T)_{\mu \mu}} + \frac{N}{d^k} \abs{\sum_{\mu \in \mathbf{M}_k} U_{\mu j}^2} \prec \abs{\sum_{\mu \in \mathbf{M}_k} U_{\mu j}^2 (G^{(j)}+T)_{\mu \mu}} + 1,
\]
where we used the simple bound $\abs{\sum_{\mu \in \mathbf{M}_k} U_{\mu j}^2} \leq d^k \max_{\mu \in \mathbf{M}_k} U_{\mu j}^2 \prec \frac{d^k}{N}$. The remaining term is also straightforward: \eqref{eqn:gmi+t} implies $\abs{U_{\mu j}^2 (G^{(j)}+T)_{\mu \mu}} \prec \frac{1}{d^k}$, and taking a maximum over $\mu$ finishes the proof of \eqref{eqn:error_e1k_claim}.
\item \textbf{($a = 2$):} In \eqref{eqn:error_k1k2s_modified} we showed
\[
    \abs{\sum_{\mu \in \mathbf{M}_k} \sum_{\nu}^{(\mu)} U_{\mu j} U_{\nu j} G^{(j)}_{\nu \mu}} \prec \frac{1}{d^{\frac{1}{2}\min(1,\ell)}}.
\]
Since the distribution of the left-hand side does not depend on $j$, we can also put a maximum over $j$ on the left-hand side, and use this to obtain
\[
    \frac{1}{N} \abs{\mc{E}^{(2)}_k} = \frac{1}{N} \abs{T_k (z+\phi\widetilde{s}(z))^{-1}} \sum_{j=1}^N \abs{\sum_{\mu \in \mathbf{M}_k} \sum_{\nu}^{(\mu)} U_{\mu j} U_{\nu j} G^{(j)}_{\nu \mu}} \prec \max_{j=1}^N \abs{\sum_{\mu \in \mathbf{M}_k} \sum_{\nu}^{(\mu)} U_{\mu j} U_{\nu j} G^{(j)}_{\nu \mu}} \prec \frac{1}{d^{\frac{1}{2}\min(1,\ell)}}.
\]
\item \textbf{($a = 3$):} From Lemma \ref{lem:u^2-1/N}, we have
\[
    \abs{\sum_{\mu \in \mathbf{M}_k} \left(U_{\mu j}^2 - \frac{1}{N} \right) T_{\mu}} \prec \frac{1}{\sqrt{d}} \frac{d^k}{N} \sqrt{\frac{N}{d^k}} = \frac{1}{\sqrt{d}} \sqrt{\frac{d^k}{N}}.
\]
Since the distribution of the left-hand side does not depend on $j$, we can also put a maximum over $j$ on the left-hand side, and obtain
\[
    \frac{1}{N} \sqrt{\frac{N}{d^k}} \abs{\sum_{\mu \in \mathbf{M}_k} \sum_{j=1}^N \left(U_{\mu j}^2 - \frac{1}{N}\right) T_\mu} \leq \sqrt{\frac{N}{d^k}} \max_{j=1}^N \abs{\sum_{\mu \in \mathbf{M}_k} \left( U_{\mu j}^2 - \frac{1}{N} \right) T_\mu} \prec \frac{1}{\sqrt{d}}.
\]
Similarly, in \eqref{eqn:error_2k} we showed 
\[
    \abs{\sum_{\mu \in \mathbf{M}_k} \left(U_{\mu j}^2 - \frac{1}{N}\right) (G^{(j)}_{\mu \mu} + T_\mu)} \prec \frac{1}{\sqrt{d}},
\]
and running the same argument about taking the maximum over $j$ yields
\[
    \frac{1}{N} \sqrt{\frac{N}{d^k}} \abs{\sum_{\mu \in \mathbf{M}_k} \sum_{j=1}^N \left(U_{\mu j}^2 - \frac{1}{N}\right) (G^{(j)}_{\mu \mu} + T_\mu)} \prec \frac{1}{\sqrt{d}}
\]
(actually we discard the $\sqrt{N/d^k}$ in the upper bound here). Then, splitting $G^{(j)}_{\mu \mu} = (G^{(j)}_{\mu\mu} + T_\mu) - T_\mu$ and using these two bounds, we obtain
\[
    \frac{1}{N} \abs{\mc{E}^{(3)}_k} = \frac{1}{N} \abs{T_\mu (z+\phi\widetilde{s}(z))^{-1}} \abs{\sum_{\mu \in \mathbf{M}_k} \sum_{j=1}^N \left(U_{\mu j}^2 - \frac{1}{N} \right) G^{(j)}_{\mu \mu}} \prec \frac{1}{N} \sqrt{\frac{N}{d^k}} \abs{\sum_{\mu \in \mathbf{M}_k} \sum_{j=1}^N \left(U_{\mu j}^2 - \frac{1}{N} \right) G^{(j)}_{\mu \mu}} \prec \frac{1}{\sqrt{d}},
\]
which is better than needed.
\item \textbf{($a = 4$):} In \eqref{eqn:error_3k} we showed that
\[
    \frac{1}{N} \abs{\sum_{\mu \in \mathbf{M}_k} (G^{(j)}_{\mu \mu} - G_{\mu \mu})} \prec \frac{1}{N}.
\]
As always we can put a maximum over $j$, then estimate
\[
    \frac{1}{N} \abs{\mc{E}^{(4)}_k} \leq \frac{1}{N^2} \abs{T_\mu (z+\phi\widetilde{s}(z))^{-1}} \sum_j \abs{\sum_{\mu \in \mathbf{M}_k} (G^{(j)}_{\mu\mu} - G_{\mu \mu})} \prec \frac{1}{N} \max_{j=1}^N \abs{\sum_{\mu \in \mathbf{M}_k} (G^{(j)}_{\mu\mu} - G_{\mu\mu})} \prec \frac{1}{N}
\]
which is much better than needed.
\end{itemize}
\end{proof}

In the following result, recall that $\ell_c$ is the least integer \emph{strictly bigger} than $\ell$. 
\begin{lemma}
\label{lem:side_error_term_phitildem}
For any fixed $\tau > 0$ and any fixed $z \in \mathbf{D}_\tau$, we have
\[
    \abs{\frac{1}{N} \sum_\mu  \frac{T_\mu^2}{T_\mu -z - \phi \widetilde{s}(z)} - \frac{\gamma_a}{\gamma_b - z - \phi \widetilde{s}(z)} + \frac{\gamma_c}{z+\phi \widetilde{s}(z)}} \prec \begin{cases} \frac{1}{d^{(\ell_c-\ell)/2}} & \text{if $\ell$ is not an integer,} \\ \frac{1}{d^{(\ell_c-\ell)/2}} + \abs{\frac{N}{d^\ell} - \kappa} & \text{if $\ell$ is an integer.} \end{cases}
\]
\end{lemma}

\begin{proof}[Proof of Lemma \ref{lem:side_error_term_phitildem}]
From the definition $T_k = c_k \sqrt{k!} \sqrt{\frac{N}{d^k}}$, we find
\[
    \frac{1}{N} \sum_\mu  \frac{T_\mu^2}{T_\mu -z - \phi \widetilde{s}(z)} = \frac{1}{N} \sum_{k=\lceil \ell \rceil}^L \sum_{\mu \in \mathbf{M}_k} \frac{T_k^2}{T_k -z - \phi \widetilde{s}(z)} = \sum_{k=\lceil \ell \rceil}^L c_k^2(k!) \frac{M_k}{d^k} \frac{1}{T_k - z - \phi\widetilde{s}(z)}.
\]
Furthermore, since $(k!)M_k$ counts the number of tuples $(a_1,\ldots,a_k) \in [d]^k$ which are all distinct, we have $(k!)M_k/d^k = 1 + \OO(1/d)$ when $c_k \neq 0$ (recall $M_k = 0$ otherwise); since we showed in \eqref{eqn:z_plus_phitildem_minus_T} that $\abs{T_k - z - \phi\widetilde{s}(z)} \succ 1$, this gives
\[
    \abs{\frac{1}{N} \sum_\mu  \frac{T_\mu^2}{T_\mu -z - \phi \widetilde{s}(z)} - \sum_{k=\lceil \ell \rceil}^L \frac{c_k^2}{T_k - z - \phi\widetilde{s}(z)}} \prec \frac{1}{d}.
\]
For $k > \ell$ with strict inequality (i.e., $k \geq \ell_c$), we have $T_k \prec \frac{1}{d^{\frac{k-\ell}{2}}}$. Since $\abs{T_k - z - \phi\widetilde{s}(z)} \succ 1$ (from \eqref{eqn:z_plus_phitildem_minus_T}) and $\abs{z+\phi\widetilde{s}(z)} \succ 1$ (from \eqref{eqn:abs_zphimtilde_orderone}), the definition $\gamma_c = \sum_{k=\ell_c}^L c_k^2$ (recall that $\ell_c$ is the smallest integer \emph{strictly} bigger than $\ell$) gives
\[
    \abs{ \sum_{k=\ell_c}^L \frac{c_k^2}{T_k - z - \phi\widetilde{s}(z)} + \frac{\gamma_c}{z+\phi\widetilde{s}(z)}} = \abs{ \sum_{k=\ell_c}^L \left( \frac{c_k^2}{T_k - z - \phi\widetilde{s}(z)} + \frac{c_k^2}{z+\phi\widetilde{s}(z)} \right) } \prec \sum_{k=\ell_c}^L \abs{T_k} \prec d^{-\frac{(\ell_c-\ell)}{2}}.
\]
If $\ell$ is not an integer, then $\lceil \ell \rceil = \ell_c$ and the proof is complete. Otherwise, the remaining term is $k = \ell$, for which we have 
\begin{equation}
\label{eqn:using_assumption}
    \abs{T_\ell - \gamma_b} = \abs{c_\ell\sqrt{\ell!} (\sqrt{\frac{N}{d^\ell}} - \sqrt{\kappa})} = \OO\left( \abs{ \frac{N}{d^\ell} - \kappa} \right)
\end{equation}
which is $\oo(1)$ by the assumption \eqref{eqn:assn:N/d^ell->kappa_weakly}, and thus (since $\abs{\gamma_b - z- \phi\widetilde{s}(z)} \succ 1$ by the same argument as in \eqref{eqn:z_plus_phitildem_minus_T})
\[
    \abs{\frac{c_\ell^2}{T_\ell - z - \phi\widetilde{s}(z)} - \frac{\gamma_a}{\gamma_b - z - \phi\widetilde{s}(z)}} \prec \abs{T_\ell - \gamma_b} \prec \abs{\frac{N}{d^\ell} - \kappa},
\]
which completes the proof.
\end{proof}

\begin{proof}[Proof of Proposition \ref{prop:widetildem_self_consistent}]
This follows immediately from Proposition \ref{prop:main_error_term_phitildem} and Lemma \ref{lem:side_error_term_phitildem}, simply by noting that
\[
    q_\ell = \min\left\{ \frac{\ell_c-\ell}{2}, \frac{\min(1,\ell)}{2} \right\}.
\]
\end{proof}

\begin{proof}[Proof of Proposition \ref{prop:m_self_consistent}]
This is just an exercise in showing that the stochastic domination bound in \eqref{eqn:widetildem_self_consistent} interacts nicely with the arithmetic. We have
\begin{equation}
\label{eqn:m_self_consistent_error_expansion}
\begin{split}
    \abs{\frac{1}{s(z)} + z + \frac{\gamma_a s(z)}{1+\gamma_b s(z)} + \gamma_c s(z)} &\leq \abs{\frac{1}{s(z)}+z+\phi \widetilde{s}(z)} + \gamma_a \abs{\frac{1}{\gamma_b + \frac{1}{s(z)}} - \frac{1}{\gamma_b - z - \phi \widetilde{s}(z)}} \\
    &\quad + \gamma_c \abs{s(z) + \frac{1}{z+\phi\widetilde{s}(z)}} +  \abs{-\phi \widetilde{s}(z) + \frac{\gamma_a}{\gamma_b - z - \phi \widetilde{s}(z)} - \frac{\gamma_c}{z+\phi\widetilde{s}(z)}}
\end{split}
\end{equation}
By Proposition \ref{prop:widetildem_self_consistent}, the last term on the right-hand side is stochastically dominated by $d^{-q_\ell}$, plus $\abs{\frac{N}{d^\ell}-\kappa}$ in the case that $\ell$ is an integer. Now we make some simple estimates before bounding the first three terms, frequently using that the quantities $\abs{s(z)}$, $\im(s(z))$, $\abs{z+\phi\widetilde{s}(z)}$, and $\im(z+\phi\widetilde{s}(z))$ are stochastically dominated above and below by $1$, from Lemmas \ref{lem:m_order_one} and \ref{lem:phi_m_tilde_order_one}, respectively. For example, since $\gamma_b$ is real, these give us
\[
    \abs{\gamma_b + 1/s(z)} \geq \abs{\im(1/s(z))} \succ 1
\]
and
\[
    \abs{\gamma_b-z-\phi\widetilde{s}(z)} \geq \abs{\im(z+\phi\widetilde{s}(z))} \succ 1,
\]
so that
\[
    \abs{\frac{1}{\gamma_b + \frac{1}{s(z)}} - \frac{1}{\gamma_b - z - \phi \widetilde{s}(z)}} \leq \frac{\abs{\frac{1}{s(z)}+z+\phi\widetilde{s}(z)}}{\abs{\gamma_b + \frac{1}{s(z)}} \abs{\gamma_b - z - \phi\widetilde{s}(z)}} \prec \abs{\frac{1}{s(z)}+z+\phi\widetilde{s}(z)}. 
\]
Since $\abs{\frac{1}{s(z)} + z + \phi\widetilde{s}(z)} \prec \abs{1+s(z)(z+\phi\widetilde{s}(z))}$ and $\abs{s(z) + \frac{1}{z+\phi\widetilde{s}(z)}} \prec \abs{1+s(z)(z+\phi\widetilde{s}(z))}$, we bound the first three terms on the right-hand side of \eqref{eqn:m_self_consistent_error_expansion} by 
\[
    (1+\gamma_a + \gamma_c) \abs{1+s(z)(z+\phi\widetilde{s}(z))} \prec \frac{1}{d^{\frac{1}{2}\min(1,\ell)}} \leq \frac{1}{d^{q_\ell}},
\]
where we applied Proposition \ref{prop:m_widetildem_self_consistent} in the penultimate step, completing the proof of \eqref{eqn:m_self_consistent}.
\end{proof}


\section{Error analysis}
\label{sec:error-terms}

The goal of this section is to prove Proposition \ref{prop:error_stieltjes_conclusion}, which replaces the given matrices $A$ and $\widetilde{A}$ with a friendlier matrix $B$ which has the same global spectral behavior. 

We start by importing the following estimate of \cite{LuYau}.
\begin{lemma}
\label{lem:lem_18_replacement}
\cite[Lemma 18]{LuYau}
If $H_1, H_2$ are $N \times N$ Hermitian matrices with Stieltjes transforms $s_1, s_2$, then 
\begin{align}
    \abs{s_1(z)-s_2(z)} &\leq \frac{\|H_1-H_2\|_{\textup{Frob}}}{\sqrt{N}\eta^2}, \\
    \abs{s_1(z)-s_2(z)} &\leq \frac{C \rank(H_1-H_2)}{N\eta}, \label{eqn:stieltjes_rank_estimate}
\end{align}
where $C$ is an absolute constant.
\end{lemma}

We will apply this to compare the given matrices $A$ and $\widetilde{A}$ with the matrix $B$ defined in \eqref{eqn:main_results_definition_b}. We will need the intermediate error matrices $\widetilde{B}^{\textup{full}}$ and $B^{\textup{full}}$, whose definitions we give below, along with recalling the definitions of $A$, $\widetilde{A}$, and $B$ for the reader's convenience. We recall that, throughout, $f$ is a finite-degree polynomial. All matrices are real-symmetric, $N \times N$, and defined entrywise: 
\begin{align*}
    A_{ij} &= \frac{\delta_{i \neq j}}{\sqrt{N}} f\left( \frac{\ip{X_i,X_j}}{\sqrt{d}} \right) \\
    \widetilde{A}_{ij} &= \begin{cases} \frac{\delta_{i \neq j}}{\sqrt{N}} f\left( \frac{\ip{X_i,X_j}}{\sqrt{d}} \frac{\sqrt{d}}{\|X_i\|} \frac{\sqrt{d}}{\|X_j\|} \right) & \text{if } \|X_i\| \neq 0 \neq \|X_j\|, \\ 0 & \text{otherwise}, \end{cases} \\
    (\widetilde{B}^{\textup{full}})_{ij} &= \begin{cases} \frac{\delta_{i \neq j}}{\sqrt{N}}  \sum_{k=0}^L \left(\frac{d}{\|X_i\|\|X_j\|}\right)^k \frac{c_k}{d^{k/2}\sqrt{k!}} \sum_{a_1,\ldots,a_k=1}^{d,\ast} X_{a_1i} \ldots X_{a_ki} X_{a_1j} \ldots X_{a_kj} & \text{if } \|X_i\| \neq 0 \neq \|X_j\|, \\ 0 & \text{otherwise}, \end{cases} \\
    (B^{\textup{full}})_{ij} &= \frac{\delta_{i \neq j}}{\sqrt{N}}  \sum_{k=0}^L \frac{c_k}{d^{k/2}\sqrt{k!}} \sum_{a_1,\ldots,a_k=1}^{d,\ast} X_{a_1i} \ldots X_{a_ki} X_{a_1j} \ldots X_{a_kj}, \\
    B_{ij} &= \frac{\delta_{i \neq j}}{\sqrt{N}}  \sum_{k=\lceil \ell \rceil}^L \frac{c_k}{d^{k/2}\sqrt{k!}} \sum_{a_1,\ldots,a_k=1}^{d,\ast} X_{a_1i} \ldots X_{a_ki} X_{a_1j} \ldots X_{a_kj},
\end{align*}
(By convention, if $k = 0$, we set $\sum_{a_1,\ldots,a_k=1}^{d,\ast} X_{a_1i} \ldots X_{a_ki} X_{a_1j} \ldots X_{a_kj} = 1$. In this section, we find it easier to work with $\sum_{a_1,\ldots,a_k=1}^{d,\ast}$ than $\sum_{a_1 < \cdots < a_k}$; this is why we write the factors $\sqrt{k!}$ where we do.) The notation ``full'' means that the sum on $k$ in the definitions of $\widetilde{B}^{\textup{full}}$ and $B^{\textup{full}}$ includes $k = 0, \ldots, \lceil \ell \rceil - 1$, which are morally low-rank terms that do not affect the global law. We remove these terms in the step going from $B^{\textup{full}}$ to $B$.

It turns out that all five of these matrices have the same global law, as we will see by considering the error matrices
\[
    E_{A,\widetilde{A}} \defeq A - \widetilde{A}, \qquad E_{\widetilde{A},\widetilde{B}^{\textup{full}}} \defeq \widetilde{A} - \widetilde{B}^{\textup{full}},  \qquad  E_{\widetilde{B}^{\textup{full}},B^{\textup{full}}} \defeq \widetilde{B}^{\textup{full}} - B^{\textup{full}}, \qquad E_{B^{\textup{full}},B} \defeq B^{\textup{full}} - B.
\]
The first three of these matrices each have small Frobenius norm. The last is treated differently, since it contains a low-rank part (this can create spikes but does not affect the global law) which may not have small Frobenius norm, but after subtracting this low-rank part the remainder has small Frobenius norm. We remark that ``small'' means only $\|\cdot\|_{\textup{Frob}} \ll \sqrt{N}$; this kind of estimate does \emph{not} suffice to compare operator norms of $A$, $\widetilde{A}$ and so on (indeed, \cite[pp. 25-26]{LuYau} gives an example where $\|\widetilde{A}\|_{\textup{op}} = \OO_\prec(1)$ but $\|A\|_{\textup{op}} \to \infty$ due to spike eigenvalues), but it \emph{does} suffice to say that all the matrices $A$, $\widetilde{A}$, etc. have the same global law. 

\begin{prop}
\label{prop:error_entrywise}
We have the entrywise bounds
\begin{equation}
\label{eqn:error_entrywise}
    \abs{(E_{A,\widetilde{A}})_{ij}} \prec \frac{1}{\sqrt{Nd}}, \qquad \abs{(E_{\widetilde{A},\widetilde{B}^{\textup{full}}})_{ij}} \prec \frac{1}{\sqrt{Nd}}, \qquad \abs{(E_{\widetilde{B}^{\textup{full}},B^{\textup{full}}})_{ij}} \prec \frac{1}{\sqrt{Nd}},
\end{equation}
and hence (immediately, since $E_{A,\widetilde{A}}$, $E_{\widetilde{A},\widetilde{B}^{\textup{full}}}$, and $E_{\widetilde{B}^{\textup{full}},B^{\textup{full}}}$ each have zero diagonal and equidistributed off-diagonal elements)
\[
    \|E_{A,\widetilde{A}}\|_{\textup{Frob}} \prec \sqrt{\frac{N}{d}}, \qquad \|E_{\widetilde{A},\widetilde{B}^{\textup{full}}}\|_{\textup{Frob}} \prec \sqrt{\frac{N}{d}}, \qquad \|E_{\widetilde{B}^{\textup{full}},B^{\textup{full}}}\|_{\textup{Frob}} \prec \sqrt{\frac{N}{d}}.
\]
\end{prop}

\begin{prop}
\label{prop:error_bf_b}
The matrix $E_{B^{\textup{full}},B}$ can be decomposed into a low-rank part and a small-Frobenius-norm part
\[
    E_{B^{\textup{full}},B} = E_{\textup{lr}} + E_{\textup{Frob}},
\]
where there exists a deterministic sequence $(r_N)_{N=1}^\infty$ such that 
\begin{align*}
    \rank(E_{\textup{lr}}) &\leq r_N \quad \text{almost surely}, \\
    r_N &= \OO(d^{\lceil \ell \rceil -1}), \\
    \|E_{\textup{Frob}}\|_{\textup{Frob}} &\prec d^{(\lceil \ell \rceil - 1)/2}. 
\end{align*}
\end{prop}

\begin{proof}[Proof of Proposition \ref{prop:error_stieltjes_conclusion}, modulo Propositions \ref{prop:error_entrywise} and \ref{prop:error_bf_b}]
Since $B^{\textup{full}} - (B+E_{\textup{lr}}) = E_{B^{\textup{full}},B} - E_{\textup{lr}} = E_{\textup{Frob}}$, Lemma \ref{lem:lem_18_replacement} gives
\begin{align*}
    \abs{s_{\widetilde{A}}(z) - s_B(z)} &\leq \abs{s_{\widetilde{A}}(z) - s_{\widetilde{B}^{\textup{full}}}(z)} + \abs{s_{\widetilde{B}^{\textup{full}}}(z) - s_{B^{\textup{full}}}(z)} + \abs{s_{B^{\textup{full}}}(z) - s_{B+E_{\textup{lr}}}(z)} + \abs{s_{B+E_{\textup{lr}}}(z) - s_B(z)} \\
    &\leq \frac{\|E_{\widetilde{A},\widetilde{B}^{\textup{full}}}\|_{\textup{Frob}}}{\sqrt{N}\eta^2} + \frac{\|E_{\widetilde{B}^{\textup{full}},B^{\textup{full}}}\|_{\textup{Frob}}}{\sqrt{N}\eta^2} + \frac{\|E_{\textup{Frob}}\|_{\textup{Frob}}}{\sqrt{N}\eta^2} + \frac{Cr_N}{N\eta}.
\end{align*}
By Propositions \ref{prop:error_entrywise} and \ref{prop:error_bf_b}, the right-hand side is stochastically dominated by $d^{(\lceil \ell \rceil - \ell - 1)/2} = d^{-r_\ell}$ (the worst term is the third), which tends to zero. As a very weak consequence of this, for any $\epsilon, D > 0$ we have
\[
    \P(\abs{s_{\widetilde{A}}(z) - s_B(z)} \geq \epsilon) \leq C_{\epsilon,D}d^{-D}
\]
which suffices for almost-sure convergence by the Borel-Cantelli lemma. The comparison of $s_A$ to $s_B$ is similar.
\end{proof}

In the remaining sections, we prove Propositions \ref{prop:error_entrywise} and \ref{prop:error_bf_b}. 

\subsection{Common estimates}\

In the proof, we will deal with generic i.i.d. vectors $X$ and $Y$, only later selecting $X = X_i$ and $Y = X_j$, and we will frequently work on the good event
\begin{equation}
\label{eqn:error_good_event}
    \mc{G}_{XY} = \{\|X\| \neq 0 \neq \|Y\|\}
\end{equation}
which has high probability, as we will see.

\begin{lemma}
\label{lem:norm_lower_bound}
If $X \in \R^d$ has i.i.d. entries, centered with unit variance, then 
\begin{align}
    \abs{\|X\|^2 - d} \prec \sqrt{d} \label{eqn:error_clt_stochastic_domination}, \\
    \abs{\|X\| - \sqrt{d}} \prec 1, \label{eqn:error_clt_stochastic_domination_not_squared} \\
    \abs{\frac{d}{\|X\|^2} - 1} \mathbf{1}\{\|X\| \neq 0\} \prec \frac{1}{\sqrt{d}}, \label{eqn:error_clt_stochastic_domination_inverted} \\
    \abs{\frac{\sqrt{d}}{\|X\|} - 1}\mathbf{1}\{\|X\| \neq 0\} \prec \frac{1}{\sqrt{d}}. \label{eqn:error_clt_stochastic_domination_inverted_not_squared}
\end{align}
\end{lemma}
\begin{proof}
Since $\|X\|^2 - d = \sum_{i=1}^d (X_i^2-1)$ is a sum of centered independent variables with all finite moments, the estimate \eqref{eqn:error_clt_stochastic_domination} is standard; see, e.g., \cite[(7.57)]{ErdYau2017}. This gives \eqref{eqn:error_clt_stochastic_domination_not_squared} which in turn gives \eqref{eqn:error_clt_stochastic_domination_inverted} and \eqref{eqn:error_clt_stochastic_domination_inverted_not_squared}. 
\end{proof}

\begin{lemma}
\label{lem:main_f_K_estimate_ymlsimplification}
Let $\mu_X, \mu_Y$ be centered probability measures on $\R$ with unit variance and all finite moments, and let $(X_a)_{a=1}^d, (Y_a)_{a=1}^d$ be independent vectors with all entries i.i.d. samples from $\mu_X$ and $\mu_Y$, respectively. Set
\[
    \widetilde{X} = \begin{cases} \frac{\sqrt{d}}{\|X\|} X & \text{if } \|X\| \neq 0, \\ 0 & \text{otherwise}, \end{cases} \qquad \widetilde{Y} = \begin{cases} \frac{\sqrt{d}}{\|Y\|} Y & \text{if } \|Y\| \neq 0, \\ 0 & \text{otherwise}, \end{cases}
\]
Then for each $g \in \N$ we have
\begin{equation}
\label{eqn:g_fold_error_bound}
    \abs{\sum_{a_1,\ldots,a_g=1}^{d,\ast} X_{a_1}\ldots X_{a_g}Y_{a_1}\ldots Y_{a_g}} \prec d^{g/2}
\end{equation}
and
\begin{equation}
\label{eqn:g_fold_error_bound_tilde}
    \abs{\sum_{a_1,\ldots,a_g=1}^{d,\ast} \widetilde{X}_{a_1}\ldots \widetilde{X}_{a_g}\widetilde{Y}_{a_1}\ldots \widetilde{Y}_{a_g}} \prec d^{g/2}
\end{equation}
\end{lemma}
\begin{proof}
Set $F_d \defeq \sum_{a_1,\ldots,a_g=1}^{d,\ast} X_{a_1} \ldots X_{a_g} Y_{a_1} \ldots Y_{a_g}$. For $p \in \N$, we have
\[
    \E[(F_d)^{2p}] = \sum_{a^{(1)}_1,\ldots,a^{(1)}_g=1}^{d,\ast} \cdots \sum_{a^{(2p)}_1,\ldots,a^{(2p)}_g=1}^{d,\ast} \underbrace{\E\left[ \prod_{b=1}^{2p} X_{a^{(b)}_1} \ldots X_{a^{(b)}_g} Y_{a^{(b)}_1} \ldots Y_{a^{(b)}_g}\right]}_{\eqdef G(a^{(1)}_1,\ldots,a^{(1)}_g,\ldots,a^{(2p)}_1,\ldots,a^{(2p)}_g)}.
\]
Since all the entries of $X$ and $Y$ are centered and independent, and $X$ and $Y$ are independent of one another, $G(a^{(1)}_1,\ldots,a^{(1)}_g,\ldots,a^{(2p)}_1,\ldots,a^{(2p)}_g)$ if any of its arguments appears only one time. This forces index coincidences, specifically of the form 
\[
    \#\{a^{(1)}_1,\ldots,a^{(1)}_g,\ldots,a^{(2p)}_1,\ldots,a^{(2p)}_g\} \leq pg
\]
when $G(a^{(1)}_1,\ldots,a^{(1)}_g,\ldots,a^{(2p)}_1,\ldots,a^{(2p)}_g) \neq 0$, instead of the naive $2pg$. Thus
\[
    \#\{(a^{(1)}_1,\ldots,a^{(1)}_g,\ldots,a^{(2p)}_1,\ldots,a^{(2p)}_g) : G(a^{(1)}_1,\ldots,a^{(1)}_g,\ldots,a^{(2p)}_1,\ldots,a^{(2p)}_g) \neq 0\} \leq C_{p,g} d^{pg},
\]
since one can first select at most $pg$ elements of $\llbracket 1, d \rrbracket$ to be the values of the set $\{a^{(1)}_1,\ldots,a^{(1)}_g,\ldots,a^{(2p)}_1,\ldots,a^{(2p)}_g\}$, and once these values are selected, choosing which values to assign to which $a$'s only adds a multiplicative factor $C_{p,g}$. Furthermore, since the $X$ and $Y$ entries have all finite moments, H\"{o}lder's inequality gives us
\[
    \abs{G(a^{(1)}_1,\ldots,a^{(1)}_g,\ldots,a^{(2p)}_1,\ldots,a^{(2p)}_g)} \leq C_{p,g},
\]
and thus
\[
    \E[(F_d)^{2p}] \leq C_{p,g}d^{pg}
\]
which suffices for \eqref{eqn:g_fold_error_bound}. For \eqref{eqn:g_fold_error_bound_tilde}, we set $\widetilde{F_d} = \sum_{a_1,\ldots,a_g=1}^{d,\ast} \widetilde{X}_{a_1}\ldots \widetilde{X}_{a_g}\widetilde{Y}_{a_1}\ldots \widetilde{Y}_{a_g}$. If $\|X\| = 0$ or $\|Y\| = 0$ then \eqref{eqn:g_fold_error_bound_tilde} is trivial, so it suffices to restrict to the good event $\mc{G}_{XY}$ from \eqref{eqn:error_good_event}; on this event, we have
\[
    \abs{\widetilde{F_d}} = \abs{F_d} \abs{\left(\frac{\sqrt{d}}{\|X\|}\right)^g} \abs{\left(\frac{\sqrt{d}}{\|Y\|}\right)^g} \prec \abs{F_d} \prec d^{g/2},
\]
where the first inequality follows from Lemma \ref{lem:norm_lower_bound} and the second from the first half of this proof. 
\end{proof}

\subsection{Estimates for \texorpdfstring{$E_{A,\widetilde{A}}$}{E{A,~A}}}\

\begin{lemma}
\label{lem:error_entrywise_a_tildea}
Let $\mu$ be a centered probability measure on $\R$ with unit variance and all finite moments, and let $(X_a)_{a=1}^d, (Y_a)_{a=1}^d$ be independent vectors with all entries i.i.d. samples from $\mu$. 
Set
\[
    \widetilde{X} = \begin{cases} \frac{\sqrt{d}}{\|X\|} X & \text{if } \|X\| \neq 0, \\ 0 & \text{otherwise}, \end{cases} \qquad \widetilde{Y} = \begin{cases} \frac{\sqrt{d}}{\|Y\|} Y & \text{if } \|Y\| \neq 0, \\ 0 & \text{otherwise}. \end{cases}
\]
Then for each $k$ we have
\begin{equation}
\label{eqn:error_yml_simplification_nonnormalized_estimate}
    \abs{H_k\left(\frac{\sum_{a=1}^d \widetilde{X}_a \widetilde{Y}_a}{\sqrt{d}}\right) - H_k\left(\frac{\sum_{a=1}^d X_a Y_a}{\sqrt{d}}\right)} \prec \frac{1}{\sqrt{d}}.
\end{equation}
\end{lemma}
\begin{proof}
The good set $\mc{G}_{XY} = \{\|X\| \neq 0 \neq \|Y\|\}$ from \eqref{eqn:error_good_event} has much higher probability than required by stochastic domination, since if $\P(X_a = 0) = p < 1$ then we have
\[
    \P(\mc{G}_{XY}^c) = p^{2d}
\]
which tends to zero exponentially quickly in $d$. Thus we can restrict to $\mc{G}_{XY}$ when showing \eqref{eqn:error_yml_simplification_nonnormalized_estimate}. On this event, from Lemma \ref{lem:norm_lower_bound} we have
\[
    \frac{\sqrt{d}}{\|X\|} = 1 + \delta_X, \qquad \frac{\sqrt{d}}{\|Y\|} = 1 + \delta_Y
\]
with error terms $\delta_X, \delta_Y$ satisfying $\abs{\delta_X} \prec 1/\sqrt{d}$ and $\abs{\delta_Y} \prec \frac{1}{\sqrt{d}}$. Thus
\[
    \frac{\sum_{a=1}^d \widetilde{X}_a \widetilde{Y}_a}{\sqrt{d}} = \frac{\sum_{a=1}^d X_a Y_a}{\sqrt{d}} \frac{\sqrt{d}}{\|X\|} \frac{\sqrt{d}}{\|Y\|} = \frac{\sum_{a=1}^d X_a Y_a}{\sqrt{d}} (1+\delta_X)(1+\delta_Y) = \frac{\sum_{a=1}^d X_a Y_a}{\sqrt{d}} + \underbrace{(\delta_X + \delta_Y + \delta_X\delta_Y) \frac{\sum_{a=1}^d X_a Y_a}{\sqrt{d}}}_{=:\epsilon_{XY}},
\]
with an error term satisfying $\abs{\epsilon_{XY}} \prec 1/\sqrt{d}$. This already completes the proof if $k = 0, 1$. For $k > 1$, we will Taylor expand; in order to do this, given $\eta > 0$ we introduce the good event
\[
    \mc{G}_\eta = \{\abs{\epsilon_{XY}} \leq d^{\frac{\eta}{4}-\frac{1}{2}}\} \cap \left\{\abs{\sum_{a=1}^d X_a Y_a} \leq d^{\frac{\eta}{4(k-1)}+\frac{1}{2}}\right\}.
\]
Since $\abs{\mc{E}_{XY}} \prec 1/\sqrt{d}$ and $\abs{\sum_a X_a Y_a} \prec \sqrt{d}$, we know that for every $D > 0$ there exists $C_D$ and $d_0(\eta,D)$ such that, for $d \geq d_0(\eta,D)$, 
\begin{equation}
\label{eqn:mcg_eta_c}
    \P(\mc{G}_\eta^c) \leq C_D d^{-D}.
\end{equation}
At the same time, since $H'_k$ is a degree-$(k-1)$ polynomial, there exists $C_k$ such that for every $\alpha > 0$ we have $\sup_{\abs{x} \leq d^\alpha} \abs{H'_k(x)} \leq C_k d^{(k-1)\alpha}$. Thus, on $\mc{G}_\eta$, a first-order Taylor expansion gives
\begin{align*}
    \abs{H_k\left(\frac{\sum_{a=1}^d \widetilde{X}_a \widetilde{Y}_a}{\sqrt{d}}\right) - H_k\left(\frac{\sum_{a=1}^d X_a Y_a}{\sqrt{d}}\right)} &= \abs{H_k\left(\frac{\sum_{a=1}^d X_a Y_a}{\sqrt{d}} + \epsilon_{XY} \right) - H_k\left(\frac{\sum_{a=1}^d X_a Y_a}{\sqrt{d}}\right)} \\
    &\leq \abs{\epsilon_{XY}} \sup_{\abs{x} \leq d^{\frac{\eta}{2(k-1)}}} \abs{H'_k(x)} \leq d^{\frac{\eta}{2}-\frac{1}{2}}
\end{align*}
and therefore
\[
    \P\left( \abs{H_k\left(\frac{\sum_{a=1}^d \widetilde{X}_a \widetilde{Y}_a}{\sqrt{d}}\right) - H_k\left(\frac{\sum_{a=1}^d X_a Y_a}{\sqrt{d}}\right)} \geq d^{\eta-1/2} \right) \leq \P(\mc{G}_\eta^c).
\]
Combined with \eqref{eqn:mcg_eta_c}, this completes the proof.
\end{proof}

\subsection{Estimates for \texorpdfstring{$E_{\widetilde{A},\widetilde{B}^{\textup{full}}}$}{E{~A,~Bfull}}}\

\begin{lemma}
\label{lem:error_entrywise_tildea_tildeb}
Let $\mu$ be a centered probability measure on $\R$ with unit variance and all finite moments, and let $(X_a)_{a=1}^d, (Y_a)_{a=1}^d$ be independent vectors with all entries i.i.d. samples from $\mu$. 
Set
\[
    \widetilde{X} = \begin{cases} \frac{\sqrt{d}}{\|X\|} X & \text{if } \|X\| \neq 0, \\ 0 & \text{otherwise}, \end{cases} \qquad \widetilde{Y} = \begin{cases} \frac{\sqrt{d}}{\|Y\|} Y & \text{if } \|Y\| \neq 0, \\ 0 & \text{otherwise}. \end{cases}
\]
Then for each $k$ we have
\begin{equation}
\label{eqn:error_yml_simplification_main_estimate}
    \abs{H_k\left(\frac{\sum_{a=1}^d \widetilde{X}_a \widetilde{Y}_a}{\sqrt{d}}\right) - \frac{1}{d^{k/2}}  \sum_{a_1, \ldots, a_k =1}^{d,\ast} \widetilde{X}_{a_1} \ldots \widetilde{X}_{a_k} \widetilde{Y}_{a_1} \ldots \widetilde{Y}_{a_k} } \prec \frac{1}{\sqrt{d}}.
\end{equation}
(Here $H_k$ is the $k$th \emph{monic} Hermite polynomial, which satisfies $H_k = \sqrt{k!}h_k$; we use this for this lemma only so that we do not need to carry $\sqrt{k!}$'s everywhere.)
\end{lemma}
\begin{proof}
We restrict to the good event $\mc{G}_{XY} = \{\|X\| \neq 0 \neq \|Y\|\}$ in the same way as before. On this event, we first compute
\[
    \sum_{a=1}^d \widetilde{X}_a^2 \widetilde{Y}_a^2 = \sum_{a=1}^d ((\widetilde{X}_a^2-1)+1)((\widetilde{Y}_a^2-1)+1) = \sum_{a=1}^d (\widetilde{X}_a^2-1)(\widetilde{Y}_a^2-1) + d,
\]
since $\sum_a (\widetilde{X}_a^2-1) = \sum_a (\widetilde{Y}_a^2-1) = 0$, so that
\[
    \frac{\sum_a \widetilde{X}_a^2 \widetilde{Y}_a^2}{d} - 1 = \frac{\sum_a (\widetilde{X}_a^2-1)(\widetilde{Y}_a^2-1)}{d} \eqdef \Delta.
\]
We will need the estimate $\abs{\Delta} \prec 1/\sqrt{d}$. This does not follow from Lemma \ref{lem:main_f_K_estimate_ymlsimplification}, because the variables $\widetilde{X}_a^2-1$ are not independent and centered, nor are they the normalizations of such variables. To handle this, we rewrite
\begin{align*}
    \Delta &= \frac{\sum_a (\widetilde{X}_a^2-1)(\widetilde{Y}_a^2-1)}{d} = \frac{\sum_a (\widetilde{X}_a^2-\frac{d}{\|X\|^2} + \frac{d}{\|X\|^2} - 1)(\widetilde{Y}_a^2 - \frac{d}{\|Y\|^2} + \frac{d}{\|Y\|^2} -1)}{d} \\
    &= \frac{\sum_a (\widetilde{X}_a^2 - \frac{d}{\|X\|^2})(\widetilde{Y}_a^2 - \frac{d}{\|Y\|^2})}{d} + \left(\frac{d}{\|X\|^2}-1\right)\left(\frac{d}{\|Y\|^2} - 1\right) \\
    &= \left(\frac{d}{\|X\|^2} \frac{d}{\|Y\|^2}\right) \frac{\sum_a (X_a^2-1)(Y_a^2-1)}{d} + \left(\frac{d}{\|X\|^2}-1\right)\left(\frac{d}{\|Y\|^2}-1\right).
\end{align*}
Since the variables $X_a^2-1$ are centered with order-one variance, Lemma \ref{lem:main_f_K_estimate_ymlsimplification} \emph{does} apply to them; the estimate \eqref{eqn:g_fold_error_bound} gives $\abs{\sum_a (X_a^2-1)(Y_a^2-1)} \prec \sqrt{d}$. Combining this with several applications of Lemma \ref{lem:norm_lower_bound}, we find
\begin{equation}
\label{eqn:error_Delta_size}
    \abs{\Delta} \prec \frac{1}{\sqrt{d}}.
\end{equation}
Now we prove \eqref{eqn:error_yml_simplification_main_estimate} by induction on $k$, using the three-term recurrence formula
\[
    H_{k+1}(x) = x H_k(x) - k H_{k-1}(x).
\]
\begin{itemize}
\item $\mathbf{k=0, 1}$: These are trivial, since the left-hand side of \eqref{eqn:error_yml_simplification_main_estimate} is deterministically zero.
\item $\mathbf{k=2}$: Since $H_2(x) = x^2 - 1$, we have
\begin{align*}
    \abs{H_2 \left( \frac{\sum_a \widetilde{X}_a \widetilde{Y}_a}{\sqrt{d}} \right) - \frac{1}{d} \sum_{a,b=1}^{d,\ast} \widetilde{X}_a \widetilde{X}_b \widetilde{Y}_a \widetilde{Y}_b } &= \abs{ \frac{1}{d} \sum_{a,b=1}^d \widetilde{X}_a \widetilde{X}_b \widetilde{Y}_a \widetilde{Y}_b - 1 - \frac{1}{d} \sum_{a,b=1}^{d,\ast} \widetilde{X}_a \widetilde{X}_b \widetilde{Y}_a \widetilde{Y}_b} \\
    &= \abs{\frac{\sum_{a=1}^d \widetilde{X}_a^2 \widetilde{Y}_a^2}{d} - 1} = \abs{\Delta} \prec \frac{1}{\sqrt{d}}.
\end{align*}
\item $\mathbf{k \geq 3}$: We claim that
\begin{equation}
\label{eqn:error_ymlsimplification_momentarily}
    \abs{ \frac{1}{d} \sum_{a_1,\ldots,a_k=1}^{d,\ast} \widetilde{X}_{a_1}^2 \widetilde{X}_{a_2} \ldots \widetilde{X}_{a_k} \widetilde{Y}_{a_1}^2 \widetilde{Y}_{a_2} \ldots \widetilde{Y}_{a_k} -  \sum_{a_1,\ldots,a_{k-1}=1}^{d,\ast} \widetilde{X}_{a_1}\ldots \widetilde{X}_{a_{k-1}} \widetilde{Y}_{a_1} \ldots \widetilde{Y}_{a_{k-1}}} \prec d^{\frac{k-2}{2}}.
\end{equation}
Assume this claim momentarily. By induction, we have
\begin{align*}
    H_k \left( \frac{\sum_a \widetilde{X}_a \widetilde{Y}_a}{\sqrt{d}} \right) &= \frac{1}{d^{k/2}}  \sum_{a_1, \ldots, a_k =1}^{d,\ast} \widetilde{X}_{a_1} \ldots \widetilde{X}_{a_k} \widetilde{Y}_{a_1} \ldots \widetilde{Y}_{a_k} + \mc{E}_k, \\
    H_{k-1} \left( \frac{\sum_a \widetilde{X}_a \widetilde{Y}_a}{\sqrt{d}} \right) &= \frac{1}{d^{(k-1)/2}}  \sum_{a_1, \ldots, a_{k-1} =1}^{d,\ast} \widetilde{X}_{a_1} \ldots \widetilde{X}_{a_{k-1}} \widetilde{Y}_{a_1} \ldots \widetilde{Y}_{a_{k-1}} + \mc{E}_{k-1}
\end{align*}
with error terms satisfying $\abs{\mc{E}_k} \prec 1/\sqrt{d}$ and $\abs{\mc{E}_{k-1}} \prec 1/\sqrt{d}$. From the three-term recurrence for Hermite polynomials, we obtain
\begin{align*}
    H_{k+1} \left( \frac{\sum_a \widetilde{X}_a \widetilde{Y}_a}{\sqrt{d}} \right) =& \, \left( \frac{\sum_b \widetilde{X}_b \widetilde{Y}_b}{\sqrt{d}} \right) H_k \left( \frac{\sum_a \widetilde{X}_a \widetilde{Y}_a}{\sqrt{d}} \right) - k H_{k-1} \left( \frac{\sum_a \widetilde{X}_a \widetilde{Y}_a}{\sqrt{d}} \right) \\
    =& \, \left( \frac{\sum_b \widetilde{X}_b \widetilde{Y}_b}{\sqrt{d}} \right)\left(\frac{1}{d^{k/2}}  \sum_{a_1, \ldots, a_k =1}^{d,\ast} \widetilde{X}_{a_1} \ldots \widetilde{X}_{a_k} \widetilde{Y}_{a_1} \ldots \widetilde{Y}_{a_k} + \mc{E}_k\right) \\
    &- k\left( \frac{1}{d^{(k-1)/2}}  \sum_{a_1, \ldots, a_{k-1} =1}^{d,\ast} \widetilde{X}_{a_1} \ldots \widetilde{X}_{a_{k-1}} \widetilde{Y}_{a_1} \ldots \widetilde{Y}_{a_{k-1}} + \mc{E}_{k-1} \right).
\end{align*}
Consider the product 
\begin{equation}
\label{eqn:error_ymlsimplification_bmatching}
    \left( \frac{\sum_b \widetilde{X}_b \widetilde{Y}_b}{\sqrt{d}} \right)\left(\frac{1}{d^{k/2}}  \sum_{a_1, \ldots, a_k =1}^{d,\ast} \widetilde{X}_{a_1} \ldots \widetilde{X}_{a_k} \widetilde{Y}_{a_1} \ldots \widetilde{Y}_{a_k}\right).
\end{equation}
When we multiply the sums together, either the index $b$ is distinct from the indices $\{a_1,\ldots,a_k\}$, or it is not. If $b$ is distinct, this contributes to the main term; if $b$ is not distinct, we end up with a term of the form $\sum_{a_1,\ldots,a_k=1}^{d,\ast} \widetilde{X}_{a_1}^2 \widetilde{X}_{a_2}\ldots \widetilde{X}_{a_k} \widetilde{Y}_{a_1}^2 \widetilde{Y}_{a_2} \ldots \widetilde{Y}_{a_k}$. Since the indices $\{a_1,\ldots,a_k\}$ are themselves distinct, $b$ can only match with one of them, and this can happen in $k$ ways; thus the expression in \eqref{eqn:error_ymlsimplification_bmatching} is equal to 
\begin{align*}
    \frac{1}{d^{(k+1)/2}} \sum_{a_1,\ldots,a_{k+1}=1}^{d,\ast} \widetilde{X}_{a_1} \ldots \widetilde{X}_{a_{k+1}} \widetilde{Y}_{a_1} \ldots \widetilde{Y}_{a_{k+1}} + \frac{k}{d^{(k+1)/2}} \sum_{a_1,\ldots,a_k=1}^{d,\ast} \widetilde{X}_{a_1}^2 \widetilde{X}_{a_2}\ldots \widetilde{X}_{a_k} \widetilde{Y}_{a_1}^2 \widetilde{Y}_{a_2} \ldots \widetilde{Y}_{a_k}.
\end{align*}
Combining this with the estimate $\abs{(\sum_b \widetilde{X}_b \widetilde{Y}_b/\sqrt{d})\mc{E}_k} \prec \abs{\mc{E}_k} \prec 1/\sqrt{d}$, from Lemma \ref{eqn:g_fold_error_bound}; the estimate $\abs{k\mc{E}_{k-1}} \prec 1/\sqrt{d}$, of course; and \eqref{eqn:error_ymlsimplification_momentarily}, we obtain
\[
    H_{k+1} \left( \frac{\sum_a \widetilde{X}_a \widetilde{Y}_a}{\sqrt{d}} \right) = \frac{1}{d^{(k+1)/2}} \sum_{a_1,\ldots,a_{k+1}=1}^{d,\ast} \widetilde{X}_{a_1} \ldots \widetilde{X}_{a_{k+1}} \widetilde{Y}_{a_1} \ldots \widetilde{Y}_{a_{k+1}} + \OO_\prec \left( \frac{1}{\sqrt{d}}\right)
\]
as desired.

Now we prove \eqref{eqn:error_ymlsimplification_momentarily}. Applying the same type of expansions as discussed just after \eqref{eqn:error_ymlsimplification_bmatching}, we find
\begin{align*}
    &\frac{1}{d} \sum_{a_1,\ldots,a_k=1}^{d,\ast} \widetilde{X}_{a_1}^2 \widetilde{X}_{a_2} \ldots \widetilde{X}_{a_k} \widetilde{Y}_{a_1}^2 \widetilde{Y}_{a_2} \ldots \widetilde{Y}_{a_k} -  \sum_{a_1,\ldots,a_{k-1}=1}^{d,\ast} \widetilde{X}_{a_1}\ldots \widetilde{X}_{a_{k-1}} \widetilde{Y}_{a_1} \ldots \widetilde{Y}_{a_{k-1}} \\
    &= \frac{1}{d} \left( \sum_{b=1}^d \widetilde{X}_b^2 \widetilde{Y}_b^2 \right) \left( \sum_{a_2,\ldots,a_k=1}^{d,\ast} \widetilde{X}_{a_2} \ldots \widetilde{X}_{a_k} \widetilde{Y}_{a_2} \ldots \widetilde{Y}_{a_k} \right) - \frac{k-1}{d} \sum_{a_1,\ldots,a_{k-1}=1}^{d,\ast} \widetilde{X}_{a_1}^3 \widetilde{X}_{a_2} \ldots \widetilde{X}_{a_{k-1}} \widetilde{Y}_{a_1}^3 \widetilde{Y}_{a_2} \ldots \widetilde{Y}_{a_{k-1}} \\
    &\quad - \sum_{a_1,\ldots,a_{k-1}=1}^{d,\ast} \widetilde{X}_{a_1} \ldots \widetilde{X}_{a_{k-1}} \widetilde{Y}_{a_1} \ldots \widetilde{Y}_{a_{k-1}} \\
    &= \Delta \left( \sum_{a_1,\ldots,a_{k-1}=1}^{d,\ast} \widetilde{X}_{a_1} \ldots \widetilde{X}_{a_{k-1}} \widetilde{Y}_{a_1} \ldots \widetilde{Y}_{a_{k-1}} \right) - \frac{k-1}{d} \sum_{a_1,\ldots,a_{k-1}=1}^{d,\ast} \widetilde{X}_{a_1}^3 \widetilde{X}_{a_2} \ldots \widetilde{X}_{a_{k-1}} \widetilde{Y}_{a_1}^3 \widetilde{Y}_{a_2} \ldots \widetilde{Y}_{a_{k-1}}. 
\end{align*}
By the estimate \eqref{eqn:error_Delta_size} and Lemma \ref{lem:main_f_K_estimate_ymlsimplification}, the first term on the right-hand side is stochastically dominated by $\frac{1}{\sqrt{d}} d^{\frac{k-1}{2}} = d^{\frac{k-2}{2}}$ in absolute value. To handle the second term on the right-hand side, we will make expansions like 
\begin{align*}
    &\sum_{a_1,\ldots,a_{k-1}=1}^{d,\ast} \widetilde{X}_{a_1}^3 \widetilde{X}_{a_2} \ldots \widetilde{X}_{a_{k-1}} \widetilde{Y}_{a_1}^3 \widetilde{Y}_{a_2} \ldots \widetilde{Y}_{a_{k-1}} \\
    &= \left( \sum_{b=1}^d \widetilde{X}_b^3 \widetilde{Y}_b^3 \right) \left( \sum_{a_2,\ldots,a_{k-1}=1}^{d,\ast} \widetilde{X}_{a_2}\ldots \widetilde{X}_{a_{k-1}} \widetilde{Y}_{a_2} \ldots \widetilde{Y}_{a_{k-1}} \right) \\
    &- (k-2) \sum_{a_1,\ldots,a_{k-2}=1}^{d,\ast} \widetilde{X}_{a_1}^4 \widetilde{X}_{a_2} \ldots \widetilde{X}_{a_{k-2}} \widetilde{Y}_{a_1}^4 \widetilde{Y}_{a_2} \ldots \widetilde{Y}_{a_{k-1}},
\end{align*}
then expand from fourth powers into fifth powers, and so on, until the process terminates when all that remains is $\sum_{a=1}^d \widetilde{X}_a^p \widetilde{Y}_a^p$ for some power $p$ (and an irrelevant prefactor $C_k$). To track this, we introduce the following bookkeeping notation: Defining
\begin{align*}
    \alpha_{k,p} &\defeq \sum_{a_1,\ldots,a_{k+2-p}=1}^{d,\ast} \widetilde{X}_{a_1}^p \widetilde{X}_{a_2} \ldots \widetilde{X}_{a_{k+2-p}} \widetilde{Y}_{a_1}^p \widetilde{Y}_{a_2} \ldots \widetilde{Y}_{a_{k+2-p}} \\
    \beta_{k,p} &\defeq \left( \sum_{b=1}^d \widetilde{X}_b^p \widetilde{Y}_b^p \right) \left( \sum_{a_2, \ldots, a_{k+2-p}=1}^{d,\ast} \widetilde{X}_{a_2} \ldots \widetilde{X}_{a_{k+2-p}} \widetilde{Y}_{a_2} \ldots \widetilde{Y}_{a_{k+2-p}} \right) 
\end{align*}
expansions like those above show, with $x_n$ the falling factorial $x_n = x(x-1)(x-2) \cdots (x-n+1)$, 
\begin{align*}
    \alpha_{k,p} &= \beta_{k,p} - (k+1-p) \alpha_{k,p+1} = \beta_{k,p} - (k+1-p)(\beta_{k,p+1} - (k-p) \alpha_{k,p+2}) \\
    &= \beta_{k,p} - (k+1-p) \beta_{k,p+1} + (k+1-p)_2 (\beta_{k,p+2} - (k-1+p)\alpha_{k,p+3}) \\
    &= \left( \sum_{j=0}^{k+1} (-1)^j (k+1-p)_j \beta_{k,p+j} \right) + (k+1-p)_{k+2} \alpha_{k,p+k+2}.
\end{align*}
It is easy to compute $\abs{\sum_{a=1}^d \widetilde{X}_a^p \widetilde{Y}_a^p} \prec d$ for any fixed $p$; in particular
\[
    \abs{\alpha_{k,p+k+2}} = \abs{\sum_{a=1}^d \widetilde{X}_a^{p+k+2} \widetilde{Y}_a^{p+k+2}} \prec d,
\]
and combining this with \eqref{eqn:g_fold_error_bound_tilde} we obtain
\[
    \abs{\beta_{k,p}} \prec d^{1+\frac{k+1-p}{2}}
\]
(since the indexing in the definition of $\beta$ starts with $a_2$). Thus
\[
    \abs{\alpha_{k,p}} \prec d^{1+\frac{k+1-p}{2}}
\]
as long as $k+1-p \geq 0$. In particular, 
\[
    \abs{\frac{k-1}{d} \sum_{a_1,\ldots,a_{k-1}=1}^{d,\ast} \widetilde{X}_{a_1}^3 \widetilde{X}_{a_2} \ldots \widetilde{X}_{a_{k-1}} \widetilde{Y}_{a_1}^3 \widetilde{Y}_{a_2} \ldots \widetilde{Y}_{a_{k-1}}} \prec \frac{1}{d} \abs{ \alpha_{k,3} } \prec d^{\frac{k-2}{2}}
\]
which completes the proof of \eqref{eqn:error_ymlsimplification_momentarily}. 
\end{itemize}
\end{proof}

\subsection{Estimates for \texorpdfstring{$E_{\widetilde{B}^{\textup{full}},B^{\textup{full}}}$}{E{~Bfull,Bfull}}}\

\begin{lemma}
\label{lem:error_entrywise_tildeb_b}
Let $\mu$ be a centered probability measure on $\R$ with unit variance and all finite moments, and let $(X_a)_{a=1}^d, (Y_a)_{a=1}^d$ be independent vectors with all entries i.i.d. samples from $\mu$. Set
\[
    \widetilde{X} = \begin{cases} \frac{\sqrt{d}}{\|X\|} X & \text{if } \|X\| \neq 0, \\ 0 & \text{otherwise}, \end{cases} \qquad \widetilde{Y} = \begin{cases} \frac{\sqrt{d}}{\|Y\|} Y & \text{if } \|Y\| \neq 0, \\ 0 & \text{otherwise}. \end{cases}
\]
Then for each $k$ we have
\begin{equation}
\label{eqn:error_entrywise_tildeb_b}
    \abs{\frac{1}{d^{k/2}}  \sum_{a_1, \ldots, a_k =1}^{d,\ast} \widetilde{X}_{a_1} \ldots \widetilde{X}_{a_k} \widetilde{Y}_{a_1} \ldots \widetilde{Y}_{a_k} - \frac{1}{d^{k/2}}  \sum_{a_1, \ldots, a_k =1}^{d,\ast} X_{a_1} \ldots X_{a_k} Y_{a_1} \ldots Y_{a_k} } \prec \frac{1}{\sqrt{d}}.
\end{equation}
\end{lemma}
\begin{proof}
We restrict to the good event $\mc{G}_{XY} = \{\|X\| \neq 0 \neq \|Y\|\}$ in the usual way. From \eqref{eqn:g_fold_error_bound} we have
\begin{align*}
    &\abs{\frac{1}{d^{k/2}}  \sum_{a_1, \ldots, a_k =1}^{d,\ast} \widetilde{X}_{a_1} \ldots \widetilde{X}_{a_k} \widetilde{Y}_{a_1} \ldots \widetilde{Y}_{a_k} - \frac{1}{d^{k/2}}  \sum_{a_1, \ldots, a_k =1}^{d,\ast} X_{a_1} \ldots X_{a_k} Y_{a_1} \ldots Y_{a_k} } \\
    &= \abs{\left( \frac{d}{\|X\|\|Y\|} \right)^k - 1} \abs{ \frac{1}{d^{k/2}}  \sum_{a_1, \ldots, a_k =1}^{d,\ast} X_{a_1} \ldots X_{a_k} Y_{a_1} \ldots Y_{a_k} } \prec \abs{\left( \frac{d}{\|X\|\|Y\|} \right)^k - 1}.
\end{align*}
But 
\[
    \abs{\frac{d}{\|X_i\|\|X_j\|} - 1} \leq \frac{\sqrt{d}}{\|X_i\|} \abs{\frac{\sqrt{d}}{\|X_j\|} - 1} + \abs{\frac{\sqrt{d}}{\|X_i\|} - 1} \prec \frac{1}{\sqrt{d}}
\]
and thus
\begin{align*}
    \abs{\left( \frac{d}{\|X_i\|\|X_j\|} \right)^k - 1 } &\leq \sum_{j=1}^k \abs{ \left(\frac{d}{\|X_i\|\|X_j\|}\right)^j - \left(\frac{d}{\|X_i\|\|X_j\|}\right)^{j-1} } = \abs{\frac{d}{\|X_i\|\|X_j\|} - 1} \sum_{j=1}^k \left(\frac{d}{\|X_i\|\|X_j\|} \right)^j \\
    &\prec \abs{\frac{d}{\|X_i\|\|X_j\|} - 1} \prec \frac{1}{\sqrt{d}},
\end{align*}
which finishes the proof.
\end{proof}

\subsection{Proof of Propositions \ref{prop:error_entrywise} and \ref{prop:error_bf_b}}

\begin{proof}[Proof of Proposition \ref{prop:error_entrywise}]
The decomposition $f(x) = \sum_{k=0}^L c_k h_k(x)$ induces decompositions
\[
    A = \sum_{k=0}^L c_k A_k, \qquad \widetilde{A} = \sum_{k=0}^L c_k \widetilde{A}_k, \qquad \widetilde{B}^{\textup{full}} = \sum_{k=0}^L c_k (\widetilde{B}^{\textup{full}})_k, \qquad B^{\textup{full}} = \sum_{k=0}^L c_k (B^{\textup{full}})_k,
\]
which in turn induce decompositions
\[
    E_{A, \widetilde{A}} = \sum_{k=0}^L c_k E_{A,\widetilde{A},k}, \qquad E_{\widetilde{A}, \widetilde{B}^{\textup{full}}} = \sum_{k=0}^L c_k E_{\widetilde{A}, \widetilde{B}^{\textup{full}}, k}, \qquad E_{\widetilde{B}^{\textup{full}}, B^{\textup{full}}} = \sum_{k=0}^L c_k E_{\widetilde{B}^{\textup{full}}, B^{\textup{full}}, k}.
\]
Since there is a finite number of terms in the sum, it suffices to prove the desired estimates one $k$ at a time. In the usual way, we can restrict to the good event $\{\|X_i\| \neq 0 \neq \|X_j\|\}$, in which case we can apply the preceding lemmas by choosing $X = X_i$ and $Y = X_j$: The estimate $|(E_{A,\widetilde{A},k})_{ij}| \prec 1/\sqrt{Nd}$ follows from Lemma \ref{lem:error_entrywise_a_tildea}; the estimate $|(E_{\widetilde{A},\widetilde{B}^{\textup{full}},k})_{ij}| \prec 1/\sqrt{Nd}$ for follows from Lemma \ref{lem:error_entrywise_tildea_tildeb}; the estimate $|(E_{\widetilde{B}^{\textup{full}},B^{\textup{full}},k})_{ij}| \prec 1/\sqrt{Nd}$ follows from Lemma \ref{lem:error_entrywise_tildeb_b}. 
\end{proof}

\begin{proof}[Proof of Proposition \ref{prop:error_bf_b}]
In the decomposition $B^{\textup{full}} = \sum_{k=0}^L c_k (B^{\textup{full}})_k$, it suffices to show that $(B^{\textup{full}})_k$ admits a low-rank-plus-small-Frobenius-norm decomposition for each $k = 0, \ldots, \lceil \ell \rceil - 1$. Fix such $k$. The decomposition merely adds in the ``missing diagonal'': Dropping $k$ from the notation, $E_{\textup{lr}} = E_{\textup{lr},k}$ and $E_{\textup{Frob}} = E_{\textup{Frob},k}$ are defined entrywise by
\begin{align*}
    (E_{\textup{lr}})_{ij} &= \frac{1}{\sqrt{N}d^{k/2}\sqrt{k!}}  \sum_{a_1,\ldots,a_k=1}^{d,\ast} X_{a_1i} \ldots X_{a_ki} X_{a_1j} \ldots X_{a_kj} \\
    (E_{\textup{Frob}})_{ij} &= \frac{\delta_{ij}}{\sqrt{N}d^{k/2}\sqrt{k!}}  \sum_{a_1,\ldots,a_k=1}^{d,\ast} X_{a_1i}^2 \ldots X_{a_ki}^2.
\end{align*}
The matrix $E_{\textup{lr}}$ is a sum of $d^k$ rank-one matrices of the form $M_{ij} = X_{a_1i} \ldots X_{a_ki} X_{a_1j} \ldots X_{a_kj}$ (actually slightly fewer than $d^k$, since the sum only counts $\{a_1, \ldots, a_k\}$ distinct), hence has rank at most $d^k \leq d^{\lceil \ell \rceil -1}$. 

The matrix $E_{\textup{Frob}}$ is diagonal, and its entries are bounded above by
\[
    \frac{1}{\sqrt{N}d^{k/2}} \sum_{a_1,\ldots,a_k=1}^d X_{a_1i}^2 \ldots X_{a_ki}^2 = \frac{1}{\sqrt{N}d^{k/2}} \|X\|^{2k} \prec \sqrt{\frac{d^k}{N}} = \OO\left(d^{(\lceil \ell \rceil - 1 - \ell)/2}\right).
\]
Hence $\|E_{\textup{Frob}}\|_{\textup{Frob}} \prec d^{(\lceil \ell \rceil - 1)/2}$, completing the proof.
\end{proof}


\appendix 


\section{General nonlinearities by approximation: Proof of Theorem \ref{thm:main_general}}
\label{appx:general_nonlinearities}

In this appendix, we prove Theorem \ref{thm:main_general} about general nonlinearities, via approximation by polynomials. The structure mimics the proof of \cite[Theorem 2]{LuYau}. For the whole section, $\mu$ will be a centered probability measure on $\R$ with unit variance and all finite moments, and $X_1, X_2 \in \R^d$ will be i.i.d. random vectors each of whose entries is an i.i.d. sample from $\mu$.

\begin{proof}[Proof of Theorem \ref{thm:main_general}]
Consider the inner product on functions with respect to Gaussian weight, 
\[
    \ip{f,g} \defeq \E_{Z \sim \mc{N}(0,1)}[f(Z)g(Z)],
\]
and corresponding norm $\|f\|^2 = \ip{f,f}$. If $f$ satisfies Assumption \ref{assn:poly_approximable}, it is easy to show that $\|f\|^2 = \sigma^2$.

Fix some $\epsilon$. The theorem involves some $z$ in the complex upper half plane; recall that we write $\eta > 0$ for its imaginary part. Since $\sum_k c_k^2$ converges, there exists some integer $L \geq \ell+1$ such that
\begin{equation}
\label{eqn:error_over_24}
    \sigma^2 - (\eta^4\epsilon^2/64) \leq \sum_{k=0}^{L-1} c_k^2 \leq \sigma^2.
\end{equation}
For this $L$, we define the approximating polynomial
\[
    f_{\textup{app}}(x) \defeq \sum_{k=0}^{L-1} c_k h_k(x) + \widehat{c_L} h_L(x)
\]
with the adjustment $\widehat{c_L} \defeq (\sigma^2 - \sum_{k=0}^{L-1} c_k^2)^{1/2}$, which is made so that $\gamma_c = \hat{\gamma_c}$, where the former is defined by \eqref{eqn:def_gammas} with respect to $f_{\textup{app}}$, and the latter is defined by \eqref{eqn:def_gammas_approximable} with respect to $f$. Notice that $f_{\textup{app}}$ always satisfies Assumption \ref{assn:poly_approximable}.

We also define the error-like function
\[
    e_{f,L}(x) = f(x) - \sum_{k=0}^L c_k h_k(x).
\]
Since the Hermite polynomials are orthogonal, we have
\[
    \|f-f_{\textup{app}}\|^2 = \|(c_L - \widehat{c_L})h_L(x) + e_{f,L}\|^2 = (c_L-\widehat{c_L})^2 + (\sigma^2 - \sum_{k=0}^L c_k^2) \leq 2c_L^2 + 2\widehat{c_L}^2 + \sigma^2 - \sum_{k=0}^L c_k^2 = c_L^2 + 2\widehat{c_L}^2 + \sigma^2 - \sum_{k=0}^{L-1} c_k^2.
\]
By \eqref{eqn:error_over_24} we have $\widehat{c_L}^2 \leq \eta^4\epsilon^2/64$ and $c_L^2 \leq \eta^4\epsilon^2/64$; thus we have
\[
    \|f-\widehat{f}\|^2 \leq \eta^4\epsilon^2/16.
\]
Now let $A_{\textup{app}}$ be the unnormalized matrix \eqref{eqn:intro_model_def} but for $f_{\textup{app}}$, with corresponding Stieltjes transform $s_{A_{\textup{app}}}$. Similarly, let $\widetilde{A}_{\textup{app}}$ be the normalized matrix \eqref{eqn:intro_model_def_normalized} but for $f_{\textup{app}}$, with corresponding Stieltjes transform $s_{\widetilde{A}_{\textup{app}}}$. On the one hand, from Theorem \ref{thm:main_polynomial} we have
\begin{align*}
    \abs{s_{A_{\textup{app}}}(z) - \mf{m}(z)} &\prec \frac{1}{\sqrt{d}}, \\
    \abs{s_{\widetilde{A}_{\textup{app}}}(z) - \mf{m}(z)} &\prec \frac{1}{\sqrt{d}}.
\end{align*}
On the other hand, from Lemma \ref{lem:lem_19_analogue} below, we have 
\begin{align*}
    \max\left( \abs{s_A(z) - s_{A_{\textup{app}}}(z)}, \abs{s_{\widetilde{A}}(z) - s_{\widetilde{A}_{\textup{app}}}(z)} \right) &\leq \frac{1}{\eta^2} \|f - f_{\textup{app}}\| + \oo\left(\frac{1}{\eta^2}\right) + \OO_\prec\left(\frac{1}{\eta\sqrt{N}} \right) \\
    &\leq \frac{\epsilon}{4} + \oo\left(\frac{1}{\eta^2}\right) + \OO_{\prec} \left(\frac{1}{\eta\sqrt{N}} \right) \leq \frac{\epsilon}{2} + \OO_{\prec} \left(\frac{1}{\eta\sqrt{N}} \right),
\end{align*}
where the last inequality holds for $d$ large enough depending on $\epsilon$ and $\eta$. Thus
\[
    \abs{s_A(z) - \mf{m}(z)} \leq \frac{\epsilon}{2} + \OO_\prec \left(\frac{1}{d^{\ell/2}}\right),
\]
meaning that for any $\epsilon$ and $D$ we have
\[
    \P(\abs{s_A(z) - \mf{m}(z)} > \epsilon) \leq C_{\epsilon,D} d^{-D}
\]
for $d$ large enough. By fixing $D$ and applying the Borel-Cantelli lemma, this suffices to show the almost-sure convergence of $s_A(z)$ to $\mf{m}(z)$.
\end{proof}

\begin{lemma}
\label{lem:lem_19_analogue}
Fix two functions $f, f_{\textup{app}} : \R \to \R$ that each satisfy Assumption \ref{assn:poly_approximable}, and a centered probability measure $\mu$ on $\R$ with unit variance and all finite moments. Write $A$ for the matrix \eqref{eqn:intro_model_def}, constructed with $f$, where the i.i.d. vectors $(X_i)_{i=1}^N$ have i.i.d. entries drawn from $\mu$, with corresponding Stieltjes transform $s_A(z)$. Write $A_{\textup{app}}$ for the analogue with $f$ replaced by $f_{\textup{app}}$, with Stieltjes transform $s_{A_\textup{app}}(z)$. Write also $\widetilde{A}$ for the corresponding normalized model \eqref{eqn:intro_model_def_normalized}, constructed with $f$, and $\widetilde{A}_{\textup{app}}$ for the analogue constructed with $f_{\textup{app}}$. Then 
\begin{align}
    \abs{s_A(z) - s_{A_{\textup{app}}}(z)} \leq \frac{1}{\eta^2}\left( \E_{Z \sim \mc{N}(0,1)}[(f(Z) - f_{\textup{app}}(Z))^2] \right)^{1/2} + \oo \left( \frac{1}{\eta^2} \right) + \OO_{\prec} \left(\frac{1}{\eta\sqrt{N}}\right), \label{eqn:lem_19_analogue} \\
    \abs{s_{\widetilde{A}}(z) - s_{\widetilde{A}_{\textup{app}}}(z)} \leq \frac{1}{\eta^2}\left( \E_{Z \sim \mc{N}(0,1)}[(f(Z) - f_{\textup{app}}(Z))^2] \right)^{1/2} + \oo \left( \frac{1}{\eta^2} \right) + \OO_{\prec} \left(\frac{1}{\eta\sqrt{N}}\right), \label{eqn:lem_19_analogue_normalized}
\end{align}
\end{lemma}
\begin{proof}
Lemma 19 of \cite{LuYau} shows
\[
    \abs{s_A(z) - s_{A_{\textup{app}}}(z)} \leq \frac{1}{\eta^2}\left( \E_{X_1,X_2}\left[\left(f\left(\frac{\ip{X_1,X_2}}{\sqrt{d}}\right) - f_{\textup{app}}\left(\frac{\ip{X_1,X_2}}{\sqrt{d}}\right)\right)^2\right] \right)^{1/2} + \OO_{\prec} \left(\frac{1}{\eta\sqrt{N}}\right).
\]
Applying Lemma \ref{lem:second_moment_growth_condition} below to the function $g(x) = f(x) - f_{\textup{app}}(x)$ yields \eqref{eqn:lem_19_analogue}. The proof of \eqref{eqn:lem_19_analogue_normalized} is a little more involved for technical reasons, since
\[
    \widetilde{A}_{ij} = \frac{\delta_{i \neq j}}{\sqrt{N}} f\left( \frac{\ip{X_i,X_j}}{\sqrt{d}}\frac{\sqrt{d}}{\|X_i\|} \frac{\sqrt{d}}{\|X_j\|} \right) \mathds{1}\{\|X_i\| \neq 0 \neq \|X_j\|\}
\]
but as written Lemma 19 of \cite{LuYau} only allows a direct comparison of matrices of the form
\[
    (\widetilde{A}_f)_{ij} = \frac{\delta_{i \neq j}}{\sqrt{N}} f\left( \frac{\ip{X_i,X_j}}{\sqrt{d}}\frac{\sqrt{d}}{\|X_i\|} \frac{\sqrt{d}}{\|X_j\|} \mathds{1}\{\|X_i\| \neq 0 \neq \|X_j\|\} \right)
\]
(by taking what they call $\mathbf{x}_i$ to be what we call $X_i/\|X_i\|$ if well-defined, or the zero vector otherwise), for various choices of $f$. That is, on the exponentially unlikely occasions where $\|X_i\| = 0$, the corresponding row and column are set to zero in $\widetilde{A}$ but $f(0)/\sqrt{N}$ in $\widetilde{A}_f$. By applying Lemma 19 of \cite{LuYau} and Lemma \ref{lem:second_moment_growth_condition} as in the unnormalized case, we obtain
\[
    \abs{s_{\widetilde{A}_f}(z) - s_{\widetilde{A}_{f_{\textup{app}}}}(z)} \leq \frac{1}{\eta^2}\left( \E_{Z \sim \mc{N}(0,1)}[(f(Z) - f_{\textup{app}}(Z))^2] \right)^{1/2} + \oo\left(\frac{1}{\eta^2}\right) + \OO_{\prec} \left(\frac{1}{\eta\sqrt{N}}\right).
\]
It remains to compare $s_{\widetilde{A}}$ and $s_{\widetilde{A}_f}$, as well as $s_{\widetilde{A}_{f_{\textup{app}}}}$ and $s_{\widetilde{A}_{\textup{app}}}$, which we will do with the rank estimate \eqref{eqn:stieltjes_rank_estimate}. Bounding the rank of a matrix by its number of nonzero rows, this gives
\begin{align*}
    \max\left( \abs{s_{\widetilde{A}}(z) - s_{\widetilde{A}_f}(z)}, \abs{s_{\widetilde{A}_{f_{\textup{app}}}}(z) - s_{\widetilde{A}_{\textup{app}}}(z)} \right) &\leq \frac{C\max(\rank(\widetilde{A} - \widetilde{A}_f),\rank(\widetilde{A}_{f_{\textup{app}}} - \widetilde{A}_{\textup{app}}))}{N\eta} \\
    &\leq \frac{C\#\{i : \|X_i\| = 0\}}{N\eta} \prec \frac{1}{N\eta},
\end{align*}
where the last estimate holds since $\P(\|X_i\| = 0)$ is exponentially small, so that $\#\{i : \|X_i\| = 0\} \prec 1$. Absorbing this into $\OO_\prec((\eta\sqrt{N})^{-1})$ completes the proof.
\end{proof}

\begin{lemma}
\label{lem:second_moment_growth_condition}
Fix $g : \R \to \R$ that satisfies Assumption \ref{assn:poly_approximable}. If $\mu$ has all finite moments, then
\begin{align}
    &\E_{X_1,X_2}\left[ g^2 \left( \frac{\ip{X_1,X_2}}{\sqrt{d}} \right) \right] \overset{d \to \infty}{\to} \E_Z[g^2(Z)], \label{eqn:second_moment_growth_condition_unnormalized} \\
    &\E_{X_1,X_2}\left[ g^2 \left( \frac{\ip{X_1,X_2}}{\sqrt{d}} \frac{\sqrt{d}}{\|X_1\|} \frac{\sqrt{d}}{\|X_2\|} \mathds{1}\{\|X_1\| \neq 0 \neq \|X_2\|\} \right) \right] \overset{d \to \infty}{\to} \E_Z[g^2(Z)]. \label{eqn:second_moment_growth_condition_normalized}
\end{align}
\end{lemma}
\begin{proof}
Write $\mu_d$ for the law of $\frac{\ip{X_1,X_2}}{\sqrt{d}}$, and $\widetilde{\mu_d}$ for the law of $\frac{\ip{X_1,X_2}}{\sqrt{d}} \frac{\sqrt{d}}{\|X_1\|} \frac{\sqrt{d}}{\|X_2\|} \mathds{1}\{\|X_1\| \neq 0 \neq \|X_2\|\}$ -- neither of which necessarily has a density with respect to Lebesgue measure -- as well as $\mu_G$ for standard Gaussian measure. By the usual central limit theorem, $\mu_G$ is the $d \to \infty$ weak limit of the measures $\mu_d$. It is straightforward to show that $\frac{d}{\|X_1\|}\mathds{1}\{\|X_1\| \neq 0\}$ and $\frac{d}{\|X_2\|}\mathds{1}\{\|X_2\| \neq 0\}$ each converge in probability to one, so that $\frac{\ip{X_1,X_2}}{\sqrt{d}} \frac{\sqrt{d}}{\|X_1\|} \frac{\sqrt{d}}{\|X_2\|} \mathds{1}\{\|X_1\| \neq 0 \neq \|X_2\|\}$ also converges to a Gaussian variable, meaning that $\mu_G$ is also the $d \to \infty$ weak limit of the measures $\widetilde{\mu_d}$.

Fix any constant $M \geq \max(1,\abs{\alpha_1}, \abs{\alpha_K})$. Write $g_M$ for the function which agrees with $g$ on $[-M,M]$, vanishes outside $[-(M+1),M+1]$, linearly interpolates on $[M,M+1]$ between $g(M)$ and $0$, and linearly interpolates on $[-(M+1),-M]$ between $0$ and $g(-M)$, and define $e_M$ by
\[
    g^2(x) = g_M^2(x) + e_M(x).
\]
The result will follow if we can show
\begin{align}
    \lim_{d \to \infty} \int_\R g_M^2(x)(\mu_d - \mu_G)(\diff x) &= 0, \label{eqn:g2_convergence_compact} \\
    \lim_{d \to \infty} \int_\R g_M^2(x)(\widetilde{\mu_d} - \mu_G)(\diff x) &= 0, \label{eqn:g2_convergence_compact_normalized}
\end{align}
for any fixed $M$, as well as
\begin{align}
    \lim_{M \to \infty} \abs{\int_\R e_M(x) \mu_G(\diff x)} &= 0, \label{eqn:em_bound_mug} \\
    \lim_{M \to \infty} \limsup_{d \to \infty} \abs{\int_\R e_M(x) \mu_d(\diff x)} &= 0, \label{eqn:em_bound_mud} \\
    \lim_{M \to \infty} \limsup_{d \to \infty} \abs{\int_\R e_M(x) \widetilde{\mu_d}(\diff x)} &= 0. \label{eqn:em_bound_tildemud}
\end{align}

We start with the proof of \eqref{eqn:g2_convergence_compact} and \eqref{eqn:g2_convergence_compact_normalized}. Notice that, since $g(x)$ is the difference of the nonlinear function $f(x)$ and a finite degree polynomial, and $f(x)$ is piecewise continuous with a polynomial growth rate, $g^2(x)$ is also piecewise continous with a polynomial growth rate, and $g_M^2(x)$ is piecewise continuous (with ``no growth rate'' since it vanishes outside $[-(M+1),M+1]$. The possible discontinuities of $g^2(x)$ and $g_M^2(x)$ occur at $\{\alpha_1, \ldots, \alpha_K\}$. 

Write $h(x) = g_M^2(x)$. For any $\epsilon \in (0, \min_{1 \le i < K} \{\alpha_{i+1} - \alpha_{i}\}/2)$, construct a ``smoothed'' version $h_\epsilon(x)$ such that $h_\epsilon(x) = h(x)$ for $x \in \R \setminus \bigcup_{1 \le i \le K} [\alpha_i-\epsilon, \alpha_i + \epsilon]$; on the interval $[\alpha_i - \epsilon, \alpha_i +\epsilon]$, for $i \in \{1, 2, \ldots, K\}$, the function $h_\epsilon(x)$ is a linear interpolation between $h(\alpha_i-\epsilon)$ and $h(\alpha_{i+1} + \epsilon)$. We have
\[
\begin{aligned}
\abs{\int h(x) (\mu_d - \mu_G)(dx)} \le \abs{\int h_\epsilon(x) (\mu_d - \mu_G)(dx)} + \int \abs{h(x) - h_\epsilon(x)} \mu_d(dx) + \int \abs{h(x) - h_\epsilon(x)} \mu_
G(dx).
\end{aligned}
\]
Since $h_\epsilon(x)$ is a continuous function with compact support, the first term on the right-hand sides converges to $0$ as $d \to \infty$ due to the CLT. To control the second and third terms on the right-hand side, we note that 
\[
\abs{h(x) - h_\epsilon(x)} \le C_M \sum_{1 \le i \le K} \mathbf{1}_{[\alpha_i-\epsilon, \alpha_i + \epsilon]}(x),
\]
where $C_M$ is some finite constant that can depend on $M$. Since $\mu_G$ has a bounded density function,
\[
\int \abs{h(x) - h_\epsilon(x)} \mu_G(dx) \le C_M \sum_{1\le i \le K} \int \mathbf{1}_{[\alpha_i-\epsilon, \alpha_i + \epsilon]}(x) \mu_G(dx) \le 2K C_M \epsilon. 
\]
Similarly, 
\[
\begin{aligned}
\int \abs{h(x) - h_\epsilon(x)} \mu_d(dx) &\le C_M \sum_{1\le i \le K} \int \mathbf{1}_{[\alpha_i-\epsilon, \alpha_i + \epsilon]}(x) \mu_d(dx)\\
&= C_M \sum_{1\le i \le K} \int \mathbf{1}_{[\alpha_i-\epsilon, \alpha_i + \epsilon]}(x) \mu_G(dx) + \oo_d(1)\\
&\le 2K C_M \epsilon + \oo_d(1),
\end{aligned}
\]
where in the second step we have applied the CLT, with the approximation error captured by $\oo_d(1)$. Since $\epsilon$ can be chosen to be arbitrarily small, we have established \eqref{eqn:g2_convergence_compact}, about the unnormalized model. The estimate for the normalized model, \eqref{eqn:g2_convergence_compact_normalized}, is the same, since it relies only on weak convergence of $\mu_d$ or $\widetilde{\mu_d}$ to $\mu_G$.

Now we study the $e_M$ terms \eqref{eqn:em_bound_mug}, \eqref{eqn:em_bound_mud}, and \eqref{eqn:em_bound_tildemud}. First we observe that, by construction,
\[
    \abs{e_M(x)} \leq C\abs{x}^C \mathds{1}_{\abs{x} \geq M}
\]
for some $C$, which can now change from line to line. Indeed, $e_M(x)$ agrees with $g^2(x)$ on $\{\abs{x} \geq M+1\}$, where this growth rate is by assumption, and is constructed from $g$ and a linear interpolation on $\{\abs{x} \in [M,M+1]\}$. From this, standard tail bounds give
\[
    \abs{\int_\R e_M(x) \mu_G(\diff x)} \leq C\int_M^\infty \abs{x}^C \mu_G(\diff x) = \OO(e^{-M^2/4}),
\]
which proves \eqref{eqn:em_bound_mug}. Next, we combine the standard estimate 
\begin{equation}
\label{eqn:inner_product_moments_unnormalized}
    \sup_d \E\left[\left(\frac{\ip{X_1,X_2}}{\sqrt{d}}\right)^{2p}\right] \leq C_p
\end{equation}
with Cauchy--Schwartz and Markov's inequality to find
\begin{equation}
\label{eqn:g2_convergence_mu_d}
\begin{split}
    \int_\R e_M(x) \mu_d(\diff x) &\leq C \int_{\abs{x} > M} \abs{x}^{C} \mu_d(\diff x) \leq 2\alpha_1^2 \left( \E\left[\abs{\ip{X_1,X_2}/\sqrt{d}}^{2C}\right] \P\left(\abs{\ip{X_1,X_2}/\sqrt{d}} > M\right) \right)^{1/2} \\
    &\leq \frac{2\alpha_1^2}{M} \left( \E\left[\abs{\ip{X_1,X_2}/\sqrt{d}}^{2C}\right] \E\left[(\ip{X_1,X_2}/\sqrt{d})^2\right] \right)^{1/2} =  \OO\left(\frac{1}{M}\right),
\end{split}
\end{equation}
which verifies \eqref{eqn:em_bound_mud}. Finally, we need to bound
\[
    \int_\R e_M(x) \widetilde{\mu_d}(\diff x) \leq C \int_{\abs{x} > M} \abs{x}^{C} \widetilde{\mu_d}(\diff x) = C \E\left[\abs{\frac{\ip{X_1,X_2}}{\sqrt{d}} \frac{\sqrt{d}}{\|X_1\|} \frac{\sqrt{d}}{\|X_2\|}}^{C} \mathds{1}_{\abs{\frac{\ip{X_1,X_2}}{\sqrt{d}} \frac{\sqrt{d}}{\|X_1\|} \frac{\sqrt{d}}{\|X_2\|} } \geq M} \right].
\]
We will split this estimate into two parts, namely on and off the good event
\[
    \mc{E}_{\textup{good}} = \left\{\|X_1\| \geq \frac{\sqrt{d}}{2} \text{ and } \|X_2\| \geq \frac{\sqrt{d}}{2}\right\}.
\]
On this event,
\begin{align*}
    \E\left[\abs{\frac{\ip{X_1,X_2}}{\sqrt{d}} \frac{\sqrt{d}}{\|X_1\|} \frac{\sqrt{d}}{\|X_2\|}}^{C} \mathds{1}_{\abs{\frac{\ip{X_1,X_2}}{\sqrt{d}} \frac{\sqrt{d}}{\|X_1\|} \frac{\sqrt{d}}{\|X_2\|}} \geq M}  \mathds{1}_{\mc{E}_{\textup{good}}}\right] \leq C \E\left[ \abs{\frac{\ip{X_1,X_2}}{\sqrt{d}}}^{C} \mathds{1}_{\abs{\frac{\ip{X_1,X_2}}{\sqrt{d}}} \geq \frac{M}{4}} \right],
\end{align*}
which is $\OO(1/M)$ by the same arguments as in \eqref{eqn:g2_convergence_mu_d}. On its complement, we apply the deterministic estimate $\frac{\abs{\ip{X_1,X_2}}}{\|X_1\|\|X_2\|} \leq 1$ to find
\begin{align*}
    \E\left[\abs{\frac{\ip{X_1,X_2}}{\sqrt{d}} \frac{\sqrt{d}}{\|X_1\|} \frac{\sqrt{d}}{\|X_2\|}}^{C} \mathds{1}_{\abs{\frac{\ip{X_1,X_2}}{\sqrt{d}} \frac{\sqrt{d}}{\|X_1\|} \frac{\sqrt{d}}{\|X_2\|}} \geq M}  \mathds{1}_{\mc{E}_{\textup{good}}^c}\right] \leq d^{C/2} \P(\mc{E}_{\textup{good}}^c).
\end{align*}
The estimate $\abs{\|X\| - \sqrt{d}} \prec 1$ from \eqref{eqn:error_clt_stochastic_domination_not_squared} gives $\P(\mc{E}_{\textup{good}}^c) \leq C_Dd^{-D}$ for any fixed $D > 0$; if we take $D > C/2$, we find that this is $\oo(1)$ as $d \to \infty$. Thus $\limsup_{d \to \infty} \int_\R e_M(x) \mu_d(\diff x) = \OO(1/M)$, which verifies \eqref{eqn:em_bound_tildemud} and finishes the proof.
\end{proof}

\addcontentsline{toc}{section}{References}
\bibliographystyle{alpha-abbrvsort}
\bibliography{kernelbib-arxiv-v1}

\end{document}